\documentclass[11pt,reqno]{amsart}
\usepackage{amsmath}
\usepackage{amsfonts}
\usepackage{amssymb, amsthm, epsfig}
\usepackage{mathrsfs}
\usepackage{hyperref}
 \usepackage{yhmath}
\usepackage{cite}
\usepackage[utf8]{inputenc}
\usepackage[T1]{fontenc}
\usepackage{geometry}
\usepackage{graphicx}
\usepackage{verbatim}
\usepackage{color}

\theoremstyle{plain}
\newtheorem{thm}{Theorem}[section]

\newtheorem{lem}[thm]{Lemma}
\newtheorem{prop}[thm]{Proposition}

\theoremstyle{definition}
\newtheorem{defn}{Definition}[section]
\theoremstyle{remark}
\newtheorem{rem}{Remark}[section]

\numberwithin{equation}{subsection}

\allowdisplaybreaks

\DeclareMathOperator*{\dist}{dist}

\DeclareSymbolFont{lettersA}{U}{pxmia}{m}{it}
\DeclareMathSymbol{\piup}{\mathord}{lettersA}{"19}

\makeatletter

\newcommand{\Rmnum}[1]{\expandafter\@slowromancap\romannumeral#1@}
\makeatother

\begin{document}
\title[Interaction of fan-shock-fan composite waves]
{On the expansion of a wedge of van der Waals gas into vacuum III: interaction of fan-shock-fan composite waves}

\author{Geng Lai}

\address{Department of Mathematics, Shanghai University,
Shanghai, 200444, P.R. China;  Newtouch Center for Mathematics, Shanghai University, Shanghai, 20044, P.R. China.}
\email{\tt laigeng@shu.edu.cn}

\keywords{Expansion in vacuum; van der Waals gas; fan-shock-fan composite wave; wave interaction;
hodograph transformation; characteristic decomposition. }
\date{\today}

\begin{abstract}
This paper studies the expansion into vacuum of a wedge of gas at rest.
This problem catches several important classes of wave interactions in the context of 2D Riemann problems.
When the gas at rest is a nonideal gas (e.g., van der Waals gas), the
gas away from the sharp corner of the wedge may expand into the vacuum as two symmetrical planar rarefaction fan
waves, shock-fan composite waves, or fan-shock-fan composite waves. Then the expansion in vacuum problem can be reduced to the interactions of these elementary waves. Global existences of classical solutions to the interaction of the fan
waves and the interaction of  the  shock-fan composite waves were obtained by the author in \cite{Lai1,Lai2}.
In the present paper we study the third case: interaction of fan-shock-fan composite waves.
In contrast to the first two cases, the third case involves shock waves in the interaction region and is actually a shock free boundary problem.
Differing from the transonic shock free boundary problems arising in 2D Riemann problems for ideal gases, the type of the shocks for this shock free boundary problem is also a priori unknown. This results in the fact that the formulation of the boundary conditions on the shocks is also a priori unknown.
By calculating the curvatures of the shocks  and  using the Liu's extended entropy condition,
we prove that the shocks in the interaction region must be post-sonic (in the sense of the flow velocity relative to the shock front). We also prove that the shocks
are envelopes of one out of the two families of wave characteristics of the flow behind them, and not characteristics.
This leads to the fact that the flow behind the shocks is not $C^1$ smooth up to the shock boundaries.
By virtue of the hodograph transformation method and the characteristic decomposition method, we construct a global-in-time piecewise smooth solution to the expansion in vacuum problem for the third case.  
 The techniques and ideas developed here may be applied for solving some other 2D Riemann problems of the compressible Euler equations with a non-convex equation of state.
\end{abstract}

\maketitle
\tableofcontents

\section{Introduction}
\subsection{Expansion into vacuum of a wedge of gas}
The two-dimensional (2D) isentropic Euler system has the following form:
\begin{equation}
\left\{
            \begin{array}{ll}
            \rho_t+(\rho u)_x+(\rho v)_y=0, \\[4pt]
            (\rho u)_t+(\rho u^{2}+p)_x+(\rho uv)_y=0,  \\[4pt]
            (\rho v)_t+(\rho uv)_x+(\rho v^{2}+p)_y=0,
            \end{array}
       \right.                       \label{Euler}
\end{equation}
where
$(u, v)$ is the velocity, $\rho$ is the density, and
 $p=p(\rho)$ is the pressure.

\begin{figure}[htbp]
\begin{center}
\includegraphics[scale=0.53]{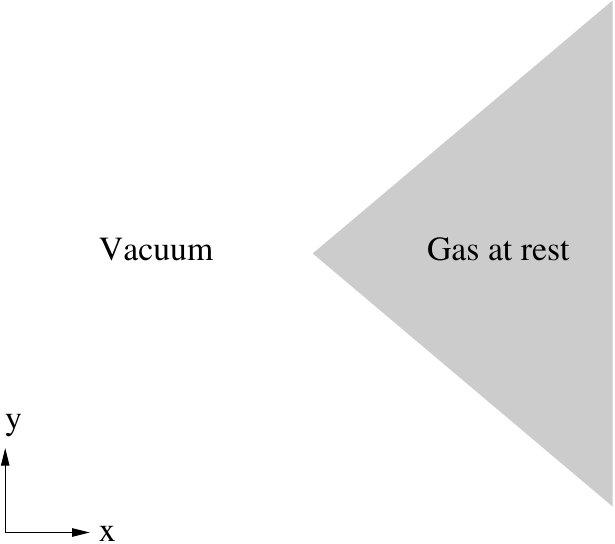}
\caption{\footnotesize Initial data.}
\label{Fig1}
\end{center}
\end{figure}

In this paper, we
study the problem of
the expansion into vacuum of a wedge of gas at rest. This problem can be formulated as a Cauchy problem for
 (\ref{Euler}) with the initial data
\begin{equation}\label{initial}
(u,v,\rho)(x,y,0)=\left\{
              \begin{array}{ll}
                (0,0,\rho_0), & \hbox{$x>0$, $|y|< x\tan\theta$;} \\[2pt]
               \mbox{vacuum}, & \hbox{otherwise,}
              \end{array}
            \right.
\end{equation}
where $\theta\in(0, \pi/2)$ represents the half-angle of the wedge and $\rho_0$ is a positive constant; see Figure \ref{Fig1}.

The gas expansion in vacuum problem (\ref{Euler}, \ref{initial}) has been an interesting problem for a long time.
It catches several important classes of wave interactions in the context of 2D Riemann problems. It has also been interpreted hydraulically as the collapse of a wedge-shaped dam containing water initially with a uniform velocity; see Levine \cite{Levine}.
When the gas at rest is a polytropic ideal gas, e.g.,
$p=\rho^{\gamma}$ where $\gamma$ is an
adiabatic constant, the problem can be reduced to the interaction of two symmetric planar rarefaction fan waves in the self-similar plane $(\xi, \eta)=(x/t, y/t)$.
As early as 1963, Suchkow \cite{Suc} studied the problem and obtained an explicit classical solution for $\theta=\arctan\sqrt{\frac{3-\gamma}{\gamma+1}}$.
The problem for more general $\theta\in (0, \pi/2)$ had been open for a long time.
In 2000, Dai and Zhang \cite{Dai} studied the problem for the pressure gradient system which can be seen as a simplified model for the Euler system. They obtained the global existence of classical solutions for general $\theta$ by the method of characteristics. Subsequently, Li \cite{Li1} studied the problem for the Euler system. By virtue of the hodograph transformation, he obtained a solution of the problem in the hodograph plane. 
In 2006, Li, Zhang and Zheng \cite{Li-Zhang-Zheng} introduced a characteristic decomposition method to study rarefaction waves of the 2D Euler system.
Subsequently, Li and Zhang \cite{Li3} applied the characteristic decomposition method and the phase plane analysis to show that the hodograph transformation of the expansion in vacuum problem in \cite{Li1} is one-to-one, and consequently established the global existence of classical solutions for all $\theta\in(0, \pi/2)$.
For further results on the expansion in vacuum problem for polytropic ideal gases, we refer the reader to \cite{Chen1,Hu,Li-Yang-Zheng,Zhao}.
The rarefaction wave interaction solution constructed in \cite{Li3} was also used as a building block of the solution to a 2D Riemann problem with four rarefaction waves; see Li and Zheng \cite{Li6}.

For a polytropic ideal gas, isentropes are always concave-up in the $(\tau, p)$-plane.
So, a natural and interesting question is to determine whether the above result
can be extended to more general gases, for example, gases with nonconvex equations of state.
When the gas at rest is a nonconvex gas, there will be some other interesting and important wave interactions in the expansion in vacuum problem.
The possibility that isentropes are concave-down within a certain region in the $(\tau, p)$-plane was first explored independently by Bethe \cite{Bethe},  Zel'dovich \cite{ZE}, and  Thompson and Lambrakis \cite{Thompson2} for van der Waals gases.
Fluids with nonconvex equation of state may significantly differ from the polytropic ideal gases. For example, physically admissible rarefaction shocks may occur in gases with nonconvex equation of state; see e.g., \cite{BBKN,Cra,NSMGC,ZGC}.
In 1972, Wendroff \cite{Wen1,Wen2} studied the 1D Riemann problem for the Euler equations with nonconvex equation of state and introduced composite waves that are composed of shock and rarefaction waves.
Liu \cite{Liu1,Liu2} introduced an extended entropy condition and established the existence and uniqueness of the entropy solutions to the 1D Riemann problem for the Euler equations with nonconvex equation of state.
We also refer the reader to the survey paper by Menikoff and Plohr
 \cite{MP} and  the  monograph by LeFloch and Thanh \cite{LeFloch} for the 1D Riemann problem for more general gases.


\begin{figure}[htbp]
\begin{center}
\includegraphics[scale=0.55]{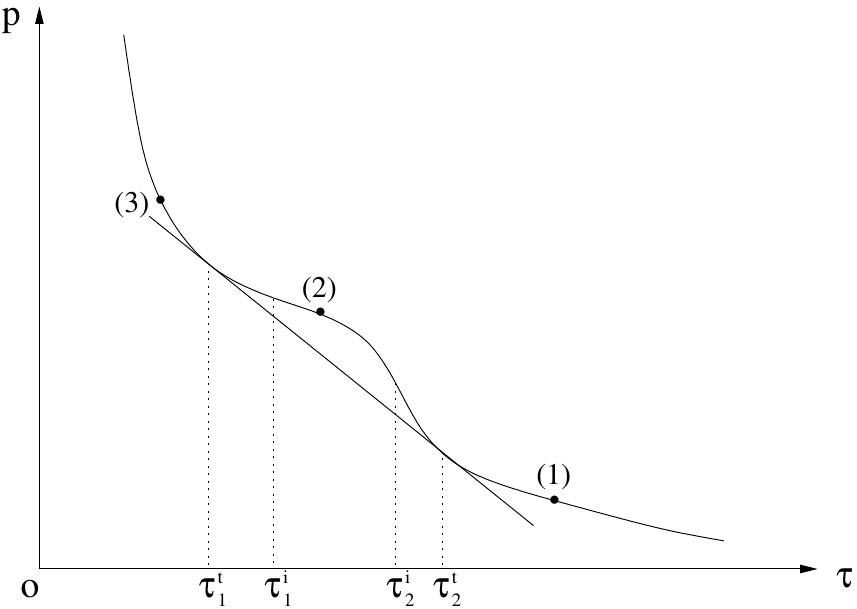}
\caption{\footnotesize An isentrope of a van der Waals gas.}
\label{Fig2}
\end{center}
\end{figure}

In this paper,
we consider
a polytropic van der Waals gas with the equation of state
\begin{equation}
p=\frac{\mathcal{S}}{(\tau-1)^{\gamma}}-\frac{1}{\tau^2}\quad (\tau>1),\label{van}
\end{equation}
where $\tau=1/\rho$ is the specific volume, $\mathcal{S}$ is a positive constant in corresponding with the entropy, and $\gamma$ is an
adiabatic constant between $1$ and $2$; see, e.g., Callen \cite{Cal}.
When $\mathcal{S}$ lies in some interval,
the equation of state (\ref{van}) has the following property.
\begin{description}
  \item[(P)] There exist $\tau_1^{i}$ and $\tau_2^{i}$ ($1<\tau_1^{i}<\tau_2^{i}$) such that
\begin{equation}\label{81601}
p'(\tau)<0~\mbox{for}~\tau>1;\quad
 p''(\tau)>0~ \mbox{for}~ \tau\in(1, \tau_1^i)\cup(\tau_2^i, +\infty);
 \quad p''(\tau)<0~ \mbox{for}~ \tau_1^i<\tau<\tau_2^i;
\end{equation}
\end{description}
see Figure \ref{Fig2}.
Under property ($\mathbf{P}$), there exist constants $\tau_1^{t}$ and $\tau_2^{t}$ ($1<\tau_1^{t}<\tau_1^i<\tau_2^i<{\tau}_2^{t}$) such that
$$
p'(\tau_1^{t})=p'(\tau_2^{t})=\frac{p(\tau_1^{t})-p(\tau_2^{t})}{\tau_1^{t}-\tau_2^{t}}.
$$

We
consider
the expansion in vacuum problem (\ref{Euler}, \ref{initial}) for the polytropic van der Waals gas. By the result of the 1D Riemann problem for the van der Waals gas (see \cite{Wen1}), we know that the problem can be divided into the following three cases.
\begin{description}
  \item[(1)] When $\tau_0\geq \tau_2^i$, the
gas away from the sharp corner of the wedge expands into the vacuum as two symmetrical planar rarefaction fan
waves; see Figure \ref{Fig3} (1).
  \item[(2)] When $\tau_1^t<\tau_0< \tau_2^i$,  the
gas away from the sharp corner of the wedge expands into the vacuum as two symmetrical planar shock-fan composite
waves; see Figure \ref{Fig3} (2).
  \item[(3)] When $1<\tau_0< \tau_1^t$,  the
gas away from the sharp corner of the wedge expands into the vacuum as two symmetrical planar fan-shock-fan composite
waves; see Figure \ref{Fig3} (3).
   \end{description}
Therefore, the problem can be reduced to the interaction of these elementary waves in the self-similar $(\xi, \eta)$-plane.
In \cite{Lai1} and \cite{Lai2},
the author constructed global shock-free solutions for the interaction of the fan waves and the interaction of the shock-fan composite waves, respectively, provided that $\frac{p'(\tau)}{\tau p''(\tau)}$ is monotonic and $\theta$ satisfies some restrictions.
Later, the restriction on the monotonicity  was removed by the author in \cite{Lai6}.
In the present paper we study the third case.
In contrast to the first two cases, the third case involves shock waves in the interaction region and is actually a shock free boundary problem.
For a nonconvex gas, physically admissible shocks may be transonic, post-sonic, pre-sonic, or double-sonic in the sense of the flow velocity relative to the shock front; see, e.g., \cite{VKG1,VKG2}.
Differing from the transonic shock free boundary problems arising in 2D Riemann problems for polytropic ideal gases (see \cite{BCF1,BCF2,Canic,CCHLQ,CX,CF1,CF2,CFX,CSX,Elling1,ELiu2,Serre1,Zheng2}), the type of the shocks
 to this shock free boundary problem is also a priori unknown. 
To the best of our knowledge, there are few results on this type of shock free boundary problems.

\begin{figure}[htbp]
\begin{center}
\includegraphics[scale=0.39]{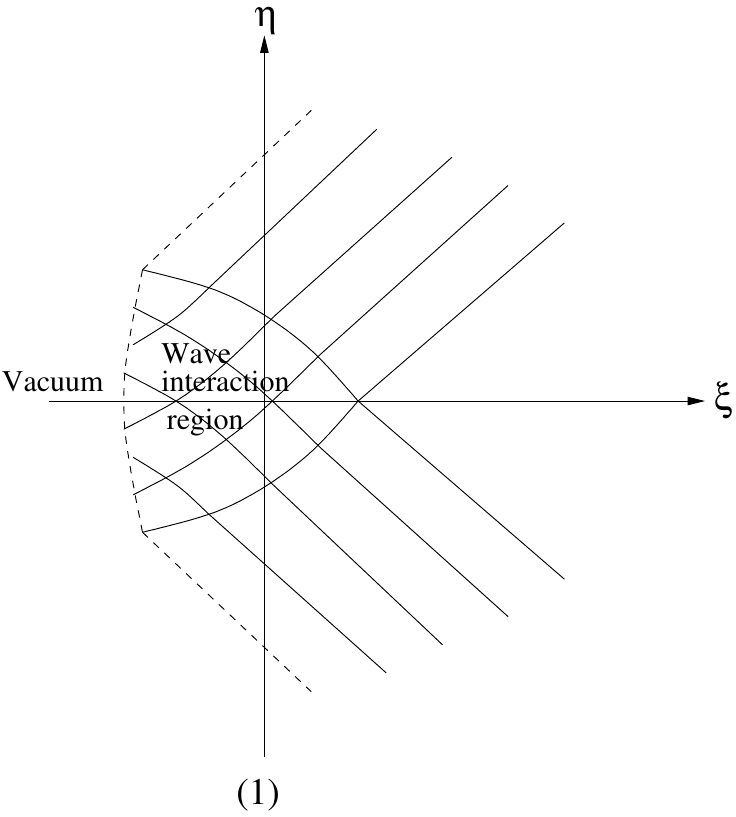}\quad\includegraphics[scale=0.39]{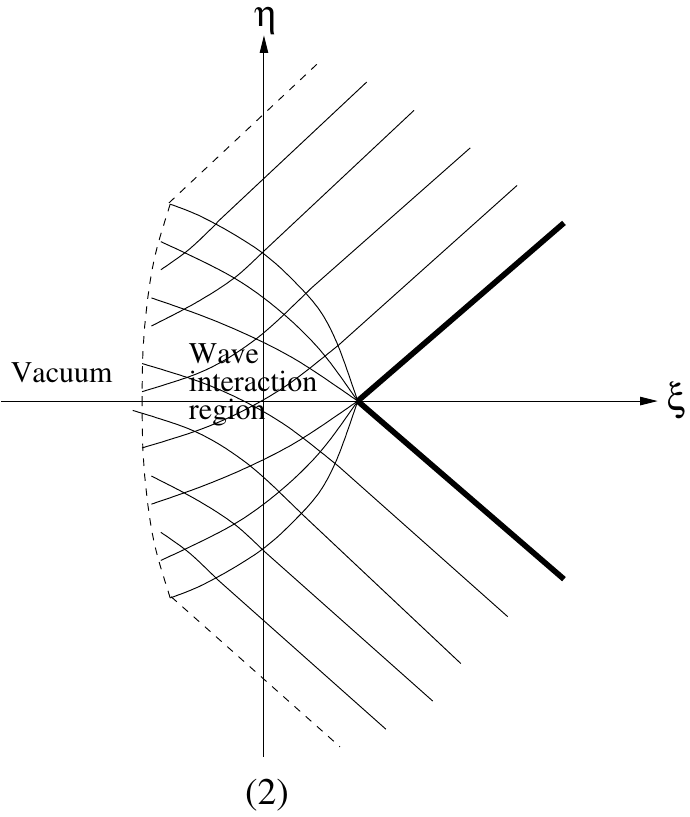}
\quad\includegraphics[scale=0.39]{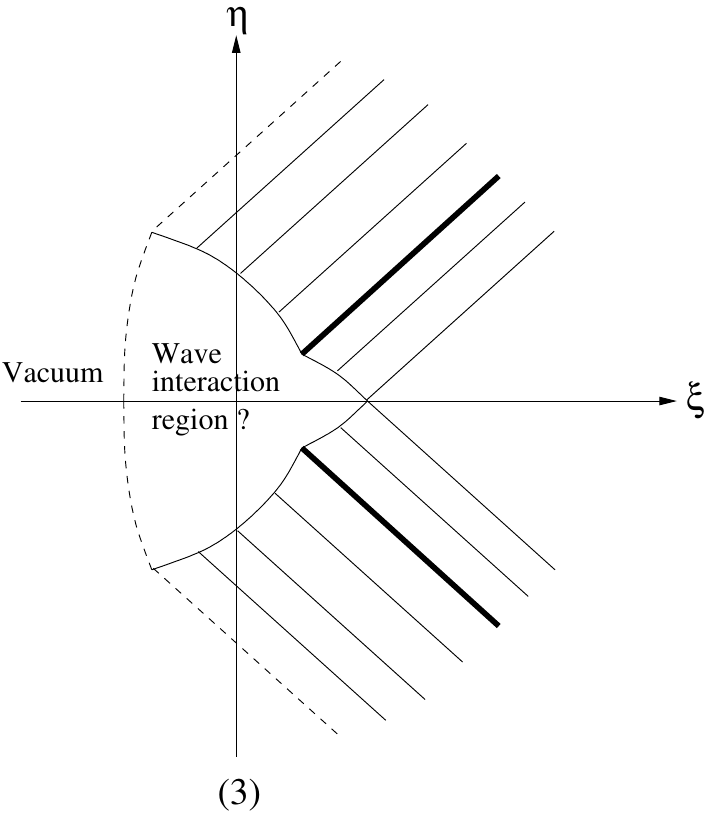}
\caption{\footnotesize Wave interactions in the self-similar $(\xi,\eta)$-plane: (1) interaction of fan waves; (2) interaction of shock-fan composite waves; (3) interaction of fan-shock-fan composite waves. Here, the solid lines represent characteristic lines; the thick lines represent shock waves; the dashed lines represent vacuum boundaries.}
\label{Fig3}
\end{center}
\end{figure}

\subsection{Systems for potential flow}
 We purpose to construct a global-in-time piecewise smooth solution to the expansion in vacuum problem for the third case.
As a first stage, we study the problem for the potential flow equation.
The two-dimensional (2D) compressible Euler equations for potential flow consist of the conservation law of mass and the Bernoulli law for the density and velocity potential $(\rho, \Phi)$:
\begin{equation}\label{Potential}
\left\{
            \begin{array}{ll}
            \rho_t+(\rho \Phi_x)_x+(\rho \Phi_y)_y=0, \\[4pt]
         \displaystyle \Phi_t+\frac{1}{2}\big(\Phi_x^2+\Phi_y^2\big)+h(\tau)=0;
            \end{array}
       \right.
\end{equation}
see Chen and Feldman \cite{CF2}.
Here, $\Phi$ is the velocity potential, i.e., $(\Phi_x, \Phi_y)=(u, v)$,  and
\begin{equation}\label{5801}
h'(\tau)=\tau p'(\tau).
\end{equation}

We look for self-similar solutions with the following form:
\begin{equation}
\rho=\rho(\xi, \eta)\quad \mbox{and}\quad  \Phi=t\phi(\xi, \eta)\quad \mbox{for}\quad(\xi, \eta)=\big(\frac{x}{t}, \frac{y}{t}\big).
\end{equation}
Then the pseudo-potential function $\varphi=\phi-\frac{1}{2}(\xi^2+\eta^2)$ satisfies the following self-similar Euler equations:
\begin{equation}\label{SE}
\left\{
  \begin{array}{ll}
   (\rho \varphi_\xi)_\xi+(\rho \varphi_\eta)_\eta+2\rho=0,\\[4pt]
  \displaystyle\frac{1}{2}(\varphi_\xi^2+\varphi_\eta^2)+h+\varphi =0.
  \end{array}
\right.
\end{equation}

Let $(U, V)=(u-\xi, v-\eta)$ be the pseudo-velocity. Then we have
$$
\varphi_\xi=U\quad \mbox{and} \quad \varphi_\eta=V.
$$
Hence, the 2D self-similar Euler equations for potential flow can be written as
\begin{equation}\label{42501}
\left\{
  \begin{array}{ll}
   (\rho U)_\xi+(\rho V)_\eta+2\rho=0,\\[4pt]
  \displaystyle  U_\eta-V_\xi=0
  \end{array}
\right.
\end{equation}
supplemented by the pseudo-Bernoulli law
\begin{equation}\label{5802}
\frac{1}{2}(U^2+V^2)+h+\varphi =0.
\end{equation}

The first equation of (\ref{42501}) can by (\ref{5801}) and (\ref{5802}) be written in the form
\begin{equation}
  (c^{2}-U^{2})u_{\xi}-UV(u_{\eta}+v_{\xi})+(c^{2}-V^{2})v_{\eta}=0,\label{mass}
\end{equation}
where $c=\sqrt{-\tau^2 p'(\tau)}$ is the speed of sound.
So, for smooth flow, system (\ref{42501}) can be written in the following matrix form:
\begin{equation} \label{matrix1}
\begin{array}{rcl}
\left(
 \begin{array}{cc}
c^{2}-U^{2} & -UV \\
  0 & -1\\
  \end{array}
  \right)\left(
           \begin{array}{c}
             u \\
             v \\
           \end{array}
         \right)_{\xi}+\left(
                         \begin{array}{ccccc}
                          -UV &  c^{2}-V^{2}\\
                           1 & 0 \\
                         \end{array}
                       \right)\left(
                                \begin{array}{c}
                                  u \\
                                  v \\
                                \end{array}
                              \right)_{\eta}=\left(
                                               \begin{array}{c}
                                                 0 \\
                                                0 \\
                                               \end{array}
                                             \right).
                                             \end{array}
\end{equation}

The eigenvalues of (\ref{matrix1}) are determined
by
\begin{equation}\label{21}
(V-\lambda U)^{2}-c^{2}(1+\lambda^{2})=0,
\end{equation}
which yields
\begin{equation}\label{23}
\lambda=\lambda_{\pm}=\frac{UV\pm
c\sqrt{U^{2}+V^{2}-c^{2}}}{U^{2}-c^{2}}.
\end{equation}
So, system (\ref{42501}) is a mixed type system and the type in each point is determined by the local pseudo-Mach number $q/c$ where $q=\sqrt{U^2+V^2}$. It is hyperbolic if and only if $q/c>1$ (supersonic) and elliptic if and only if $q/c<1$ (subsonic).
For supersonic flow, system (\ref{42501}) has two families of wave characteristic curves defined by
\begin{equation}\label{24}
C_{\pm}:~\frac{{\rm d}\eta}{{\rm d}\xi}=\lambda_\pm.
\end{equation}



\subsection{Oblique shocks in 2D pseudo-steady potential flows of non-convex gases}
We assume that the equation of state (\ref{van}) has the property ($\mathbf{P}$).
Then
there exist $\tau_{1}^a$ and $\tau_{2}^a$ where $\tau_1^a<\tau_{1}^{i}<\tau_{2}^{i}<\tau_{2}^a$ such that
\begin{equation}\label{5901}
p'(\tau_{1}^a)=p'(\tau_{2}^{i})\quad \mbox{and}\quad p'(\tau_{1}^{i})=p'(\tau_{2}^a);
 \end{equation}
see Figure \ref{Fig4} (right). For any $\tau\in[\tau_1^a, \tau_1^i]$, we let $a(\tau)\in [\tau_2^i, \tau_2^a]$ be determined by  $p'(\tau)=p'(a(\tau))$. From (\ref{5901}) one has
\begin{equation}\label{101701}
a(\tau_1^a)=\tau_2^i\quad
\mbox{and} \quad a(\tau_1^i)=\tau_2^a.
\end{equation}

\begin{figure}[htbp]
\begin{center}
\includegraphics[scale=0.66]{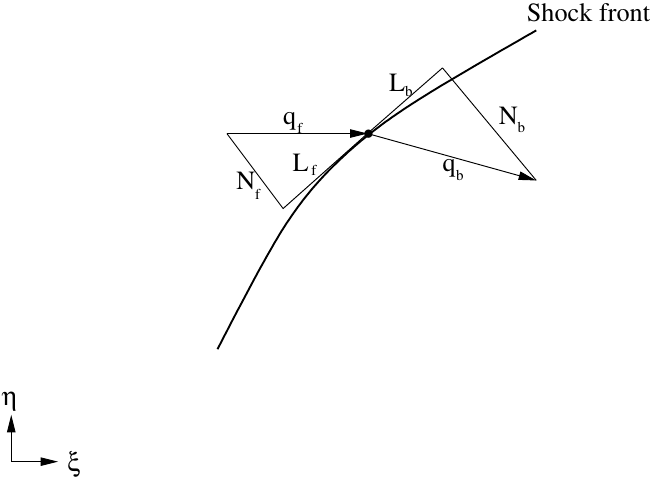}
\caption{\footnotesize A 2D self-similar shock curve.}
\label{Fig5}
\end{center}
\end{figure}

We now discuss oblique shocks for the system (\ref{42501})--(\ref{5802}) for the van der Waals gas.
Let $\chi$ be the inclination angle of the shock front. We denote by
$$
N=U\sin\chi-V\cos\chi\quad \mbox{and}\quad L=U\cos\chi+V\sin\chi
$$
the components of $(U, V)$ normal and tangential to the shock direction, respectively.
We use subscripts `f' and `b' to denote the front side and the backside states of the
shock front; see Figure \ref{Fig5}. Then
the Rankine-Hugoniot conditions can be written in the following form:
\begin{equation}\label{RH}
\left\{
  \begin{array}{ll}
    m:=\rho_f N_f=\rho_b N_b, \\[4pt]
    L_{f}=L_b,\\[4pt]
   N_f^2+2h(\tau_f)= N_b^2+2h(\tau_b).
  \end{array}
\right.
\end{equation}

Since the gas considered here is nonconvex, the oblique shocks are required to satisfy the
Liu's extended entropy condition (\cite{Liu1}):
\begin{equation}\label{43001}
-\frac{2h(\tau_f)-2h(\tau_b)}{\tau_f^2-\tau_b^2}\geq -\frac{2h(\tau_f)-2h(\tau)}{\tau_f^2-\tau^2}
\end{equation}
for all $\tau\in(\min\{\tau_f,\tau_b\}, \max\{\tau_f,\tau_b\})$.

\begin{figure}[htbp]
\begin{center}
\includegraphics[scale=0.48]{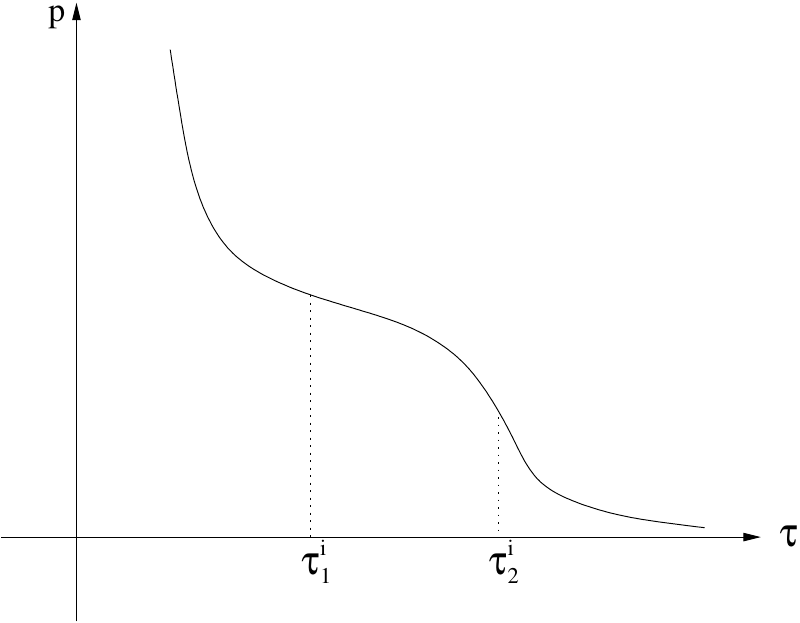}\qquad\includegraphics[scale=0.468]{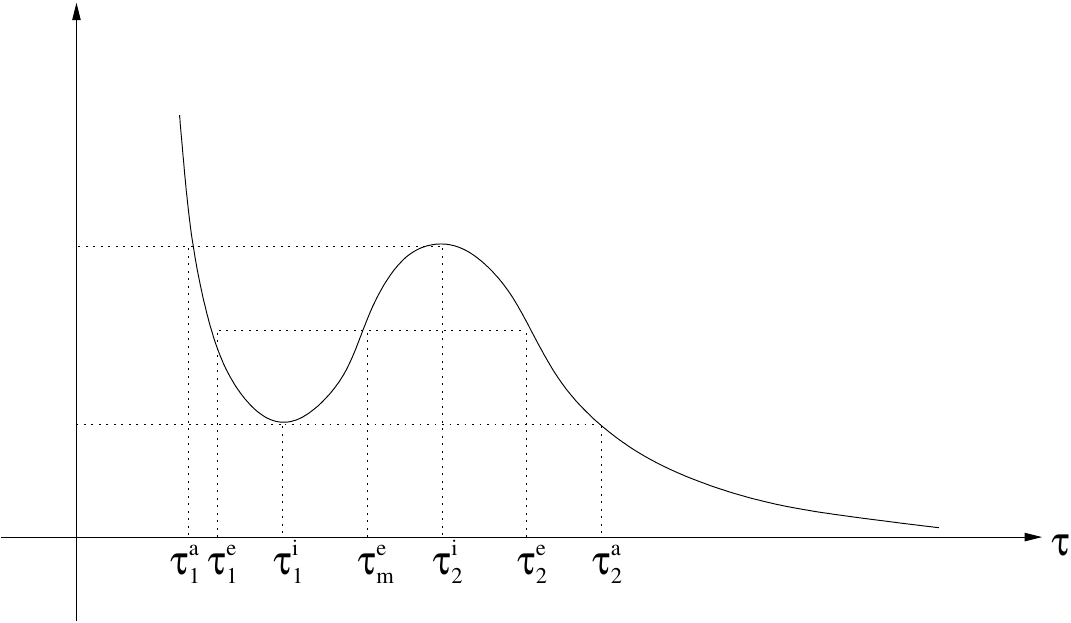}
\caption{\footnotesize Left: the isentrope $p=p(\tau)$; right: the graph of $-p'(\tau)$.}
\label{Fig4}
\end{center}
\end{figure}

\begin{prop}\label{5602}
({\bf Double-sonic})
There exists a unique pair $\tau_1^e$ and $\tau_2^e$ where $1<\tau_1^e<\tau_2^e$, such that
\begin{equation}\label{42301}
p'(\tau_1^e)=p'(\tau_2^e)=\frac{2h(\tau_1^e)-2h(\tau_2^e)}{(\tau_1^e)^2-(\tau_2^e)^2},
\end{equation}
and
\begin{equation}\label{42301aa}
-\frac{2h(\tau_1^e)-2h(\tau_2^e)}
{(\tau_1^e)^2-(\tau_2^e)^2}>-\frac{2h(\tau_1^e)-2h(\tau)}{(\tau_1^e)^2-\tau^2}
\end{equation}
for all $\tau\in(\tau_1^e, \tau_2^e)$.
Moreover, $\tau_1^e\in (\tau_1^a, \tau_1^i)$ and $\tau_2^e\in (\tau_2^i, \tau_2^a)$.
\end{prop}

\begin{proof}
 Firstly, by Property (P) and the Cauchy mean value theorem we know that if there exists $\tau_1^e$ and $\tau_2^e$ such that (\ref{42301}) holds then
$\tau_1^e\in (\tau_1^a, \tau_1^i)$ and $\tau_2^e=a(\tau_1^e)\in (\tau_2^i, \tau_2^a)$.
The rest of proof proceeds in four steps.
\vskip 2pt

\noindent
{\bf 1.} We first prove the existence of a $\tau_1^e$.
Set
$$
g(\tau)=-p'(\tau)+\frac{2h(\tau)-2h(a(\tau))}{\tau^2-a^2(\tau)}, \quad \tau\in [\tau_1^a, \tau_1^i].
$$
By virtue of the mean value theorem and recalling (\ref{5801}) and (\ref{101701}), one has
$$
g(\tau_1^a)=-p'(\tau_1^a)+\frac{2h(\tau_1^a)-2h(\tau_2^i)}{(\tau_1^a)^2-(\tau_2^i)^2}>0
$$
and
$$
g(\tau_1^i)=-p'(\tau_1^i)+\frac{2h(\tau_1^i)-2h(\tau_2^a)}{(\tau_1^i)^2-(\tau_2^a)^2}<0.
$$
Thus, there exists a $\tau_1^e\in (\tau_1^a, \tau_1^i)$ such that (\ref{42301}) holds for $\tau_2^e=a(\tau_1^e)$.

\vskip 4pt

\noindent
{\bf 2.}
Assume $\tau_1^e\in (\tau_1^a, \tau_1^i)$ satisfies (\ref{42301}).
 We shall show that for any $\tau\in(\tau_1^e, \tau_2^i)$, there exists one and only one $s_{po}(\tau)\in(\tau,+\infty)$ such that
\begin{equation}\label{5904}
p'(s_{po}(\tau))=\frac{2h(s_{po}(\tau))-2h(\tau)}{s_{po}^2(\tau)-\tau^2}<p'(\tau).
\end{equation}

Differentiating the equation
$p'(s_{po}(\tau))(s_{po}^2(\tau)-\tau^2)=2h(s_{po}(\tau))-2h(\tau)$ on both sides
with respect to $\tau$ and recalling (\ref{5801}), one gets
\begin{equation}\label{5902}
 -p''(s_{po}(\tau))(\tau^2-s_{po}^2(\tau))s_{po}'(\tau)=2\tau(p'(s_{po}(\tau))-p'(\tau)).
\end{equation}
To prove the existence,
we consider (\ref{5902}) with data
\begin{equation}\label{59022}
    s_{po}(\tau_1^e)=\tau_2^e.
\end{equation}

From $\tau_1^e\in (\tau_1^a, \tau_1^i)$ we have  $\tau_2^e>\tau_2^i$.
Then the initial value problem (\ref{5902}, \ref{59022}) admits a local solution and $s_{po}'(\tau_1^e)=0$.
We are going to show that the initial value problem (\ref{5902}, \ref{59022}) admits a solution $s_{po}(\tau)$ for $\tau\in(\tau_1^e, \tau_2^i)$. Moreover, the solution satisfies
\begin{equation}\label{5906}
p'(s_{po}(\tau))<p'(\tau), \quad  s_{po}'(\tau)<0,\quad \mbox{and}\quad \tau_2^i<s_{po}(\tau)<\tau_2^e\quad \mbox{for}~  \tau\in (\tau_1^e, \tau_2^i). 
\end{equation}

We differentiate (\ref{5902}) with respect to $\tau$ to yield
\begin{equation}\label{5903}
\begin{aligned}
&-p''(s_{po}(\tau))(\tau^2-s_{po}^2(\tau))s_{po}''(\tau)
-p'''(s_{po}(\tau))(\tau^2-s_{po}^2(\tau))(s_{po}'((\tau))^2\\&\qquad-p''(s_{po}(\tau))(2\tau-2s_{po}
(\tau)s_{po}'(\tau))s_{po}'(\tau)
\\~=~&2(p'(s_{po}(\tau))-p'(\tau))+2\tau(p''(s_{po}(\tau))s_{po}'(\tau)-p''(\tau)).
\end{aligned}
\end{equation}
So, by $ s_{po}(\tau_1^e)=\tau_2^e$ and $s_{po}'(\tau_1^e)=0$ we obtain $$s_{po}''(\tau_1^e)=\frac{2\tau_1^e p''(\tau_1^e)}{((\tau_1^e )^2-(\tau_2^e)^2)p''(\tau_2^e)}<0.$$
This implies that $s_{po}'(\tau)<0$ and $\tau_2^i<s_{po}(\tau)<\tau_2^e$
 as $\tau-\tau_1^e>0$ is small.
 Furthermore, by (\ref{81601}) we know that
the first inequality in (\ref{5906}) holds when $\tau-\tau_1^e>0$ is small.

Next, we shall use the argument of continuity to prove (\ref{5906}).
We assert that for any $\tau_{m}\in (\tau_1^e, \tau_2^i)$, if
the inequalities in (\ref{5906}) hold for $\tau\in (\tau_1^e, \tau_m)$
then they also hold at $\tau_m$.

From $s_{po}'(\tau)<0$ for $\tau\in (\tau_1^e, \tau_m)$, we have $s_{po}(\tau_m)<\tau_2^e$.
 Since $p'(\tau_2^i)<p'(\tau)$ for $\tau\in (\tau_m, \tau_2^i)$, by virtue of the Cauchy mean value theorem we have
 $$
 p'(\tau_2^i)<\frac{2h(\tau_2^i)-2h(\tau_m)}{(\tau_2^i)^2-\tau_m^2}.
 $$
 So, $s_{po}(\tau_m)\neq\tau_2^i$.
This implies $s_{po}(\tau_m)>\tau_2^i$.

Suppose $p'(s_{po}(\tau_m))=p'(\tau_m)$. Then by $\tau_2^i <s_{po}(\tau_m)<\tau_2^e$ we have $\tau_m>\tau_1^i$. While,  by the Cauchy mean value theorem
we have
$$
\frac{2h(s_{po}(\tau_m))-2h(\tau_m)}{s_{po}^2(\tau_m)-\tau_m^2}<p'(s_{po}(\tau_m)),
$$
since $p'(\tau)<p'(s_{po}(\tau_m))$ for $\tau\in (\tau_m, s_{po}(\tau_m))$. This leads to a contradiction.
So, we obtain $p'(s_{po}(\tau_m))<p'(\tau_m)$.

 Using (\ref{5902}) and recalling $p'(s_{po}(\tau_m))<p'(\tau_m)$ and $s_{po}(\tau_m)>\tau_2^i$, one immediately has $s_{po}'(\tau_m)<0$.
This completes the proof of the assertion.
Therefore by the argument of continuity we know that
the initial value problem (\ref{5902}) admits a solution $s_{po}(\tau)$ for $\tau\in(\tau_1^e, \tau_2^i)$ and the solution satisfies (\ref{5906}).

By the Cauchy mean value theorem,
we also have $s_{po}(\tau)\rightarrow \tau_2^i$ as $\tau\rightarrow\tau_2^i-0$.

\vskip 4pt

We now prove the uniqueness of $s_{po}(\tau)\in(\tau,+\infty)$.
For any $\tau\in (\tau_1^e, \tau_2^i)$,
we let
\begin{equation}\label{5907}
r(s)=p'(s)(s^2-\tau^2)-2h(s)+2h(\tau), \quad s\geq\tau.
\end{equation}
We divide the discussion into the following two cases: (1) $\tau_1^e<\tau< \tau_1^i$; (2) $\tau_1^i\leq \tau<\tau_2^i$.

(1) $\tau_1^e<\tau< \tau_1^i$.
Then by the Cauchy mean value theorem we have
$$
r(\tau_1^i)=\big((\tau_1^i)^2-\tau^2\big)\left(p'(\tau_1^i)-\frac{2h(\tau_1^i)-2h(\tau)}{(\tau_1^i)^2-\tau^2}\right)>0
$$
and
\begin{equation}\label{81604}
r(\tau_2^i)=\big((\tau_2^i)^2-\tau^2\big)\left(p'(\tau_2^i)-\frac{2h(\tau_2^i)-2h(\tau)}{(\tau_2^i)^2-\tau^2}\right)<0.
\end{equation}
Furthermore, by $r'(s)=p''(s)(s^2-\tau^2)$ one knows that the equation $r(s)=0$ only has two roots $s_{po}(\tau)$ and $\bar{s}_{po}(\tau)$ in $(\tau, +\infty)$, where $s_{po}(\tau)\in(\tau_2^i, \tau_2^e)$ is already obtained in the proof of existence and $\bar{s}_{po}(\tau)\in (\tau_1^i, \tau_2^i)$.

Since $\tau_1^e<\tau< \tau_1^i$, there exists one and only one $\tau'\in (\tau_1^i, \tau_2^i)$ such that $p'(\tau)=p'(\tau')$.
Thus by the Cauchy mean value theorem one has
$\bar{s}_{po}(\tau)<\tau'$.
Combining this with $\bar{s}_{po}(\tau)>\tau_1^i$ and (\ref{81601}) we have
 $p'(\bar{s}_{po}(\tau))>p'(\tau)$ which violates the right inequality of (\ref{5904}).
So we get the uniqueness of $s_{po}(\tau)$ for $\tau\in (\tau_1^e, \tau_1^i)$.

(2) $\tau_1^i\leq \tau< \tau_2^i$. Using $r(\tau)=0$, (\ref{81604}), and $r'(s)=p''(s)(s^2-\tau^2)$ one knows that the equation $r(s)=0$ only has only one root $s_{po}(\tau)$ which is already obtained in the proof of existence. So we get the uniqueness of $s_{po}(\tau)$ for $\tau\in [\tau_1^i, \tau_2^i)$.

\vskip 4pt

\noindent
{\bf 3.} We now prove the uniqueness of $\tau_1^e$.
Assume that $\tau_1^e\in (\tau_1^a, \tau_1^i)$ satisfies (\ref{42301}).
Assume as well that there exists a $\tau_3^e\in (\tau_1^e, \tau_1^i)$ such that
$$
p'(\tau_3^e)=p'(\tau_4^e)=\frac{2h(\tau_3^e)-2h(\tau_4^e)}{(\tau_3^e)^2-(\tau_4^e)^2},
$$
where $\tau_4^e=a(\tau_3^e)>\tau_2^e$. While, by the result of the previous step we have
$$
p'(s)(s^2-(\tau_3^e)^2)-2h(s)+2h(\tau_3^e)>0\quad \mbox{for}\quad s>\tau_2^e.
$$
 This leads to a contradiction.

\vskip 4pt

\noindent
{\bf 4.}
Set $\hat{g}(\tau)=2h(\tau)-2h(\tau_1^e)-p'(\tau_1^e)\tau^2+p'(\tau_1^e)(\tau_1^e)^2$.
We have
\begin{equation}\label{43002}
\hat{g}(\tau_1^e)=\hat{g}(\tau_2^e)=0
\end{equation}
and
$
\hat{g}'(\tau)=2\tau(p'(\tau)-p'(\tau_1^e))
$.
Since $\tau_1^a<\tau_1^e<\tau_1^i$, there exists a $\tau_m^e\in(\tau_1^i, \tau_2^i)$ such that $p'(\tau_m^e)=p'(\tau_1^e)$. Consequently, we have
$$
\hat{g}'(\tau)>0\quad \mbox{for}\quad \tau_1^e<\tau<\tau_m^e; \quad \hat{g}'(\tau)<0\quad \mbox{for}\quad \tau_m^e<\tau<\tau_2^e.
$$
Combining this with (\ref{43002}) we have $\hat{g}(\tau)>0$ for $\tau\in(\tau_1^e, \tau_2^e)$.
This immediately implies  that (\ref{42301aa}) holds for all $\tau\in(\tau_1^e, \tau_2^e)$.

This completes the proof Proposition \ref{5602}.
\end{proof}


\begin{prop}\label{51003}
({\bf Post-sonic}) For any $\tau\in(\tau_1^e, \tau_2^i)$, there exists one and only one $s_{po}(\tau)\in(\tau,+\infty)$ such that
\begin{equation}\label{43003}
p'(s_{po}(\tau))=\frac{2h(s_{po}(\tau))-2h(\tau)}{s_{po}^2(\tau)-\tau^2}<p'(\tau),
\end{equation}
and
\begin{equation}\label{43004}
-\frac{2h(s_{po}(\tau))-2h(\tau)}{s_{po}^2(\tau)-\tau^2}>-\frac{2h(s)-2h(\tau)}{s^2-\tau^2}
\end{equation}
for all $s\in (\tau, s_{po}(\tau))$.
Moreover, $s_{po}(\tau)$ satisfies (\ref{5906}) and
\begin{equation}\label{43005}
s_{po}(\tau)\rightarrow \tau_2^i\quad \mbox{as}\quad  \tau\rightarrow\tau_2^i-0.
\end{equation}
\end{prop}
\begin{proof}
By  the result of the second step in the proof of Proposition \ref{5602}, one has that for any $\tau\in(\tau_1^e, \tau_2^i)$, there exists one and only one $s_{po}(\tau)\in(\tau,+\infty)$ such that (\ref{43003}) holds, and $s_{po}(\tau)$ satisfies (\ref{5906}) and
(\ref{43005}).

Set $\check{g}(s)=2h(s)-2h(\tau)-p'(s_{po}(\tau))(s^2-\tau^2)$.
We have
\begin{equation}\label{43006}
\check{g}(\tau)=\check{g}(s_{po}(\tau))=0
\end{equation}
and
$\check{g}'(s)=2s(p'(s)-p'(s_{po}(\tau)))$.
Since $\tau_2^i<s_{po}(\tau)<\tau_2^e$, there exists a $\tau_m\in(\tau_m^e, \tau_2^i)$ such that $p'(\tau_m)=p'(s_{po}(\tau))$. Consequently, we have
$$
\check{g}'(s)>0\quad \mbox{for}\quad \tau<s<\tau_m; \quad \check{g}'(s)<0\quad \mbox{for}\quad \tau_m<\tau<s_{po}(\tau).
$$
Combining this with (\ref{43006}) we have $\check{g}(s)>0$ for $s\in(\tau, s_{po}(\tau))$.
This immediately implies  that (\ref{43004}) holds for all $s\in (\tau, s_{po}(\tau))$.
\end{proof}

\begin{prop}\label{pre}
({\bf Pre-sonic})
For any $\tau\in (\tau_1^e, \tau_1^i)$, there exists one and only one $s_{pr}(\tau)\in (\tau, +\infty)$ such that
\begin{equation}\label{5701}
p'(\tau)=\frac{2h(s_{pr}(\tau))-2h(\tau)}{s_{pr}^2(\tau)-\tau^2}>p'(s_{pr}(\tau))
\end{equation}
and
\begin{equation}\label{43007}
-\frac{2h(s_{pr}(\tau))-2h(\tau)}{s_{pr}^2(\tau)-\tau^2}>-\frac{2h(\tau)-2h(s)}{\tau^2-s^2}
\end{equation}
for all $s\in (\tau, s_{pr}(\tau))$.
Moreover, one has $s_{pr}(\tau)\in (\tau_1^i, \tau_2^e)$ for $\tau\in (\tau_1^e, \tau_1^i)$.
\end{prop}
\begin{proof}
The proof proceeds in three steps.

\vskip 2pt
\noindent
{\bf 1.}
We first prove the existence.
Differentiating $p'(\tau)(s_{pr}^2(\tau)-\tau^2)=2h(s_{pr}(\tau))-2h(\tau)$ with respect to $\tau$, one gets
\begin{equation}\label{5702}
2s_{pr}\big(p'(s_{pr})-p'(\tau)\big){\rm d}s_{pr}=p''(\tau)(s_{pr}^2-\tau^2){\rm d}\tau.
\end{equation}
We consider (\ref{5702}) with data
\begin{equation}\label{5703}
s_{pr}\big|_{\tau=\tau_1^e}=\tau_2^e.
\end{equation}

A direct computation yields
$$
\frac{{\rm d}\tau}{{\rm d}s_{pr}}\bigg|_{s_{pr}=\tau_2^e}=0\quad \mbox{and}\quad \frac{{\rm d^2}\tau}{{{\rm d}s_{pr}^2}}\bigg|_{s_{pr}=\tau_2^e}>0.
$$
So, there exists a small $\varepsilon>0$ such that the initial value problem (\ref{5702}, \ref{5703}) admits a solution with the properties
\begin{equation}\label{51001}
s_{pr}>\tau_1^i, \quad p'(s_{pr})<p'(\tau), \quad \mbox{and}\quad \frac{{\rm d}s_{pr}}{{\rm d}\tau}<0
\end{equation}
for $\tau\in (\tau_1^e, \tau_1^e+\varepsilon)$.

We next prove the following assertion: for any $\tau_m\in (\tau_1^e, \tau_1^i)$, if the solution satisfies (\ref{51001}) for $\tau\in  (\tau_1^e, \tau_m)$ then it also satisfies (\ref{51001}) at $\tau_m$.

Suppose $p'(s_{pr}(\tau_m))=p'(\tau_m)$. Then we have $s_{_{pr}}(\tau_m)\in (\tau_1^i, \tau_2^i)$ and by the Cauchy mean value theorem one has $\frac{2h(s_{pr}(\tau_m))-2h(\tau_m)}{s_{pr}^2(\tau_m)-\tau_m^2}>p'(\tau_m)$. This leads to a contradiction. So, $p'(s_{pr}(\tau_m))<p'(\tau_m)$. From this one immediately has  $s_{pr}'(\tau_m)<0$.
 By the Cauchy mean value theorem we also have $s_{pr}(\tau_m)>\tau_1^i$.
 This completes the proof of the assertion.

  From the assertion we know that the initial value problem (\ref{5702}, \ref{5703}) admits a solution for $\tau\in (\tau_1^e, \tau_1^i)$ such that (\ref{51001}) hold. Moreover, by the Cauchy mean value theorem one has $s_{pr}(\tau)\rightarrow \tau_1^i$ as $\tau\rightarrow \tau_1^i$.

\noindent
{\bf 2.}
We next prove the uniqueness.
For any fixed $\tau\in(\tau_1^e, \tau_1^i)$, we let
$$
\hat{r}(s)=p'(\tau)(s^2-\tau^2)-2h(s)+2h(\tau).
$$
From $\hat{r}'(s)=2s(p'(\tau)-p'(s))$ one can see that
 equation $\hat{r}(s)=0$ has at least one root for $s>\tau$ such that  $p'(s)<p'(\tau)$.

\noindent
{\bf 3.}
Set $\breve{g}(s)=2h(s)-2h(\tau)-p'(\tau)(s^2-\tau^2)$.
We have
\begin{equation}\label{43010}
\breve{g}(\tau)=\breve{g}(s_{pr}(\tau))=0
\end{equation}
and
$\breve{g}'(s)=2s(p'(s)-p'(\tau)).$
Since $\tau_1^e<\tau<\tau_1^i$, there exists a $\tau_m\in(\tau_1^i, \tau_m^e)$ such that $p'(\tau_m)=p'(\tau)$. Consequently, we have
$$
\breve{g}'(s)>0\quad \mbox{for}\quad \tau<s<\tau_m; \quad \breve{g}'(s)<0\quad \mbox{for}\quad \tau_m<\tau<s_{pr}(\tau).
$$
Combining this with (\ref{43010}) we have $\breve{g}(s)>0$ for $s\in(\tau, s_{pr}(\tau))$.
This immediately implies  that (\ref{43007}) holds for all $s\in (\tau, s_{pr}(\tau))$.

This completes the proof of the proposition.
\end{proof}

\begin{prop}
For $\tau\in (\tau_1^e, \tau_1^i)$, there holds $s_{po}(\tau)>s_{pr}(\tau)$.
\end{prop}
\begin{proof}
From $s_{po}(\tau)\in (\tau_2^i, \tau_2^e)$ and $\tau\in (\tau_1^e, \tau_1^i)$ we have $p'(s_{po}(\tau))<p'(\tau)$.

Suppose $s_{pr}(\tau)\geq s_{po}(\tau)$. Then by (\ref{43003}) and (\ref{43007}) we have
$p'(s_{po}(\tau))\geq p'(\tau)$. This leads to a contradiction.
Then we have $s_{po}(\tau)>s_{pr}(\tau)$ for $\tau\in (\tau_1^e, \tau_1^i)$.
\end{proof}

\begin{prop}\label{tran}
({\bf Transonic}) Assume $\tau_f\in (\tau_1^e, \tau_1^i)$. Then,
\begin{equation}\label{43012}
p'(\tau_b)<\frac{2h(\tau_b)-2h(\tau_f)}{\tau_b^2-\tau_f^2}<p'(\tau_f)\quad \mbox{iff}\quad \tau_b\in (1, \tau_f)\cup(s_{pr}(\tau_f), s_{po}(\tau_f)).
\end{equation}
Moreover,
\begin{equation}\label{43013}
-\frac{2h(\tau_b)-2h(\tau_f)}{\tau_b^2-\tau_f^2}>-\frac{2h(s)-2h(\tau_f)}{s^2-\tau_f^2}
\end{equation}
for all $s\in (\min\{\tau_f, \tau_b\}, \max\{\tau_f, \tau_b\})$.
\end{prop}

\begin{proof}
The proof proceeds in four steps.

\vskip 2pt
\noindent
{\bf 1.} Since $\tau_f\in(\tau_1^e, \tau_1^i)$, one has $p''(\tau)>0$ for $\tau\in (1, \tau_f)$. This implies
$p'(\tau_b)<\frac{2h(\tau_b)-2h(\tau_f)}{\tau_b^2-\tau_f^2}<p'(\tau_f)$ for $\tau_b\in (1, \tau_f)$.

\vskip 2pt
\noindent
{\bf 2.}
Let $\hat{r}(\tau_b)=p'(\tau_f)(\tau_b^2-\tau_f^2)-2h(\tau_b)+2h(\tau_f)$. Then we have $\hat{r}(s_{pr}(\tau_f))=0$
and $\hat{r}'(\tau_b)=2\tau_b(p'(\tau_f)-p'(\tau_b))$.
By the definition of $s_{po}(\tau_f)$ and $s_{pr}(\tau_f)$ we know that
$p'(\tau_b)<p'(\tau_f)$ for $\tau_b\in (s_{pr}(\tau_f), s_{po}(\tau_f))$.
Thus we have $\hat{r}(\tau_b)>0$ for $\tau_b\in (s_{pr}(\tau_f), s_{po}(\tau_f))$. This implies
\begin{equation}\label{250501}
\frac{2h(\tau_b)-2h(\tau_f)}{s^2-\tau^2}<p'(\tau_f)\quad \mbox{for}\quad  \tau_b\in [s_{pr}(\tau_f), s_{po}(\tau_f)].
\end{equation}

We next prove $\hat{r}(\tau_b)<0$ for $\tau_b\in (\tau_f, s_{pr}(\tau_f))$. Suppose there exists a $\tau_*\in (\tau_f, s_{pr}(\tau_f))$ such that $\hat{r}(\tau_*)=0$. Then by $\hat{r}(\tau_f)=0$ and $\hat{r}(s_{pr}(\tau_f))=0$ we know that $\hat{r}'(\tau_b)=0$
has two roots in $(\tau_f, s_{pr}(\tau_f))$. This is impossible, since $s_{pr}(\tau_f)<\tau_2^e$.  So we have $\hat{r}(\tau_b)<0$ for $\tau_b\in (\tau_f, s_{pr}(\tau_f))$.

\vskip 2pt
\noindent
{\bf 3.}
Let ${r}(\tau_b)=p'(\tau_b)(\tau_b^2-\tau_f^2)-2h(\tau_b)+2h(\tau_f)$. Then we have
\begin{equation}\label{250502}
{r}(s_{po}(\tau_f))=0\quad \mbox{and}\quad {r}'(s_{po}(\tau_f))=p''(s_{po}(\tau_f))(s_{po}^2(\tau_f)-\tau_f^2)>0.
\end{equation}
Combining (\ref{250501}), (\ref{250502}) the uniqueness of $s_{pr}(\tau_f)$ we have $p'(\tau_b)<\frac{2h(\tau_b)-2h(\tau_f)}{\tau_b^2-\tau_f^2}$ for $\tau_b\in (s_{pr}(\tau_f), s_{po}(\tau_f))$. From (\ref{250502}) we also have $p'(\tau_b)>\frac{2h(\tau_b)-2h(\tau_f)}{\tau_b^2-\tau_f^2}$ for $\tau_b>s_{po}(\tau_f)$.

\vskip 2pt
\noindent
{\bf 4.} 
Let
$\bar{g}(s)=(2h(s)-2h(\tau_f))(\tau_b^2-\tau_f^2)-(2h(\tau_b)-2h(\tau_f))(s^2-\tau_f^2)$.
Then we have $\bar{g}(\tau_f)=\bar{g}(\tau_b)=0$.
We compute
$$
\bar{g}'(s)=2s(\tau_b^2-\tau_f^2)\Big(p'(s)-\frac{2h(\tau_b)-2h(\tau_f)}{\tau_b^2-\tau_f^2}\Big).
$$
By (\ref{43012}) we have
\begin{equation}\label{50501}
\hat{g}'(\tau_f)>0\quad \mbox{and}\quad  \hat{g}'(\tau_b)<0.
\end{equation}
Suppose there exits a point $\tau_m\in (\tau_f, \tau_b)$ such that $\hat{g}(\tau_m)=0$. Then by (\ref{50501}) we know that there exist at least three points $\tau_{k}$ ($k=1, 3, 3$) in $(\tau_f, \tau_b)$ such that
$p'(\tau_k)=\frac{2h(\tau_b)-2h(\tau_f)}{\tau_b^2-\tau_f^2}$  ($k=1, 3, 3$).
This is impossible, since $\tau_1^e<\tau_f<\tau_b<\tau_2^e$.
So, we get $\bar{g}(s)>0$ for $s\in (\tau_f, \tau_b)$. Hence, the inequality (\ref{43013}) holds for all $s\in (\tau_f, \tau_b)$.
The proof of (\ref{43013}) for $\tau_b\in (1, \tau_f)$ is similarly. We omit the details.

This completes the proof.
\end{proof}

From the first and the third relations of (\ref{RH}), we have
$$
m^2=-\frac{2h(\tau_f)-2h(\tau_b)}{\tau_f^2-\tau_b^2}.
$$
So, we have the following classifications about the oblique shocks in 2D pseudo-steady potential flows of the polytropic van der Waals gases:
\begin{itemize}
  \item for oblique double-sonic shocks, i.e., $\tau_f=\tau_1^e$ and $\tau_b=\tau_2^e$, one has
$$
N_f=c_f:=\sqrt{-\tau_f^2p'(\tau_f)} \quad \mbox{and}\quad N_b=c_b:=\sqrt{-\tau_b^2p'(\tau_b)};
$$
  \item for oblique post-sonic shocks, i.e., $\tau_f\in (\tau_1^e, \tau_2^i)$ and $\tau_b=s_{po}(\tau_f)$, one has
$$
N_f>c_f\quad \mbox{and}\quad N_b=c_b;
$$
\item  for oblique pre-sonic shocks, i.e., $\tau_f\in (\tau_1^e, \tau_1^i)$ and $\tau_b=s_{pr}(\tau_f)$, one has
$$
N_f=c_f\quad \mbox{and}\quad N_b<c_b;
$$
\item for oblique transonic shocks,  one has
$$
N_f>c_f\quad \mbox{and}\quad N_b<c_b.
$$
\end{itemize}

\subsection{Interaction of fan-shock-fan composite waves}
To study the expansion into vacuum of a wedge of gas for potential flow,
 we consider (\ref{Potential}) with data
 \begin{equation}\label{initial1}
(\rho, \Phi)(x,y,0)=\left\{
              \begin{array}{ll}
                (\rho_0, 0), & \hbox{$x>0$, $|y|< x\tan\theta$;} \\[2pt]
               \mbox{vacuum}, & \hbox{otherwise.}
              \end{array}
            \right.
 \end{equation}

Let $\tau_0=1/\rho_0$.
Then when $1<\tau_0<\tau_1^e$, the gas away from the sharp corner expands into the vacuum as two symmetric planar fan-shock-fan composite
waves composed of an expansion fan, a shock, and another expansion fan. In order to construct the planar fan-shock-fan composite waves of the gas expansion in vacuum problem,
we consider (\ref{Potential}) with data
\begin{equation}\label{1D}
(\rho, \Phi)(x, y, 0)=\left\{
                        \begin{array}{ll}
                          (\rho_0, 0), & \hbox{$n_1x+n_2y>0$;} \\[4pt]
                          \mbox{vacuum}, & \hbox{$n_1x+n_2y<0$;}
                        \end{array}
                      \right.
\end{equation}
see Figure \ref{Fig7} (left).

\begin{figure}[htbp]
\begin{center}
\qquad\quad\includegraphics[scale=0.48]{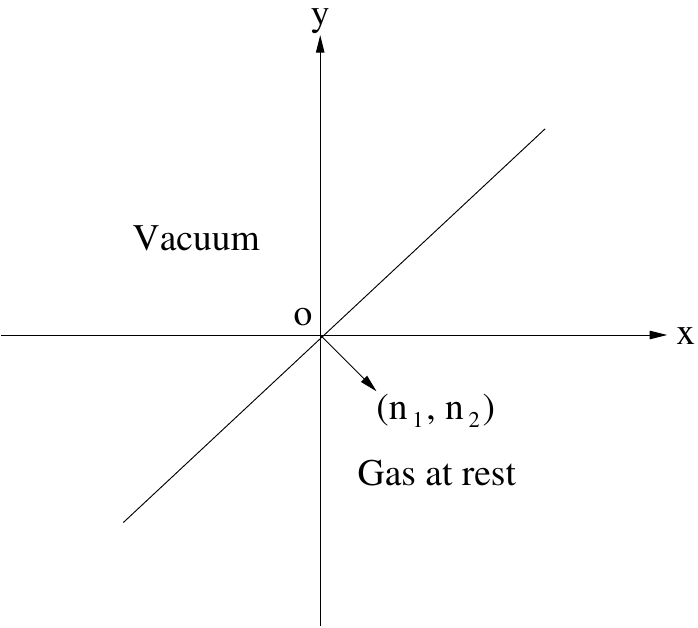}\qquad\qquad\includegraphics[scale=0.48]{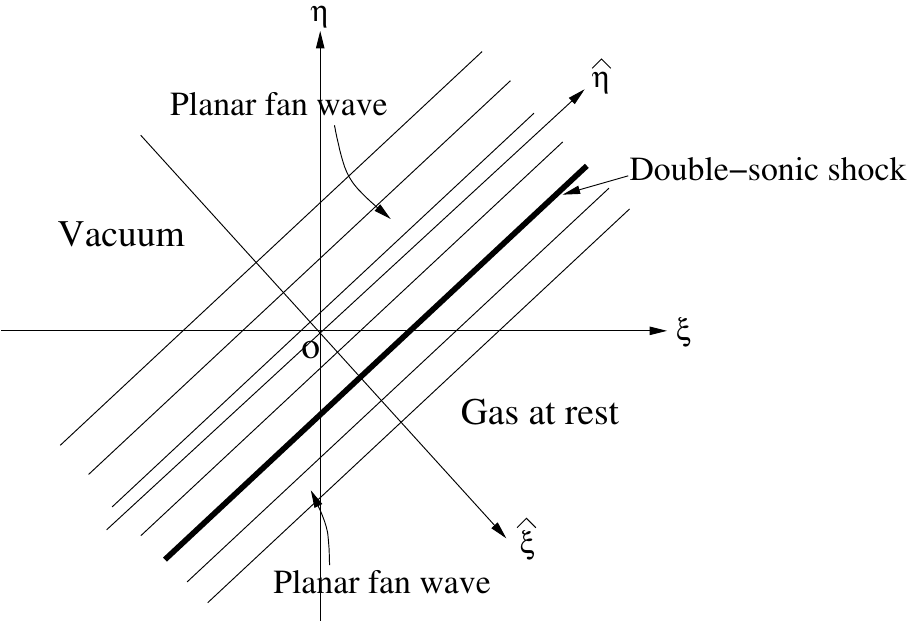}
\caption{\footnotesize A planar fan-shock-fan composite wave in the $(\xi,\eta)$-plane.}
\label{Fig7}
\end{center}
\end{figure}

The problem (\ref{Potential}, \ref{1D}) is actually a 1D  Riemann problem,
if we make the coordinate transformation
$\hat{x}=n_1 x+n_2 y$ and $\hat{y}=-n_2x+n_1y$. Let $\hat{u}=n_1 u+n_2 v$ and $\hat{v}=-n_2 u+n_1v$.
Then by solving a 1D Riemann problem for non-convex gases (see, e.g., \cite{Wen1}),  the self-similar solution to the problem (\ref{Potential}, \ref{1D}) takes the following form:
\begin{equation}\label{planar}
(\hat{u}, \hat{v}, \tau)(\hat{\xi}, \hat{\eta})=\left\{
                          \begin{array}{ll}
(0, 0, \tau_0), & \hbox{$\hat{\xi}>\hat{\xi}_0$;} \\[4pt]
                            \big(\hat{u}_r(\hat{\xi}), 0, \hat{\tau}_r(\hat{\xi})\big), & \hbox{$\hat{\xi}_1<\hat{\xi}<\hat{\xi}_0$;} \\[4pt]
                            \big(\hat{u}_l(\hat{\xi}), 0, \hat{\tau}_l(\hat{\xi})\big), & \hbox{$\hat{\xi}_2<\hat{\xi}<\hat{\xi}_1$;} \\[2pt]
                            \mbox{vacuum}, & \hbox{$\hat{\xi}<\hat{\xi}_2$.}
                          \end{array}
                        \right.
\end{equation}
In (\ref{planar}), $(\hat{\xi}, \hat{\eta})=(\frac{\hat{x}}{t}, \frac{\hat{y}}{t})$;  $\hat{\xi}_0=c_0:=\sqrt{-\tau_0^2p'(\tau_0)}$;  $\hat{u}_r(\hat{\xi})$ and $\hat{\tau}_r(\hat{\xi})$ are determined by $$
\hat{\tau}_r(\hat{\xi}_0)=\tau_0, \quad \hat{u}_r(\hat{\xi})+\int_{\tau_0}^{\hat{\tau}_r(\hat{\xi})}\sqrt{-p'(\tau)}{\rm d}\tau=0\quad \mbox{and}\quad \hat{\xi}=\hat{u}_r(\hat{\xi})+\sqrt{-\hat{\tau}_r^2(\hat{\xi})p'(\hat{\tau}_r(\hat{\xi}))};
$$
$\hat{\xi}_1$ is determined by $\hat{\tau}_r(\hat{\xi}_1)=\tau_1^e$;
$\hat{u}_1=\hat{u}_r(\hat{\xi}_1)$; $\hat{u}_2=\frac{\tau_2^e}{\tau_1^e}(\hat{u}_1-\hat{\xi}_1)+\hat{\xi}_1$; $\hat{\xi}_2=\hat{u}_2-\int_{\tau_2^e}^{+\infty}\sqrt{-p'(\tau)}{\rm d}\tau$;
$\hat{u}_l(\hat{\xi})$ and $\hat{\tau}_l(\hat{\xi})$ are determined by $$
\hat{\tau}_l(\hat{\xi}_1)=\tau_2^e, \quad \hat{u}_l(\hat{\xi})+\int_{\tau_2^e}^{\hat{\tau}_l(\hat{\xi})}\sqrt{-p'(\tau)}{\rm d}\tau=\hat{u}_2\quad \mbox{and}\quad \hat{\xi}=\hat{u}_l(\hat{\xi})+\sqrt{-\hat{\tau}_l^2(\hat{\xi})p'(\hat{\tau}_l(\hat{\xi}))}.
$$
The formula (\ref{planar}) represents a fan-shock-fan composite wave solution of (\ref{42501}), and the straight line $\hat{\xi}=\hat{\xi}_1$ is the double-sonic shock line; see Figure \ref{Fig7} (right).

\begin{figure}[htbp]
\begin{center}
\includegraphics[scale=0.35]{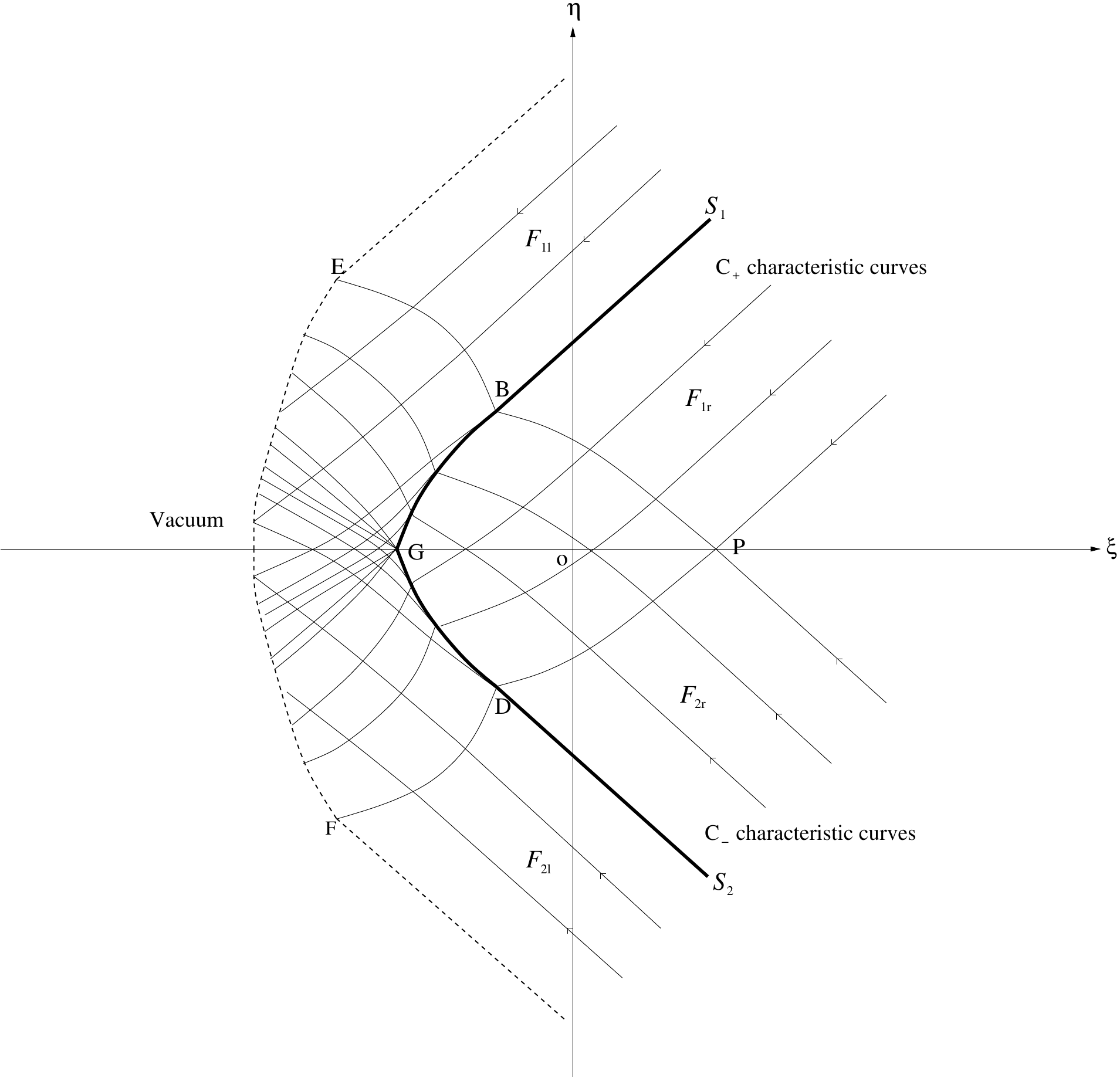}
\caption{\footnotesize Wave structure of the interaction of the planar fan-shock-fan composite waves in the $(\xi,\eta)$-plane. Here, the solid lines represent characteristic lines; the thick lines represent shock waves; the dashed lines represent vacuum boundaries.}
\label{Fig8}
\end{center}
\end{figure}

For the gas expansion in vacuum problem (\ref{Potential}, \ref{initial1}), the gas away from the sharp corner expands to the vacuum as two symmetric planar fan-shock-fan composite
waves ${\it FSF}_1$ and ${\it FSF}_2$ which are expressed in (\ref{planar}) with $(n_1, n_2)=(\sin\theta, -\cos\theta)$ and
$(n_1, n_2)=(\sin\theta, \cos\theta)$, respectively.
The fan-shock-fan composite waves ${\it FSF}_i$ ($i=1, 2$) are composed of an expansion fan wave
${\it F}_{ir}$,
 a double-sonic shock wave
${\it S}_i$,  and another expansion fan wave ${\it F}_{il}$; see Figure \ref{Fig8}.
The two composite waves start to interact with each other at the point $\mbox{P}=\big(\hat{\xi}_0/\sin\theta, 0\big)$ due to the presence of the sharp corner.

From the point $\mbox{P}$,
we draw a $C_{-}$ and a $C_{+}$ cross characteristic curves in ${\it F}_{1r}$ and ${\it F}_{2r}$, respectively.
The $C_{-}$ and $C_{+}$ cross characteristic curves intersect
${\it S}_{1}$ and ${\it S}_{2}$ at some points $\mbox{B}$ and $\mbox{D}$, respectively.
We then draw a $C_{-}$ and a $C_{+}$ cross characteristic curves from $\mbox{B}$ and $\mbox{D}$, respectively.
The $C_{-}$ and $C_{+}$ characteristic curve will cross the whole fan waves ${\it F}_{1l}$ and ${\it F}_{2l}$ and end up at some points $\mbox{E}=(\xi_{_\mathrm{E}}, \eta_{_\mathrm{E}})$ and $\mbox{F}=(\xi_{_\mathrm{F}}, \eta_{_\mathrm{F}})$, respectivly, where $\xi_{_\mathrm{E}}=\xi_{_\mathrm{F}}=\hat{u}_l(\hat{\xi}_2)\sin\theta$ and $\eta_{_\mathrm{E}}=-\eta_{_\mathrm{F}}=-\hat{u}_l(\hat{\xi}_2)\cos\theta$.
The existence of the cross characteristic curves $\wideparen{\mbox{PB}}$, $\wideparen{\mbox{PD}}$, $\wideparen{\mbox{BE}}$ and $\wideparen{\mbox{DF}}$ will be given in Section 4.1. Let $\Sigma$ be a domain bounded by the $C_{+}$ characteristic curve $\wideparen{\mbox{PB}}\cup\wideparen{\mbox{BE}}$,
 the $C_{-}$ characteristic curve $\wideparen{\mbox{PD}}\cup\wideparen{\mbox{DF}}$, and a vacuum boundary $\wideparen{\mbox{EF}}$ connecting $\mathrm{E}$ and $\mathrm{F}$. The vacuum boundary $\wideparen{\mbox{EF}}$ is yet to be determined. The domain $\Sigma$ is actually the interaction region of the composite waves ${\it FSF}_1$ and  ${\it FSF}_2$.
The flow outside the domain $\Sigma$ is know. In order to determine the flow inside the interaction region $\Sigma$, we consider (\ref{42501}) with the following boundary conditions:
\begin{equation}\label{426011}
(u, v, \tau)=\left\{
            \begin{array}{ll}
               \big(\sin\theta \hat{u}_r, -\cos\theta \hat{u}_r, \hat{\tau}_r\big)(\xi\sin\theta-\eta\cos\theta), & \hbox{$(\xi, \eta)\in\wideparen{\mbox{PB}}$;} \\[4pt]
 \big(\sin\theta \hat{u}_r, \cos\theta\hat{u}_r, \hat{\tau}_r\big)(\xi\sin\theta+\eta\cos\theta), & \hbox{$(\xi, \eta)\in\wideparen{\mbox{PD}}$;}  \\[4pt]
             \big(\sin\theta\hat{u}_l, -\cos\theta\hat{u}_l, \hat{\tau}_l\big)(\xi\sin\theta-\eta\cos\theta), & \hbox{$(\xi, \eta)\in\wideparen{\mbox{BE}}$;} \\[4pt]
             \big(\sin\theta\hat{u}_l, \cos\theta \hat{u}_l, \hat{\tau}_l\big)(\xi\sin\theta+\eta\cos\theta), & \hbox{$(\xi, \eta)\in\wideparen{\mbox{DF}}$.}
            \end{array}
          \right.
\end{equation}
Problem (\ref{42501}, \ref{426011}) is a discontinuous characteristic boundary value problem with ``large" boundary data.

\subsection{Main theorem}
We construct a global piecewise smooth solution in $\Sigma$ to the boundary value problem (\ref{42501}, \ref{426011}).
Before stating the main theorem of the paper, we define the following constants:
\begin{equation}\label{62903}
\qquad\boxed{
\begin{aligned}
&c_1^e:=\tau_1^e\sqrt{-p'(\tau_1^e)},\quad c_2^e:=\tau_2^e\sqrt{-p'(\tau_2^e)}, \quad u_{*}:=c_1^e\Big(1-\frac{\tau_2^e}{\tau_1^e}\Big)=c_1^e-c_2^e<0,\\&
\sigma_{*}:=\arctan\left(\frac{u_*\cos\theta}{u_*\sin\theta-c_1^e\csc\theta}\right),\quad
q_{*}:=\sqrt{\big(u_*\sin\theta-c_1^e\csc\theta\big)^2+\big(u_*\cos\theta\big)^2},~~\\
&q_{_{lim}}:=\bigg(q_*^2-2\int_{\tau_2^e}^{+\infty}\tau p'(\tau){\rm d}\tau\bigg)^{\frac{1}{2}},\quad\mbox{and}\quad
\sigma_{\infty}:=\int_{q_*}^{q_{_{lim}}}\frac{\sqrt{q^2-c^2(\bar{\tau}(q))}}{qc(\bar{\tau}(q))}~{\rm d}q.
\end{aligned}}
\end{equation}
In (\ref{62903}), the function $\bar{\tau}(q)$ ($q>q_*$) is determined by
$$
\frac{q^2}{2}+\int_{\tau_2^e}^{\bar{\tau}(q)}\tau p'(\tau){\rm d}\tau =\frac{q_*^2}{2},
$$
and $c(\bar{\tau}(q))=\bar{\tau}(q)\sqrt{-p'(\bar{\tau}(q))}$.

If $\sigma_*<\sigma_{\infty}$, there exists a $q_{d}>q_*$ such that
$$
\int_{q_*}^{q_{d}}\frac{\sqrt{q^2-c^2(\bar{\tau}(q))}}{qc(\bar{\tau}(q))}{\rm d}q=\sigma_*.
$$
We then define
\begin{equation}\label{101801}
A_{d}:=\arcsin \left(\frac{c(\bar{\tau}(q_{d}))}{q_{d}}\right).
\end{equation}

We also
define the following variables:
\begin{equation}\label{111901}
\begin{aligned}
&\quad \kappa(\tau):=-\frac{2p'(\tau)}{2p'(\tau)+\tau p''(\tau)}, \quad \mu(\tau):=\frac{2p'(\tau)+\tau p''(\tau)}{\tau p''(\tau)}, \\&
\varpi(\tau):=-\frac{4p'(\tau)+\tau p''(\tau)}{\tau p''(\tau)},\quad \mbox{and}\quad \bar{A}(\tau):=\arctan\sqrt{\varpi(\tau)}.
\end{aligned}
\end{equation}

\begin{rem}
For polytropic ideal gases $p=\rho^{\gamma}$ ($1<\gamma<3$), one has $\kappa\equiv\frac{2}{\gamma-1}$, $\mu\equiv\frac{\gamma-1}{\gamma+1}$, and $\varpi\equiv\frac{3-\gamma}{\gamma+1}$.
\end{rem}

In order to obtain a global solution to the boundary value problem (\ref{42501}, \ref{426011}),
we make the following assumptions:
\begin{description}
  \item[(A1)] $\varpi(\tau)>0$ and $\varpi'(\tau)\leq 0$ for $\tau\geq\tau_2^e$;
  \vskip 2pt
  \item[(A2)] $\theta+2\sigma_*<\frac{\pi}{2}$;
  \vskip 2pt
  \item[(A3)] $\frac{\pi}{2}+\min\{A_{d}, \theta\}>2\bar{A}(\tau_2^e)$.
\end{description}
The main result of the paper is stated as the following theorem.
\begin{thm}\label{thm2}
({\bf Main theorem})
Suppose that assumptions ($\mathbf{A}1$)--($\mathbf{A}3$) hold. Then when $\tau_1^e-\tau_0$ is sufficiently small, the boundary value problem (\ref{42501}, \ref{426011}) admits a piecewise smooth  solution in  $\Sigma$. The vacuum boundary  $\wideparen{\mathrm{EF}}$ can be represented by a Lipschitz continuous function $\xi=\mathcal{V}(\eta)$, $\eta\in[\eta_{_\mathrm{F}}, \eta_{_\mathrm{E}}]$ which satisfies
$$
\sup \limits _{\eta_1\neq\eta_2}\frac{\big|\mathcal{V}(\eta_1)-\mathcal{V}(\eta_2)\big|}{\big|\eta_1-\eta_1\big|}~
\leq~\sqrt{\frac{\gamma+1}{3-\gamma}}.
$$
Moreover, the solution satisfies
$$
\rho>0 \quad \mbox{in}\quad \Sigma\setminus \wideparen{\mathrm{EF}} \quad \mbox{and}\quad   \rho=0\quad \mbox{on}\quad \wideparen{\mathrm{EF}}.
$$
\end{thm}

The conditions ($\mathbf{A}1$)--($\mathbf{A}3$) are sufficient but not necessary for the existence of a global solution.
We shall show that assumptions ($\mathbf{A}1$)--($\mathbf{A}3$) can be satisfied in some cases.
By the definitions of $q_*$ and $\sigma_*$ we know
\begin{equation}\label{62904}
\sigma_*=\arcsin\Big(\frac{c_2^e}{q_*}\Big)-\theta.
\end{equation}
This can be also illustrated geometrically in Figure \ref{Fig6}.
By the definitions of $\sigma_*$ and $A_{d}$, we have
\begin{equation}\label{81801}
\sigma_*\rightarrow 0\quad \mbox{and}\quad A_{d}\rightarrow \frac{\pi}{2}\quad \mbox{as}\quad \theta\rightarrow \frac{\pi}{2}.
\end{equation}
So,  when $\theta$ is close to $\frac{\pi}{2}$, there hold $\sigma_*<\sigma_{\infty}$ and $\frac{\pi}{2}+\min\{A_{d}, \theta\}>2\bar{A}(\tau_2^e)$.

\begin{figure}[htbp]
\begin{center}
\qquad\quad\includegraphics[scale=0.6]{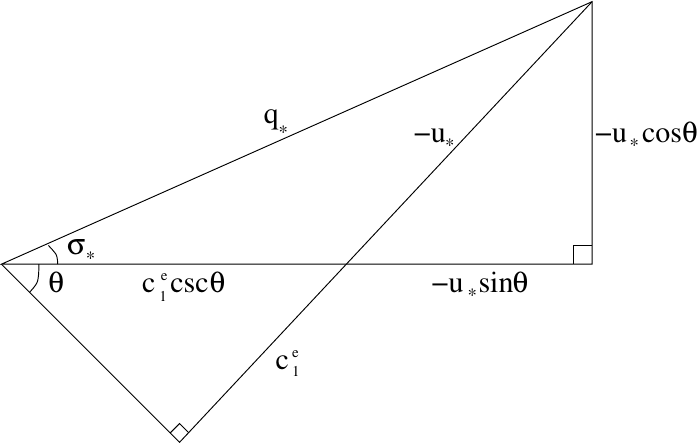}
\caption{\footnotesize A geometrical explanation for $\theta+\sigma_*=\arcsin\Big(\frac{c_2^e}{q_*}\Big)$.}
\label{Fig6}
\end{center}
\end{figure}

By a direct computation one has
$$
\frac{{\rm d}(\theta+2\sigma_*)}{{\rm d}\theta}\bigg|_{\theta=\frac{\pi}{2}}=~2\frac{\tau_1^e}{\tau_2^e}-1.
$$
Combining this with (\ref{81801}) we obtain that if $\frac{\tau_1^e}{\tau_2^e}>\frac{1}{2}$ and $\theta$ is close to $\frac{\pi}{2}$ then $\theta+2\sigma_*<\frac{\pi}{2}$.

In the appendix we shall show that for the polytropic van der Waals gas (\ref{van}),  when $s\in (\frac{8}{27}, \frac{81}{256})$ and $\gamma$ is close to $1$, assumption ($\mathbf{A}$1) holds and $\frac{\tau_1^e}{\tau_2^e}>\frac{1}{2}$; see Proposition \ref{92001}.
For a nonconvex equation of state with two inflection points, if  $p=\rho^\gamma$ for $\tau>\tau_2^e$ then
the assumption ($\mathbf{A}$1) is also hold.

\subsection{Overview}
One main difficulty for the global existence is that the system (\ref{42501}) is a mixed type system and the type in each point is determined by the state of the flow at that point.
Elling and Liu \cite{ELiu1} established an ellipticity principle of the system (\ref{42501}). This principle guarantees that if $q/c\leq 1$ on the boundary of a bounded region
then the equation is uniformly elliptic in any interior area of this region, and is important in constructing solutions of elliptic boundary value problems for (\ref{42501}). However, one does not have a general hyperbolicity principle for hyperbolic boundary value problems for (\ref{42501}). Actually, sonic curve may appear as a part of the boundary of the  rarefaction wave interaction region in some cases; see Zheng \cite{Zheng3}.
The other difficulty is that there are shocks in the interaction region and the type (double-sonic, pre-sonic, post-sonic, or transonic) of the shocks is also a priori unknown.  This results in the the fact that the formulation of the boundary conditions on the shocks is also a priori unknown.
Let us briefly describe the main processes to construct a solution to the boundary value problem  (\ref{42501}, \ref{426011}).
\vskip 2pt

\noindent
{\it Step 1.} We first show that the forward $C_{+}$ and $C_{-}$ characteristics issued from $\mathrm{P}$ can cross the whole fan-shock-fan composite waves  ${\it FSF}_1$ and ${\it FSF}_2$ and end up at the points $\mathrm{F}$ and $\mathrm{E}$, respectively; see Figure \ref{Fig8}.
\vskip 2pt

\noindent
{\it Step 2.}
 We solve a Goursat  problem for (\ref{42501}) with $\wideparen{\mbox{PB}}$ and $\wideparen{\mbox{PD}}$
as the characteristic boundaries in a curved quadrilateral domain $\Sigma^1$ closed by $\wideparen{\mbox{PB}}$, $\wideparen{\mbox{PD}}$, a forward $C_{-}$ characteristic curve issued from $\mbox{B}$, and a forward $C_{+}$ characteristic curve issued from $\mbox{D}$; see Figure \ref{Fig13} (left).

\begin{figure}[htbp]
\begin{center}
\includegraphics[scale=0.38]{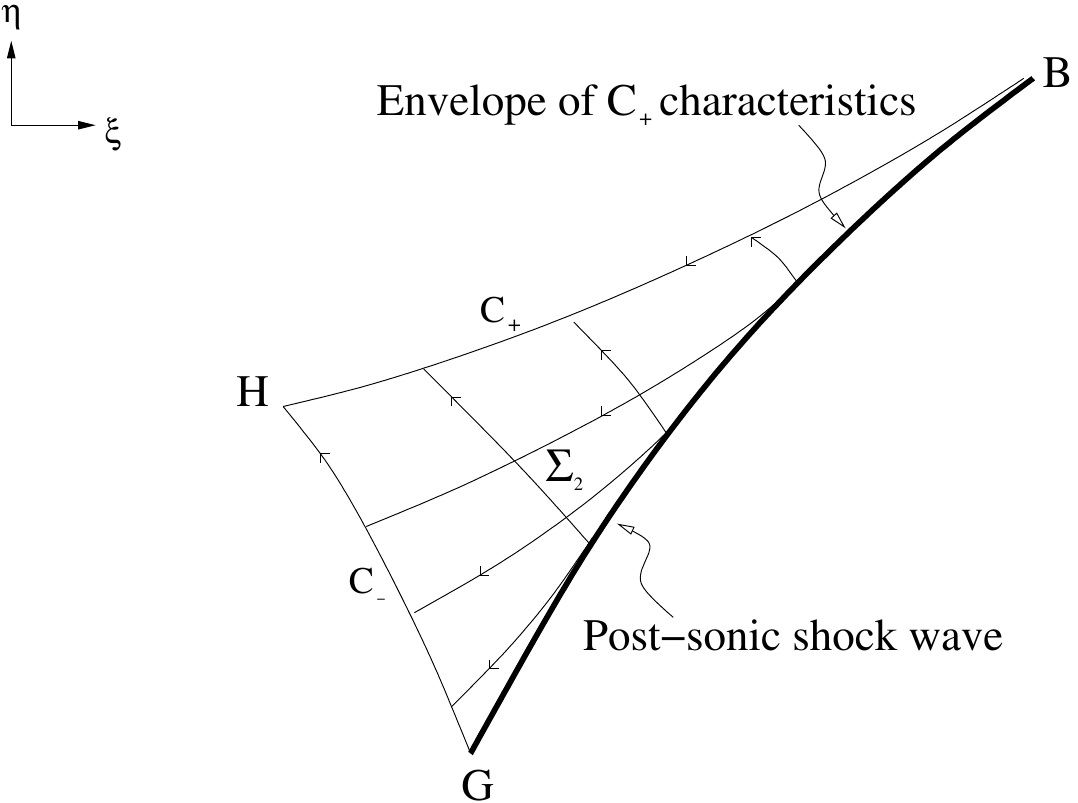}
\caption{\footnotesize  A post-sonic shock wave and the $C_{\pm}$ characteristic curves behind it.}
\label{Fig9}
\end{center}
\end{figure}

\noindent
{\it Step 3.}
The double-sonic shocks ${\it S}_1$ and ${\it S}_2$ will go into the interaction region from $\mathrm{B}$ and $\mathrm{D}$, respectively.
By calculating the curvatures of the shocks  and  using the Liu's extended entropy condition,
we find that the shocks must be post-sonic
and are always stay in $\Sigma^1$ until they intersect with each other
at some point $\mathrm{G}$ in $\Sigma^1$; see Figure \ref{Fig13} (left).
We denote by $\wideparen{\mbox{BG}}$ and $\wideparen{\mbox{DG}}$ the
post-sonic shocks from $\mbox{B}$ and $\mbox{D}$, respectively.
The post-sonic shocks $\wideparen{\mbox{BG}}$ and $\wideparen{\mbox{DG}}$ and the backside states of them can then be determined by solving an initial value problem for an ordinary differential equations derived from the Rankine-Hugoniot relations.

\noindent
{\it Step 4.}
Since the shock $\wideparen{\mbox{BG}}$ is post-sonic,
the shock curve $\wideparen{\mathrm{BG}}$ and the $C_{+}$ characteristic have the same direction at each point on $\wideparen{\mathrm{BG}}$.
 However, by calculating the direction derivatives of the backside states in the shock direction we find that the corresponding characteristic equation do not hold along them.
This implies that $\wideparen{\mathrm{BG}}$ is not characteristic of the flow behind it.
Moreover, the flow behind the shock $\wideparen{\mathrm{BG}}$ is not $C^1$ smooth up to $\wideparen{\mathrm{BG}}$.
 In order to overcome the difficulty cased by the singularity, we use the hodograph transformation method to find a solution near the backside of them.
  We solve a Cauchy problem for a linearly degenerate hyperbolic system in the hodograph plane and
 prove  that the hodograph transformation is one-to-one.
By doing this, we obtain a solution in a triangle region $\Sigma_2$ bounded by $\wideparen{\mbox{BG}}$, $\wideparen{\mbox{BH}}$, and $\wideparen{\mbox{GH}}$, where $\wideparen{\mbox{BH}}$ is a forward $C_{+}$ characteristic issued from $\mbox{B}$ and $\wideparen{\mbox{GH}}$ is a forward $C_{-}$ characteristic issued from $\mbox{G}$; see Figure \ref{Fig9}. We find that the shock curve $\wideparen{\mbox{BG}}$ is an envelope of the $C_{+}$ characteristic curves in $\Sigma_2$, and the directional derivatives of the unknown functions in the $C_{-}$ characteristic direction are infinity at each point on $\wideparen{\mbox{BG}}$. By symmetry, we obtain the flow in a triangle region $\Sigma_3$ bounded by $\wideparen{\mbox{DG}}$, $\wideparen{\mbox{GI}}$, and $\wideparen{\mbox{DI}}$, where $\wideparen{\mbox{GI}}$  a forward $C_{+}$ characteristic curve issued from $\mbox{G}$ and $\wideparen{\mbox{DI}}$  a forward $C_{-}$ characteristic curve issued from $\mbox{D}$; see Figure \ref{Fig10}.

\begin{figure}[htbp]
\begin{center}
\includegraphics[scale=0.32]{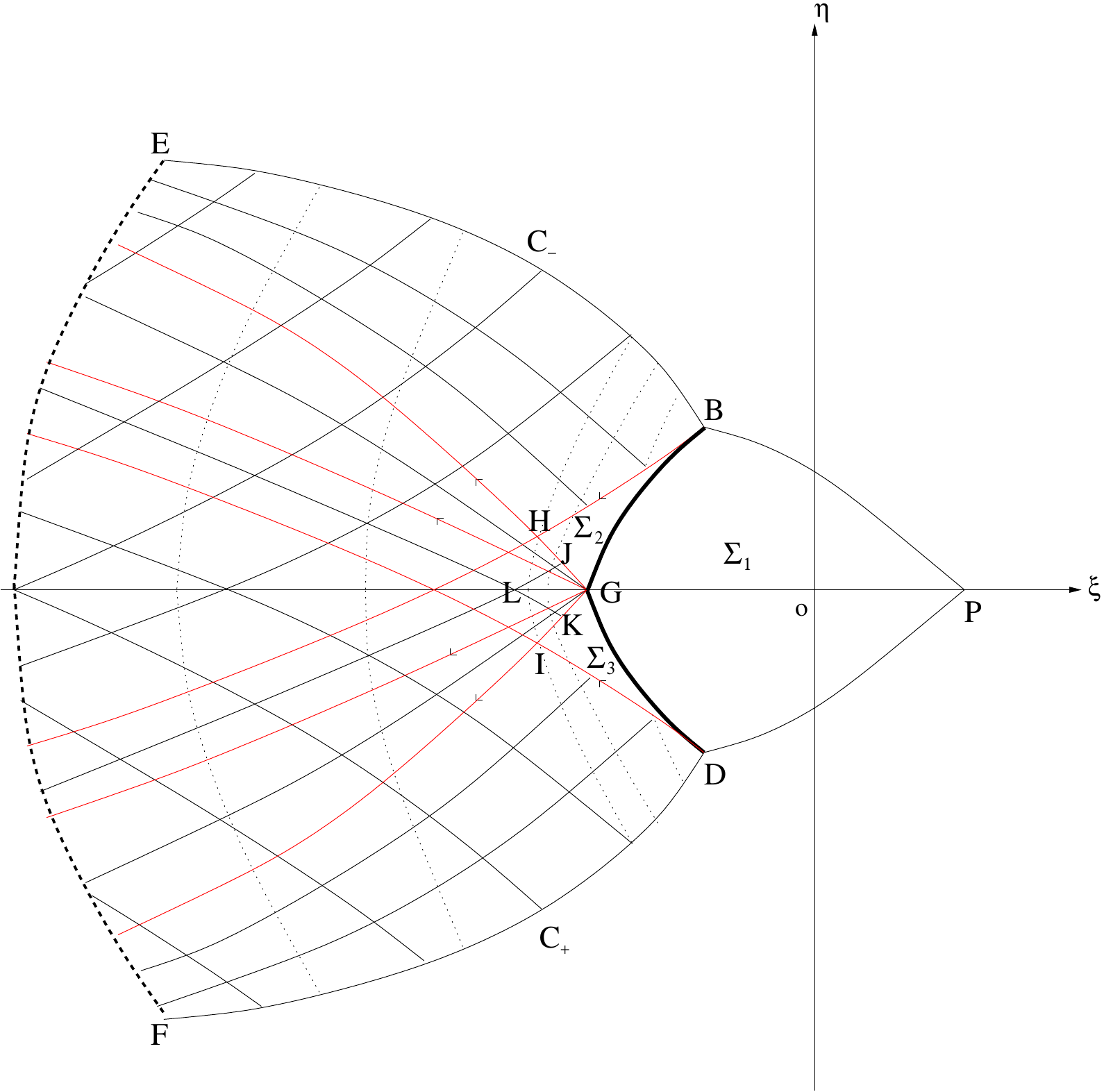}
\caption{\footnotesize Fan-shock-fan composite wave interaction region. Here, the solid lines represent characteristic lines; the thick lines represent shock waves; the red lines represent  weak discontinuity curves; the dotted lines represent the level curves of $\tau$. }
\label{Fig10}
\end{center}
\end{figure}

\noindent
{\it Step 5.}
The post-sonic shocks  $\wideparen{\mbox{BG}}$ and $\wideparen{\mbox{DG}}$  interact with each other at the point $\mbox{G}$. By the classical phase plane analysis method, we find that this shock wave interaction generates two centered waves issued from the point $\mbox{G}$.
So, in order solve the shock wave interaction, we consider a characteristic boundary value problem for (\ref{42501}) with $\wideparen{\mbox{GH}}$ and $\wideparen{\mbox{GI}}$ as the characteristic boundaries.
We still use the hodograph transformation to find a local classical solution with two centered waves to this characteristic boundary value problem.
 The result in this step can be stated as follows. Let $\mbox{J}$ and $\mbox{K}$ be two points on $\wideparen{\mbox{GH}}$ and $\wideparen{\mbox{GI}}$, respectively.
 When the points $\mbox{J}$ and $\mbox{K}$ are sufficiently close to $\mbox{G}$, this characteristic boundary value problem admits a classical solution in a curved quadrilateral domain closed by characteristic curves $\wideparen{\mbox{GJ}}$, $\wideparen{\mbox{GK}}$, $\wideparen{\mbox{JL}}$ and $\wideparen{\mbox{KL}}$, where $\wideparen{\mbox{JL}}$ is a forward $C_{+}$ characteristic curve from $\mbox{J}$ and $\wideparen{\mbox{KL}}$ is a forward $C_{-}$ characteristic curve from $\mbox{K}$; see Figure \ref{Fig10}.

\noindent
{\it Step 6.}
In order to obtain a solution in the remaining part of the domain $\Sigma$, we consider a characteristic boundary value problem for (\ref{42501}) with $\wideparen{\mbox{BE}}$, $\wideparen{\mbox{BH}}$, $\wideparen{\mbox{JH}}$, $\wideparen{\mbox{JL}}$, $\wideparen{\mbox{KL}}$, $\wideparen{\mbox{KI}}$,
$\wideparen{\mbox{DI}}$, and $\wideparen{\mbox{DF}}$ as the characteristic boundaries.
We apply the method of characteristics to obtain a global classical solution in a ``determinate" region bounded by these characteristic boundaries and a vacuum boundary $\wideparen{\mathrm{EF}}$ connecting $\mbox{E}$ and $\mbox{F}$.
The main difficulty for the global existence is to establish the hyperbolicity of (\ref{42501}) and the a priori estimate for the $C^1$ norm of the solution.
We establish a ``maximum principle" to prove that the system (\ref{42501}) is uniformly hyperbolic in the determinate region and use the characteristic decomposition method to establish some uniform  a priori estimates for the derivatives of the solution.
 The reason to make assumptions ($\mathbf{A}1$)--($\mathbf{A}3$) is to control the hyperbolicity of the system (\ref{42501}) in $\Sigma$.

\vskip 2pt
The rest of the paper is organized as follows.

In Section 2 we present the characteristics of the system (\ref{42501}).
The concepts of the characteristic directions and the characteristic angles $\alpha$ and $\beta$
are given in Section 2.1. In Section 2.1 we also derive a group of characteristic equations for the variables $\alpha$, $\beta$, and $c$.
These characteristic equations will be widely used to establish the a priori estimate for the hyperbolicity of the system (\ref{42501}).
In Section 2.2 we derive several groups of characteristic decompositions for the variables $\rho$ and $c$. These characteristic decompositions will be widely used to control the bounds of the derivatives of the solution.

In Section 3 we present the hodograph transformation for the system (\ref{42501}).
 The concept of the hodograph transformation is presented in Section 3.1.
 The systems in the hodograph plane are derived in Section 3.2. We derive a group of characteristic decompositions for the system (\ref{42501}) in the hodograph plane in Section 3.3. These characteristic decompositions will be used to establish an important monotonicity of the hodograph transformation. This monotonicity will deduce
 that the hodograph transformation is one-to-one.

In Section 4 we solve the interaction of the fan-shock-fan composite waves. In Section 4.1 we construct the characteristic boundaries of the interaction region. In Section 4.2 we solve the interaction of the fan waves ${\it F}_{1r}$ and ${\it F}_{2r}$. In Section 4.3 we construct the post-sonic shocks $\wideparen{\mbox{BG}}$ and $\wideparen{\mbox{DG}}$ and the states on the backside of them.  In Section 4.4 we construct the solution near the backside of the post-sonic shocks. In Section 4.5 we find a local solution to the interaction of the post-sonic shocks. In Section 4.6 we obtain the solution in the remaining part of the fan-shock-fan composite wave interaction region.

\vskip 8pt
\section{Characteristic and characteristic decomposition}
\subsection{Characteristic directions and characteristic angles}
In what follows, we shall confine ourselves to supersonic flow. The direction of the wave
characteristics is defined as the tangent direction that forms an
acute angle $A$ with the pseudoflow velocity
$(U, V)$. By computation, we have
\begin{equation}
c^{2}=q^{2}\sin^{2}A,\label{210cqo}
\end{equation}
where $q^{2}=U^{2}+V^{2}$. The angle $A$ is called the
 pseudo-Mach angle. The quantity
$
q/c
$
is called the pseudo-Mach number of the flow. For pseudosonic flow,
for which $q/c=1$, we have that the direction of the
pseudoflow velocity
 is perpendicular to the wave characteristic directions.

By a direct computation, we have that
the $C_+$
characteristic direction forms with the pseudoflow velocity
$(U, V)$ the angle $A$ from $C_{+}$ to  $(U, V)$ in
the clockwise direction, and the $C_-$ characteristic direction forms with the
pseudoflow direction the angle $A$ from $C_{-}$ to $(U, V)$
in the counterclockwise direction; see Figure \ref{Fig11}.


From (\ref{21}) we have
\begin{equation}\label{c}
c=\frac{|(U,V)\cdot(\lambda,-1)|}{|(\lambda,-1)|},
\end{equation}
which implies that the component of the pseudoflow velocity normal to the
direction of a characteristic $C_{+}$ ($C_{-}$) is equal to the sound speed.
 Equivalently, it can be stated as that the tangent line of a $C_{+}$ ($C_{-}$)
characteristic at a point is tangent to the sonic circle of the
state at that point. Furthermore, a $C_{+}$ ($C_{-}$) characteristic must be straight
if $(u, v, c)$ is constant along it.

 The $C_{+}$ ($C_{-}$) characteristic angle is defined as the counterclockwise angle from the positive $\xi$-axis to  the $C_{+}$ ($C_{-}$) characteristic direction.
We denote by $\alpha$ and
 $\beta$ the counterclockwise angle from the positive $\xi$-axis to $C_{+}$ and $C_{-}$ characteristic directions, respectively, i.e.,
\begin{equation}\label{6803}
\lambda_{+}=\tan\alpha\quad \mbox{and}\quad \lambda_{-}=\tan\beta.
\end{equation}

\begin{figure}[htbp]
\begin{center}
\includegraphics[scale=0.6]{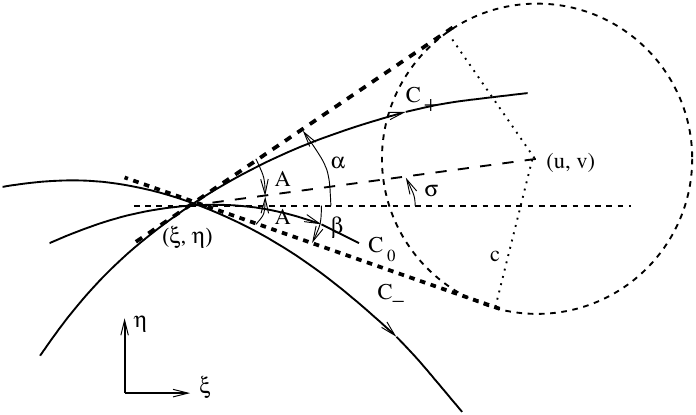}\qquad\qquad\includegraphics[scale=0.42]{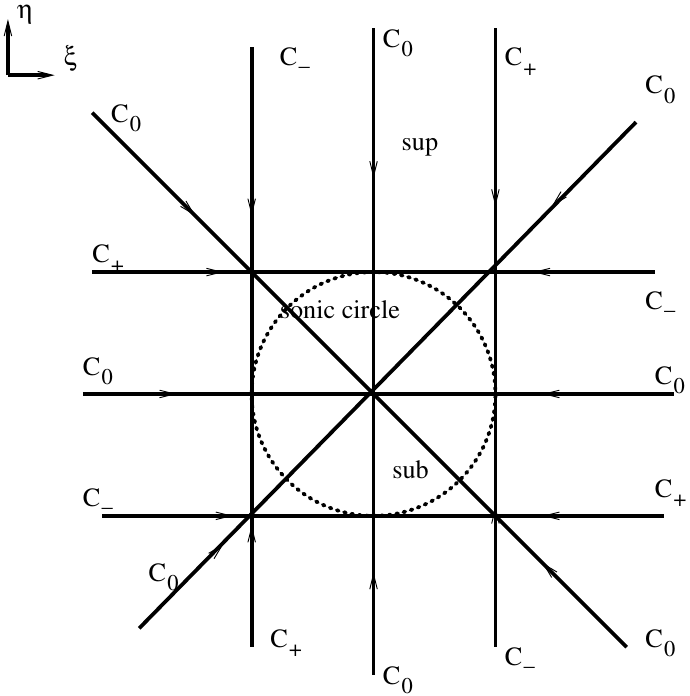}
\caption{ \footnotesize Left: characteristic directions and characteristic angles; right: $C_{\pm}$ and $C_{0}$ characteristic lines for  a constant state.}
\label{Fig11}
\end{center}
\end{figure}

 Let $\sigma$ be the counterclockwise  angle from the positive $\xi$-axis to the pseudoflow velocity.
Then we have
\begin{equation}
\alpha=\sigma+A,\quad \beta=\sigma-A, \quad\sigma=\frac{\alpha+\beta}{2},\quad A=\frac{\alpha-\beta}{2}. \label{tau}
\end{equation}
Therefore, the relations between $(U, V, c)$ and $(\alpha, \beta, c)$ are
\begin{equation} \label{U}
U=c\frac{\cos\sigma}{\sin A}=q\cos\sigma,\quad V=c\frac{\sin\sigma}{\sin A}=q\sin\sigma.
\end{equation}
The mapping $(U, V, c)\rightarrow (\alpha, \beta, c)$ is bijective as long as system (\ref{42501}) is hyperbolic.


Multiplying (\ref{matrix1}) on the left by
$(1,\mp c\sqrt{U^{2}+V^{2}-c^{2}})$, we get the characteristic equations of the system (\ref{42501}):
\begin{equation}
\left\{
  \begin{array}{ll}
  \displaystyle \bar{\partial}_{+}u+\lambda_{-}\bar{\partial}_{+}v
  =0,  \\[10pt]
    \displaystyle  \bar{\partial}_{-}u+\lambda_{+}\bar{\partial}_{-}v=0,
  \end{array}
\right.\label{form}
\end{equation}
where
\begin{equation}\label{3303}
\bar{\partial}_{+}:=\cos\alpha\partial_{\xi}+\sin\alpha\partial_{\eta}\quad \mbox{and}\quad \bar{\partial}_{-}:=\cos\beta\partial_{\xi}+\sin\beta\partial_{\eta}.
\end{equation}

\subsection{Characteristic equations and characteristic decompositions}

From (\ref{U}) we have
\begin{equation}
\bar{\partial}_{\pm}u=\cos(\sigma\pm A)+\frac{\cos\sigma}{\sin A}\bar{\partial}_{\pm}c+\frac{c\cos\alpha\bar{\partial}_{\pm}\beta
-c\cos\beta\bar{\partial}_{\pm}\alpha}{2\sin^{2}A},\label{1}
\end{equation}
\begin{equation}
\bar{\partial}_{\pm}v=\sin(\sigma\pm A)+\frac{\sin\sigma}{\sin A}\bar{\partial}_{\pm}c
+\frac{c\sin\alpha\bar{\partial}_{\pm}\beta
-c\sin\beta\bar{\partial}_{\pm}\alpha}{2\sin^{2}A}.\label{2}
\end{equation}
We insert (\ref{1}) and (\ref{2}) into (\ref{form}) to  obtain
\begin{equation}\label{3}
\bar{\partial}_{+}c=-\frac{\cos2A}{\cot A}+\frac{c}{\sin2A}
(\bar{\partial}_{+}\alpha-\cos2A\bar{\partial}_{+}\beta),
\end{equation}
\begin{equation}
\bar{\partial}_{-}c=-\frac{\cos2A}{\cot A}+\frac{c}{\sin2A}
(\cos2A\bar{\partial}_{-}\alpha-\bar{\partial}_{-}\beta).\label{4}
\end{equation}

Differentiating the pseudo-Bernoulli law (\ref{5802}) in the $C_{\pm}$ characteristic directions and
using (\ref{1}) and (\ref{2}), we obtain
\begin{equation}
\left(\frac{1}{\sin^{2}A}+\kappa\right)\bar{\partial}_{+}c=
\frac{c\cos A}{2\sin^{3}A}(\bar{\partial}_{+}\alpha-\bar{\partial}_{+}\beta)-\cot A\label{bB1}
\end{equation}
and
\begin{equation}
\left(\frac{1}{\sin^{2}A}+\kappa\right)\bar{\partial}_{-}c=
\frac{c\cos A}{2\sin^{3}A}(\bar{\partial}_{-}\alpha-\bar{\partial}_{-}\beta)-\cot A,\label{bB2}
\end{equation}
where the variable $\kappa=\kappa(\tau)$ is defined in (\ref{111901}).

Inserting (\ref{bB1}) into (\ref{3}), we obtain
\begin{equation}\label{6}
c\bar{\partial}_{+}\alpha=\Omega\cos^{2}A(c\bar{\partial}_{+}\beta+2\sin^{2}A),
\end{equation}
where
$$
\Omega:=~\varpi(\tau)-\tan^2A
$$
and the variable $\varpi(\tau)$ is defined in (\ref{111901}).
Similarly, we insert (\ref{bB2}) into (\ref{4}) to obtain
\begin{equation}\label{5}
c\bar{\partial}_{-}\beta=\Omega\cos^{2}A(c\bar{\partial}_{-}\alpha-2\sin^{2}A).
\end{equation}

Combining (\ref{3}) with (\ref{6}), we have
\begin{equation}
c\bar{\partial}_{+}\beta=-\frac{\tan A}{\mu}\bar{\partial}_{+}c-2\sin^{2} A=-\frac{p''(\tau)}{2c\rho^4}\tan A\bar{\partial}_{+}\rho-2\sin^{2} A,\label{7}
\end{equation}
\begin{equation}
c\bar{\partial}_{+}\alpha=-\left(\frac{1+\kappa}{2}\right)\Omega\sin2 A\bar{\partial}_{+}c=-\frac{p''(\tau)}{4c\rho^4}\Omega\sin2 A\bar{\partial}_{+}\rho.\label{10}
\end{equation}
Combining (\ref{4}) with (\ref{5}), we have
\begin{equation}\label{82301}
c\bar{\partial}_{-}\alpha=\frac{\tan A}{\mu}\bar{\partial}_{-}c+2\sin^{2} A=\frac{p''(\tau)}{2c\rho^4}\tan A\bar{\partial}_{-}\rho+2\sin^{2} A,
\end{equation}
\begin{equation}
c\bar{\partial}_{-}\beta=\left(\frac{1+\kappa}{2}\right) \Omega\sin2 A\bar{\partial}_{-}c=\frac{p''(\tau)}{4c\rho^4}\Omega\sin2 A\bar{\partial}_{-}\rho.\label{8}
\end{equation}


From (\ref{1}), (\ref{2}), and (\ref{7})--(\ref{8}) we have
\begin{equation}\label{11}
\bar{\partial}_{+}u=\kappa\sin\beta\bar{\partial}_{+}c=c\tau\sin\beta\bar{\partial}_{+}\rho,\quad\bar{\partial}_{-}u=-\kappa\sin\alpha\bar{\partial}_{-}c=-c\tau\sin\alpha\bar{\partial}_{-}\rho,
\end{equation}
\begin{equation}\label{72804}
\bar{\partial}_{+}v=-\kappa\cos\beta\bar{\partial}_{+}c=-c\tau\cos\beta\bar{\partial}_{+}\rho,
\quad\bar{\partial}_{-}v=\kappa\cos\alpha\bar{\partial}_{-}c=c\tau\cos\alpha\bar{\partial}_{-}\rho.
\end{equation}

\begin{prop}
(Commutator relation)  We have
\begin{equation}
\bar{\partial}_{-} \bar{\partial}_{+}- \bar{\partial}_{+} \bar{\partial}_{-}=
\frac{1}{\sin2 A}\Big[(\cos2 A \bar{\partial}_{+}\beta- \bar{\partial}_{-}\alpha) \bar{\partial}_{-}-
(\bar{\partial}_{+}\beta-\cos2 A \bar{\partial}_{-}\alpha) \bar{\partial}_{+}\Big].
\label{comm}
\end{equation}
\end{prop}
\begin{proof}
From (\ref{3303}) we have
\begin{equation}\label{61513}
\partial_{\xi}=-\frac{\sin\beta\bar{\partial}_{+}-\sin\alpha\bar{\partial}_{-}}{\sin2A}\quad \mbox{and}\quad
\partial_{\eta}=\frac{\cos\beta\bar{\partial}_{+}-\cos\alpha\bar{\partial}_{-}}{\sin2A}.
\end{equation}
Then the commutator relation can be obtained by a direct computation. This commutator relation was first derived by Li, Zhang and Zheng in \cite{Li-Zhang-Zheng}.
\end{proof}

\begin{prop}\label{51501}
For the variable $c$, we have the following characteristic decompositions:
\begin{equation}\label{cd}
\left\{
  \begin{array}{ll}
    \displaystyle c\bar{\partial}_{+}\bar{\partial}_{-}c=\bar{\partial}_{-}c\left(\sin2 A+\frac{ \bar{\partial}_{-}c}{2\mu\cos^2 A}
+
\Big(\frac{1-\frac{\kappa\sin^2 2A}{\kappa+1}}{2\mu\cos^2 A}-\frac{c}{\kappa}\cdot\frac{{\rm d}\kappa}{{\rm d}c}\Big)\bar{\partial}_{+}c\right), \\[18pt]
   \displaystyle c\bar{\partial}_{-}\bar{\partial}_{+}c=\bar{\partial}_{+}c\left(\sin2 A+\frac{ \bar{\partial}_{+}c}{2\mu\cos^2 A}+
\Big(\frac{1-\frac{\kappa\sin^2 2A}{\kappa+1}}{2\mu\cos^2 A}-\frac{c}{\kappa}\cdot\frac{{\rm d}\kappa}{{\rm d}c}\Big)\bar{\partial}_{-}c\right),
  \end{array}
\right.
\end{equation}
where  the variable $\mu=\mu(\tau)$ is defined in (\ref{111901}).
\end{prop}

\begin{proof}
The characteristic decompositions (\ref{cd}) was first derived in Lai \cite{Lai1}. For the sake of completeness we sketch the proof.
It follows from (\ref{11}) and (\ref{comm}) that
\begin{equation}
\begin{array}{rcl}
&&\bar{\partial}_{-} \bar{\partial}_{+}u- \bar{\partial}_{+} \bar{\partial}_{-}u\\[6pt]&=&
\bar{\partial}_{+}\big(\kappa\sin\alpha\bar{\partial}_{-}c\big)
\bar{\partial}_{-}\big(\kappa\sin\beta\bar{\partial}_{+}c\big)\\[4pt]
&=&-\displaystyle\frac{1}{\sin2 A}\Big((\bar{\partial}_{+}\beta-\cos2 A \bar{\partial}_{-}\alpha)\kappa\sin\beta \bar{\partial}_{+}c-( \bar{\partial}_{-}\alpha- \cos2 A\bar{\partial}_{+}\beta)\kappa\sin\alpha \bar{\partial}_{-}c\Big).
\end{array}
\end{equation}
Hence we have
\begin{equation}
\begin{array}{lll}
\displaystyle&&\sin\alpha\bar{\partial}_{+}\bar{\partial}_{-}c+
\sin\beta\bar{\partial}_{-}\bar{\partial}_{+}c+\displaystyle\frac{1}{\kappa}\cdot\frac{{\rm d}\kappa}{{\rm d}c}(\sin\alpha+\sin\beta)\bar{\partial}_{+}c\bar{\partial}_{-}c
\\[4pt]&=&-\displaystyle\frac{1}{\sin2 A}
\Big((\sin\beta \bar{\partial}_{+}\beta-\sin\beta \cos2 A \bar{\partial}_{-}\alpha+\cos\beta\sin2 A\bar{\partial}_{-}\beta) \bar{\partial}_{+}c\\[8pt]&&\qquad\displaystyle-( \sin\alpha\bar{\partial}_{-}\alpha- \sin\alpha\cos2 A\bar{\partial}_{+}\beta-\cos\alpha\sin2 A\bar{\partial}_{+}\alpha) \bar{\partial}_{-}c\Big).
\end{array}\label{12a}
\end{equation}

Applying the commutator relation (\ref{comm}) for $c$, we get
\begin{equation}
\bar{\partial}_{-} \bar{\partial}_{+}c- \bar{\partial}_{+} \bar{\partial}_{-}c=
\frac{1}{\sin2 A}\Big((\cos2 A \bar{\partial}_{+}\beta- \bar{\partial}_{-}\alpha) \bar{\partial}_{-}c-
(\bar{\partial}_{+}\beta-\cos2 A \bar{\partial}_{-}\alpha) \bar{\partial}_{+}c\Big).\label{13a}
\end{equation}
Inserting (\ref{13a}) into (\ref{12a}) we get
$$
\begin{array}{rcl}
&&\displaystyle(\sin\alpha+\sin\beta)\bar{\partial}_{+}\bar{\partial}_{-}c
\\[4pt]&=&-\displaystyle\frac{1}{\sin2 A}
\Big(\cos\beta\sin2 A\bar{\partial}_{-}\beta\bar{\partial}_{+}c+( \sin\alpha+\sin\beta)\cos2 A\bar{\partial}_{+}\beta\bar{\partial}_{-}c\\[6pt]&&
\qquad-(\sin\alpha+\sin\beta)\bar{\partial}_{-}\alpha\bar{\partial}_{-}c+\cos\alpha\sin2 A\bar{\partial}_{+}
\alpha\bar{\partial}_{-}c\Big)-\displaystyle\frac{1}{\kappa}\cdot\frac{{\rm d}\kappa}{{\rm d}c}(\sin\alpha+\sin\beta)\bar{\partial}_{+}c\bar{\partial}_{-}c.
\end{array}
$$
Thus, by (\ref{7})--(\ref{8}) we get
 the first equation of (\ref{cd}).

The proof for the other is similar. This completes the proof.
\end{proof}

\begin{prop} We have the following characteristic decompositions:
\begin{equation}
\left\{
  \begin{array}{ll}
\begin{aligned}
c\bar{\partial}_{+}\left(\frac{\bar{\partial}_{-}c}{\sin^{2} A}\right)
   &=\bar{\partial}_{-}c\Bigg[\frac{1}{2\mu\cos^{2} A}\Big(\frac{\bar{\partial}_{-}c}{\sin^2A}-\frac{\bar{\partial}_{+}c}{\sin^2A}\Big)
\\&\qquad\qquad\qquad+\left(\frac{1}{\mu\cos^{2} A}-4\kappa\sin^{2} A-2-\frac{c}{\kappa}\cdot\frac{{\rm d}\kappa}{{\rm d}c}\right)
\frac{\bar{\partial}_{+}c}{\sin^{2} A}\Bigg],
\end{aligned}
\\[32pt]
\begin{aligned}
c\bar{\partial}_{-}\left(\frac{\bar{\partial}_{+}c}{\sin^{2} A}\right)
   &=\bar{\partial}_{+}c\Bigg[\frac{1}{2\mu\cos^{2} A}\Big(\frac{\bar{\partial}_{+}c}{\sin^2A}-\frac{\bar{\partial}_{-}c}{\sin^2A}\Big)
\\&\qquad\qquad\qquad+\left(\frac{1}{\mu\cos^{2} A}-4\kappa\sin^{2} A-2-\frac{c}{\kappa}\cdot\frac{{\rm d}\kappa}{{\rm d}c}\right)
\frac{\bar{\partial}_{-}c}{\sin^{2} A}\Bigg].
\end{aligned}
\end{array}
\right.\label{cd1}
\end{equation}
\end{prop}
\begin{proof}
From (\ref{7})--(\ref{8}) we have
\begin{equation}\label{6802}
c\bar{\partial}_{-}A=\frac{\sin^2 A}{2}\left(2+\frac{\sin A}{\mu}\left(\frac{\sin^2 A}{\cos A}+\frac{1}{\cos A}-\varpi(\tau)\cos A\right)\frac{\bar{\partial}_{-}c}{\sin^2 A}\right).
\end{equation}
Combining this with (\ref{cd}) one get (\ref{cd1}).
\end{proof}

\begin{prop}
For the variable $\rho$, we have the following characteristic decompositions:
\begin{equation}\label{81104}
\left\{
  \begin{array}{ll}
   \displaystyle c\bar{\partial}_{+}\bar{\partial}_{-}\rho=\sin2A\bar{\partial}_{-}\rho+
    \frac{\tau^4p''(\tau)}{4c\cos^{2}A }\Big[(\bar{\partial}_{-}\rho)^2+(f-1)\bar{\partial}_{-}\rho\bar{\partial}_{+}\rho\Big],\\[12pt]
  \displaystyle c\bar{\partial}_{-}\bar{\partial}_{+}\rho=\sin2A\bar{\partial}_{+}\rho+
    \frac{\tau^4p''(\tau)}{4c\cos^{2}A }\Big[(\bar{\partial}_{+}\rho)^2+(f-1)\bar{\partial}_{-}\rho\bar{\partial}_{+}\rho\Big],
  \end{array}
\right.
\end{equation}
where
\begin{equation}
f=2\sin^2A-\frac{8p'(\tau)\cos^4A}{\tau p''(\tau)}>0\quad \mbox{for}\quad p''>0\quad \mbox{and}\quad p'<0.
\end{equation}
\end{prop}

\begin{proof}
This decomposition was first obtained in \cite{Lai3}. For the sake of completeness, we sketch the proof.

It follows from (\ref{11}) and (\ref{comm}) that
\begin{equation}
\begin{array}{rcl}
&&\bar{\partial}_{+}\big[c\tau\sin\alpha\bar{\partial}_{-}\rho\big]+
\bar{\partial}_{-}\big[c\tau\sin\beta\bar{\partial}_{+}\rho\big]
\\[6pt]&=&-\displaystyle\frac{1}{\sin2 A}\big[(\bar{\partial}_{+}\beta-\cos2 A \bar{\partial}_{-}\alpha)c\tau\sin\beta \bar{\partial}_{+}\rho-( \bar{\partial}_{-}\alpha- \cos2 A\bar{\partial}_{+}\beta)c\tau\sin\alpha \bar{\partial}_{-}\rho\big].
\end{array}
\end{equation}
Hence, we have
\begin{equation}
\begin{array}{rcl}
&&\displaystyle(\sin\alpha+\sin\beta)\frac{1}{c\tau }\frac{{\rm d}(c\tau)}{{\rm d}\rho }\bar{\partial}_{+}\rho\bar{\partial}_{-}\rho+\sin\alpha\bar{\partial}_{+}\bar{\partial}_{-}\rho+
\sin\beta\bar{\partial}_{-}\bar{\partial}_{+}\rho\\[8pt]&=&-\displaystyle\frac{1}{\sin2 A}
\Big[(\sin\beta \bar{\partial}_{+}\beta-\sin\beta \cos2 A \bar{\partial}_{-}\alpha+\cos\beta\sin2 A\bar{\partial}_{-}\beta) \bar{\partial}_{+}\rho\\[4pt]&&\qquad\qquad\qquad\qquad\displaystyle-( \sin\alpha\bar{\partial}_{-}\alpha- \sin\alpha\cos2 A\bar{\partial}_{+}\beta-\cos\alpha\sin2 A\bar{\partial}_{+}\alpha) \bar{\partial}_{-}\rho\Big].
\end{array}\label{12}
\end{equation}

Applying the commutator relation (\ref{comm}) for $\rho$, we get
\begin{equation}
\bar{\partial}_{-} \bar{\partial}_{+}\rho- \bar{\partial}_{+} \bar{\partial}_{-}\rho=
\frac{1}{\sin2 A}\big[(\cos2 A \bar{\partial}_{+}\beta- \bar{\partial}_{-}\alpha) \bar{\partial}_{-}\rho-
(\bar{\partial}_{+}\beta-\cos2 A \bar{\partial}_{-}\alpha) \bar{\partial}_{+}\rho\big].\label{13}
\end{equation}
Inserting this into (\ref{12}), we get
\begin{equation}
\begin{array}{rcl}
&&\displaystyle(\sin\alpha+\sin\beta)\frac{1}{c\tau }\frac{{\rm d}(c\tau)}{{\rm d}\rho }\bar{\partial}_{+}\rho\bar{\partial}_{-}\rho+(\sin\alpha+\sin\beta)\bar{\partial}_{+}\bar{\partial}_{-}\rho\\[6pt]
&=&-\displaystyle\frac{1}{\sin2 A}
\Big[\cos\beta\sin2 A\bar{\partial}_{-}\beta\bar{\partial}_{+}\rho+( \sin\alpha+\sin\beta)\cos2 A\bar{\partial}_{+}\beta\bar{\partial}_{-}\rho\\[6pt]&&
\qquad\qquad\qquad\qquad-( \sin\alpha+\sin\beta)\bar{\partial}_{-}\alpha\bar{\partial}_{-}\rho+\cos\alpha\sin2 A\bar{\partial}_{+}
\alpha\bar{\partial}_{-}\rho\Big].
\end{array}\label{14}
\end{equation}
Thus, by (\ref{7})--(\ref{8}) we get
\begin{equation}
 c\bar{\partial}_{+}\bar{\partial}_{-}\rho=\sin2A\bar{\partial}_{-}\rho+
    \frac{\tau^4p''(\tau)}{4c\cos^{2}A }\Big[(\bar{\partial}_{-}\rho)^2+(f-1)\bar{\partial}_{-}\rho\bar{\partial}_{+}\rho\Big].
\end{equation}

The proof for the other is similar. This completes the proof.
\end{proof}

\begin{prop}\label{pnew}
For any positive integer $n$, we have the following characteristic decompositions:
\begin{equation}\label{new}
\left\{
  \begin{array}{ll}
   \displaystyle c\bar{\partial}_{+}\Big(\frac{\rho^n\bar{\partial}_{-}\rho}{\sin^2 A}\Big)=
    \frac{\tau^4p''(\tau)\bar{\partial}_{-}\rho}{4c\cos^{2}A }\Bigg[\frac{\rho^n\bar{\partial}_{-}\rho}{\sin^2 A}+\mathcal{H}\frac{\rho^n\bar{\partial}_{+}\rho}{\sin^2 A}\Bigg],\\[16pt]
  \displaystyle c\bar{\partial}_{-}\Big(\frac{\rho^n\bar{\partial}_{+}\rho}{\sin^2 A}\Big)=
    \frac{\tau^4p''(\tau)\bar{\partial}_{+}\rho}{4c\cos^{2}A }\Bigg[\frac{\rho^n\bar{\partial}_{+}\rho}{\sin^2 A}+\mathcal{H}\frac{\rho^n\bar{\partial}_{-}\rho}{\sin^2 A}\Bigg],
  \end{array}
\right.
\end{equation}
where
\begin{equation}\label{82803}
\mathcal{H}=2\sin^2A-\frac{8p'(\tau)\cos^4A}{\tau p''(\tau)}-\frac{4np'(\tau)\cos^{2}A }{\tau p''(\tau)}-2\cos^2 A+2\Omega \cos^4 A-1.
\end{equation}
\end{prop}
\begin{proof}
The decompositions (\ref{new}) can be obtained by using (\ref{7})--(\ref{8}) and (\ref{81104}); we omit the details.
\end{proof}

\section{Hodograph transformation}

\subsection{Concept of hodograph transformation}
Let the hodograph transformation be
$$
T:  (\xi, \eta)\rightarrow (u, v)
$$
for (\ref{42501}).
If the Jacobian
$$
j(u, v;\xi, \eta)=\frac{\partial(u, v)}{\partial(\xi, \eta)}=u_{\xi}v_{\eta}-u_{\eta}v_{\xi}\neq 0
$$
for a solution $(u, v)(\xi, \eta)$ of (\ref{42501}), we may consider $\xi$ and $\eta$ as functions of $u$ and $v$. From
\begin{equation}
u_{\xi}=j\eta_{v}, \quad u_{\eta}=-j\xi_{v}, \quad v_{\xi}=-j\eta_{u}, \quad v_{\eta}=j\xi_{u},
\end{equation}
and regard $h$ as a function of $u$ and $v$,
we then see that $\xi(u, v)$ and $\eta(u, v)$ satisfy the equations
\begin{equation}\label{HT}
\left\{
  \begin{array}{ll}
  (c^{2}-U^{2})\eta_{v}+UV(\xi_{v}+\eta_{u})+(c^{2}-V^{2})\xi_{u}=0, \\[2pt]
  \xi_{v}-\eta_{u}=0.
  \end{array}
\right.
\end{equation}


The eigenvalues of (\ref{HT}) are
\begin{equation}
\Lambda=\Lambda_{\pm}=\frac{UV\pm
c\sqrt{U^{2}+V^{2}-c^{2}}}{c^{2}-V^{2}}.
\end{equation}
The characteristic curves of (\ref{HT}) in the $(u, v)$-plane are defined as the integral curves
\begin{equation}\label{823301}
\Gamma_{\pm}:~\frac{{\rm d}v}{{\rm d}u}=\Lambda_\pm.
\end{equation}
It is easy to see
\begin{equation}\label{52204}
\Lambda_{\pm}=-\frac{1}{\lambda_{\mp}},
\end{equation}
which implies that the directions of the $C$-characteristic of one kind are perpendicular to the $\Gamma$-characteristic of the other kind. Or, more precisely, the directions of $C_{+}$ and $\Gamma_{-}$ and of $C_{-}$ and $\Gamma_{+}$ through corresponding points $(\xi,\eta)$ and $(u, v)$ are perpendicular.

We introduce the normalized directional derivatives along $\Gamma$-characteristics,
\begin{equation}\label{52503}
\hat{\partial}_{+}=-\sin\beta\partial_{u}+\cos\beta\partial_{v}, \quad
\hat{\partial}_{-}=\sin\alpha\partial_{u}-\cos\alpha\partial_{v}.
\end{equation}

By (\ref{6803}) and (\ref{HT}) we have
\begin{equation}\label{38}
\hat{\partial}_{+}\eta=\tan \alpha \hat{\partial}_{+}\xi, \quad \hat{\partial}_{-}\eta=\tan \beta \hat{\partial}_{-}\xi.
\end{equation}
This implies a correspondence between $C_{\pm}$ and $\Gamma_{\pm}$. 

\subsection{Systems in the holograph plane}
From the pseudo-Bernoulli's law (\ref{5802}), one has
\begin{equation}\label{52201}
h_u=\xi-u=-U, \quad h_{v}=\eta-v=-V.
\end{equation}
We obtain the following first order system for $(U, V, h)$
\begin{equation}\label{52301}
\left(
  \begin{array}{ccc}
    c^2-V^2 & UV & 0 \\
    0 & 1 & 0 \\
    0 & 0 & 1 \\
  \end{array}
\right)\left(
         \begin{array}{c}
           U \\
           V \\
           h \\
         \end{array}
       \right)_u+ \left(
  \begin{array}{ccc}
    UV & c^2-U^2 & 0 \\
    -1 & 0 & 0 \\
    0 & 0 & 0 \\
  \end{array}
\right)\left(
         \begin{array}{c}
           U \\
           V \\
           h \\
         \end{array}
       \right)_v=\left(
                   \begin{array}{c}
                     2c^2-U^2-V^2 \\
                     0 \\
                     -U \\
                   \end{array}
                 \right).
\end{equation}
The system has three eigenvalues
\begin{equation}
\Lambda_0=0, \quad \Lambda_{\pm}=\frac{UV\pm
c\sqrt{U^{2}+V^{2}-c^{2}}}{c^{2}-V^{2}}.
\end{equation}

Multiplying (\ref{52301}) on the left by the matrix $(l_{+}, l_{-}, l_{0})^{T}$ where
$$
l_0=(0, 0, 1)\quad \mbox{and} \quad l_{\pm}=(1, \pm c\sqrt{U^{2}+V^{2}-c^{2}}, 0),
$$
 we obtain
\begin{equation}\label{52303}
\left\{
  \begin{array}{ll}
    \displaystyle\hat{\partial}_{+}U+\Lambda_{-}\hat{\partial}_{+}V=-\frac{2\cos 2A}{\sin\alpha}, \\[8pt]
   \displaystyle \hat{\partial}_{-}U+\Lambda_{+}\hat{\partial}_{-}V=\frac{2\cos 2A}{\sin\beta},\\[8pt]
    h_u=-U,
  \end{array}
\right.
\end{equation}

By (\ref{U}) and (\ref{52201}) we have
\begin{equation}\label{52501}
\hat{\partial}_{\pm}c=-\frac{1}{\kappa}.
\end{equation}
From (\ref{U}) we have
\begin{equation}\label{6601}
\hat{\partial}_{\pm}U=\frac{\cos\sigma \hat{\partial}_{\pm}c}{\sin A}-\frac{c\sin\sigma}{2\sin A}(\hat{\partial}_{\pm}\alpha+\hat{\partial}_{\pm}\beta)-\frac{c\cos\sigma\cos A}{2\sin^2 A}(\hat{\partial}_{\pm}\alpha-\hat{\partial}_{\pm}\beta)
\end{equation}
and
\begin{equation}\label{6602}
\hat{\partial}_{\pm}V=\frac{\sin\sigma \hat{\partial}_{\pm}c}{\sin A}+\frac{c\cos\sigma}{2\sin A}(\hat{\partial}_{\pm}\alpha+\hat{\partial}_{\pm}\beta)-\frac{c\sin\sigma\cos A}{2\sin^2 A}(\hat{\partial}_{\pm}\alpha-\hat{\partial}_{\pm}\beta).
\end{equation}

Inserting (\ref{52501})--(\ref{6602}) into (\ref{52303}) one gets
\begin{equation}\label{52305}
\left\{
  \begin{array}{ll}
    \displaystyle\hat{\partial}_{+}\alpha=-\frac{\tau p''(\tau)}{4cp'(\tau)}\Big[\varpi(\tau)-\tan^2A\Big]\sin 2A,  \\[10pt]
    \displaystyle\hat{\partial}_{-}\beta=\frac{\tau p''(\tau)}{4cp'(\tau)}\Big[\varpi(\tau)-\tan^2A\Big]\sin 2A,  \\[12pt]
   \displaystyle \tau_u=\frac{\tau\cos\sigma}{c \sin A}.
  \end{array}
\right.
\end{equation}
Actually, system (\ref{52305}) is linearly degenerate in the sense of Lax \cite{Lax}.

\subsection{Characteristic decompositions in the hodograph plane}

\begin{prop}
(Commutator relation of $\hat{\partial}_{\pm}$) We have
\begin{equation}\label{52401}
\hat{\partial}_{+}\hat{\partial}_{-}-\hat{\partial}_{-}\hat{\partial}_{+}=\tan A\hat{\partial}_{+}\alpha(\hat{\partial}_{+}-\hat{\partial}_{-}).
\end{equation}
\end{prop}
\begin{proof}
From (\ref{52503}) we have
\begin{equation}\label{317}
\partial_u=\frac{\cos\alpha\hat{\partial}_{+}+\cos\beta\hat{\partial}_{-}}{\sin 2A}, \quad \partial_v=\frac{\sin\alpha\hat{\partial}_{+}+\sin\beta\hat{\partial}_{-}}{\sin 2A}.
\end{equation}
By a direct computation one can obtain (\ref{52401}).
\end{proof}

Using the commutator relation, we easily derive:
\begin{prop}
For the variables $\alpha$ and $\beta$, we have the characteristic decompositions:
\begin{equation}\label{52502}
\left\{
  \begin{array}{ll}
    \hat{\partial}_{+}\hat{\partial}_{-}\alpha+\mathcal{W}(A, c)\hat{\partial}_{-}\alpha=\mathcal{Q}(A, c),  \\[4pt]
 \hat{\partial}_{-}\hat{\partial}_{+}\beta+\mathcal{W}(A, c)\hat{\partial}_{+}\beta=-\mathcal{Q}(A, c),
  \end{array}
\right.
\end{equation}
where
$$
\mathcal{W}(A, c)=-\frac{\tau p''(\tau)}{4cp'(\tau)}\Big[\big(\varpi(\tau)-\tan^2A\big)(4\sin^2 A-1)+2\tan^2A\Big],
$$
$$
\mathcal{Q}(A, c)=\left(\frac{\tau p''(\tau)}{4cp'(\tau)}\right)^2\big(\varpi(\tau)-\tan^2A\big)(1-3\tan^2A)\sin 2A-\frac{\tan A}{2\tau p'(\tau)}\cdot\frac{{\rm d}}{{\rm d}\tau}\left(\frac{\tau p''(\tau)}{p'(\tau)}\right).
$$
\end{prop}
\begin{proof}
We apply the commutator relation (\ref{52401}) to obtain
$$
\hat{\partial}_{+}\hat{\partial}_{-}\alpha-\hat{\partial}_{-}\hat{\partial}_{+}\alpha=\tan A\hat{\partial}_{+}\alpha(\hat{\partial}_{+}\alpha-\hat{\partial}_{-}\alpha).
$$
Using the expressions of $\hat{\partial}_{+}\alpha$ and $\hat{\partial}_{-}\beta$ in (\ref{52305}) and recalling (\ref{52501}), we directly obtain the first equation of (\ref{52502}). The second equation of (\ref{52502}) can be proved similarly.
\end{proof}

We need to show that the Jacobian $j^{-1}(u, v; \xi, \eta)$ does not vanish:
$$
j^{-1}(u, v; \xi, \eta)=\xi_u\eta_v-\xi_v\eta_u\neq0.
$$
By computation we have
$$
j^{-1}(u, v; \xi, \eta)=\frac{\sin^2 2A}{\cos\alpha\cos\beta}\hat{\partial}_{-}\xi\cdot\hat{\partial}_{+}\xi.
$$

A direct computation yields
\begin{equation}\label{61312}
\hat{\partial}_{+}\xi=-\hat{\partial}_{+}U-\sin\beta\quad \mbox{and}\quad \hat{\partial}_{-}\xi=-\hat{\partial}_{-}U+\sin\alpha.
\end{equation}
We compute
\begin{equation}\label{61311}
\begin{aligned}
\hat{\partial}_{+}U+\sin\beta&=\frac{\cos\sigma \hat{\partial}_{+}c}{\sin A}-\frac{c\sin\sigma}{2\sin A}(\hat{\partial}_{+}\alpha+\hat{\partial}_{+}\beta)-\frac{c\cos\sigma\cos A}{2\sin^2 A}(\hat{\partial}_{+}\alpha-\hat{\partial}_{+}\beta)+\sin\beta\\&=
\frac{\tau p''(\tau)}{2p'(\tau)}\cdot\frac{\cos\alpha}{\sin 2A}+\frac{c\cos\alpha}{2\sin^2 A}\hat{\partial}_{+}\beta
\\&=\frac{c\cos\alpha}{2\sin^2 A}\left(\hat{\partial}_{+}\beta+\frac{\tau p''(\tau)\tan A}{2cp'(\tau)}\right)
\end{aligned}
\end{equation}
and
\begin{equation}\label{61319}
\begin{aligned}
\hat{\partial}_{-}U-\sin\alpha&=\frac{\cos\sigma \hat{\partial}_{-}c}{\sin A}-\frac{c\sin\sigma}{2\sin A}(\hat{\partial}_{-}\alpha+\hat{\partial}_{-}\beta)-\frac{c\cos\sigma\cos A}{2\sin^2 A}(\hat{\partial}_{-}\alpha-\hat{\partial}_{-}\beta)-\sin\alpha\\&=
\frac{\tau p''(\tau)}{2p'(\tau)}\cdot\frac{\cos\beta}{\sin 2A}-\frac{c\cos\beta}{2\sin^2 A}\hat{\partial}_{-}\alpha
\\&=-\frac{c\cos\beta}{2\sin^2 A}\left(\hat{\partial}_{-}\alpha-\frac{\tau p''(\tau)\tan A}{2cp'(\tau)}\right).
\end{aligned}
\end{equation}
Therefore, we have
\begin{equation}\label{6805}
j^{-1}(u, v; \xi, \eta)=-(c\cot A)^2 \mathcal{Z}_{+}\mathcal{Z}_{-},
\end{equation}
where
\begin{equation}\label{62701}
\mathcal{Z}_{+}=\hat{\partial}_{+}\beta+\frac{\tau p''(\tau)\tan A}{2cp'(\tau)}\quad \mbox{and} \quad
\mathcal{Z}_{-}=\hat{\partial}_{-}\alpha-\frac{\tau p''(\tau)\tan A}{2cp'(\tau)}.
\end{equation}
So, in order that $j^{-1}$ does not vanish we need to obtain $\mathcal{Z}_{\pm}\neq0$, $c\neq0$, and $A\neq\frac{\pi}{2}$.

In order to estimate $\mathcal{Z}_{\pm}$, we use (\ref{52502}) to derive the following characteristic equations:
\begin{prop}
For the variables $\mathcal{Z}_{\pm}$, we have
\begin{equation}\label{322}
\left\{
  \begin{array}{ll}
   \hat{\partial}_{+}\mathcal{Z}_{-}+\mathcal{W}\mathcal{Z}_{-}=\displaystyle\frac{\tau p''(\tau)}{4cp'(\tau)}(\tan^2 A+1)\mathcal{Z}_{+}, \\[10pt]
\hat{\partial}_{-}\mathcal{Z}_{+}+\mathcal{W}\mathcal{Z}_{+}=\displaystyle\frac{\tau p''(\tau)}{4cp'(\tau)}(\tan^2 A+1)\mathcal{Z}_{-}.
  \end{array}
\right.
\end{equation}
\end{prop}

 We compute
$$
\hat{\partial}_{+}c=c_{\xi}\hat{\partial}_{+}\xi+c_{\eta}\hat{\partial}_{+}\eta
=(c_{\xi}+c_{\eta}\tan\alpha)\hat{\partial}_{+}\xi=\sec\alpha\bar{\partial}_{+}c \hat{\partial}_{+}\xi.
$$
Thus, we have
$$
\bar{\partial}_{+}c=\frac{2\sin^2 A}{c\kappa\mathcal{Z}_{+}}.
$$
Combining this with $\frac{{\rm d}c}{{\rm d}\tau}=-\frac{c}{\tau \kappa}$, we get
\begin{equation}\label{6807}
\bar{\partial}_{+}\tau=-\frac{2\tau\sin^2 A}{c^2\mathcal{Z}_{+}}.
\end{equation}
Similarly, one has
\begin{equation}\label{6808}
\bar{\partial}_{-}\tau=\frac{2\tau\sin^2 A}{c^2\mathcal{Z}_{-}}.
\end{equation}

\section{Interaction of fan-shock-fan composite waves}
\subsection{Characteristic boundaries of the interaction region}
In this part we shall construct the cross characteristic curves $\wideparen{\mbox{PB}}$, $\wideparen{\mathrm{BE}}$, $\wideparen{\mathrm{PD}}$, and $\wideparen{\mathrm{DF}}$.

Set $\hat{\xi}=\xi\sin\theta-\eta\cos\theta$, $\hat{\eta}=\xi\cos\theta+\eta\sin\theta$. Then, by the result of Section 1.5 we know that
\begin{equation}\label{42601}
(u, v, \tau)=\left\{
            \begin{array}{ll}
\big(0, 0, \tau_0\big), & \hbox{$\hat{\xi}>\hat{\xi}_0$,} \\[4pt]
               \big( \hat{u}_r(\hat{\xi})\sin\theta, -\hat{u}_r(\hat{\xi})\cos\theta, \hat{\tau}_r(\hat{\xi})\big), & \hbox{$\hat{\xi}_1<\hat{\xi}<\hat{\xi}_0$;} \\[4pt]
             \big(\hat{u}_l(\hat{\xi})\sin\theta , -\hat{u}_l(\hat{\xi})\cos\theta, \hat{\tau}_l(\hat{\xi})\big), & \hbox{$\hat{\xi}_2<\hat{\xi}<\hat{\xi}_1$;} \\[4pt]
\mbox{vacuum}, & \hbox{$\hat{\xi}<\hat{\xi}_2$}
            \end{array}
          \right.
\end{equation}
is a fan-shock-fan composite wave solution to (\ref{42501}). Moreover, for any fixed $\hat{\xi}\in(\hat{\xi}_2, \hat{\xi}_0]$, the half line $\xi\sin\theta-\eta\cos\theta=\hat{\xi}$, $\xi\cos\theta+\eta\sin\theta>0$ is a $C_{+}$ characteristic line with characteristic angle $\alpha=\pi+\theta$.
So, we have
\begin{equation}\label{62001}
\alpha\equiv\pi+\theta\quad \mbox{on}\quad \wideparen{\mathrm{PB}}\cup\wideparen{\mathrm{BE}},
\end{equation}
provided that the $C_{-}$ characteristic curves $\wideparen{\mathrm{PB}}$ and $\wideparen{\mathrm{BE}}$ exist; see Figure \ref{Fig12}.

\begin{figure}[htbp]
\begin{center}
\includegraphics[scale=0.48]{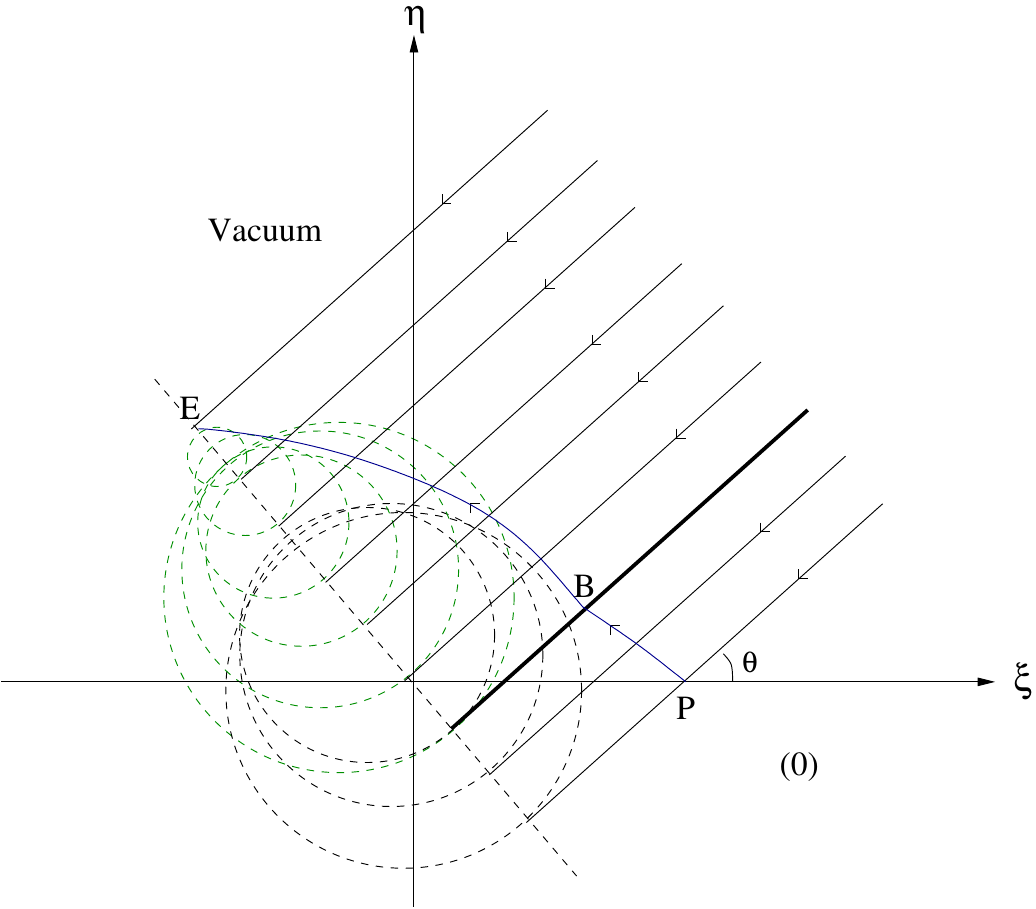}
\caption{\footnotesize Fan-shock-fan composite wave $FSF_1$ with straight $C_{+}$ characteristic lines and the cross $C_{-}$ characteristic curve $\wideparen{\mathrm{PB}}\cup\wideparen{\mathrm{BE}}$.}
\label{Fig12}
\end{center}
\end{figure}

We describe the $C_{-}$ characteristic curve $\wideparen{\mathrm{PB}}$ by the parametric form
$\xi=\xi_{-}^{r}(\hat{\xi})$, $\eta=\eta_{-}^{r}(\hat{\xi})
$, $\hat{\xi}_1\leq\hat{\xi}\leq\hat{\xi}_0$.
The values of $u$, $v$, and $\tau$ along $\wideparen{\mathrm{PB}}$ are $u=\hat{u}_r(\hat{\xi})\sin\theta$, $v=-\hat{u}_r(\hat{\xi})\cos\theta$, and $\tau=\hat{\tau}_r(\hat{\xi})$. We denote the values of the pseudo-Mach number along $\wideparen{\mathrm{PB}}$ by $A=A_{-}^{r}(\hat{\xi})$.
Then by $\xi=u-U=u-\frac{c\cos\sigma}{\sin A}$ and $\eta=v-V=v-\frac{c\sin\sigma}{\sin A}$ we have
$$
\left\{
    \begin{array}{ll}
 \xi_{-}^{r}(\hat{\xi})=\hat{u}_r(\hat{\xi})\sin\theta-\hat{c}_r(\hat{\xi})
\csc(A_{-}^{r}(\hat{\xi}))\cos\big(\pi+\theta-A_{-}^{r}(\hat{\xi})\big),\\[8pt]
\eta_{-}^{r}(\hat{\xi})=-\hat{u}_r(\hat{\xi})\cos\theta-\hat{c}_r(\hat{\xi})\csc(A_{-}^{r}(\hat{\xi}))\sin\big(\pi+\theta-A_{-}^{r}(\hat{\xi})\big),
    \end{array}
  \right.
$$
where $\hat{c}_r(\hat{\xi})=c(\hat{\tau}_r(\hat{\xi}))$.
In addition,
by (\ref{tau}), (\ref{8}) and (\ref{62001}) we know that $A_{-}^{r}(\hat{\xi})$ can be determined by
$$
\left\{
  \begin{array}{ll}
   \displaystyle \frac{{\rm d}A_{-}^{r}(\hat{\xi})}{{\rm d}\hat{\xi}}=\frac{\hat{\tau}_r^2(\hat{\xi})p''(\hat{\tau}_r(\hat{\xi}))}
{8\hat{c}_r^2(\hat{\xi})}\Big[\varpi(\hat{\tau}_r(\hat{\xi}))-\tan^2(A_{-}^{r}(\hat{\xi}))\Big]\sin (2A_{-}^{r}(\hat{\xi}))\hat{\tau}_r'(\hat{\xi}), & \hbox{$\hat{\xi}_1\leq\hat{\xi}\leq\hat{\xi}_0$;} \\[12pt]
    A_{-}^{r}(\hat{\xi}_0)=\theta.
  \end{array}
\right.
$$
Here, the condition $A_{-}^{r}(\hat{\xi}_0)=\theta$ is determined by $\big(\xi_{-}^{r}, \eta_{-}^{r}\big)(\hat{\xi}_0)=\big(\frac{\hat{\xi}_0}{\sin\theta}, 0\big)$ and $A\big(\frac{\hat{\xi}_0}{\sin\theta}, 0\big)=\theta$.

We denote by $(\xi_{_\mathrm{B}}, \eta_{_\mathrm{B}})$ the coordinate of the point $\mathrm{B}$. Then we have
$
(\xi_{_\mathrm{B}}, \eta_{_\mathrm{B}})=(\xi_{-}^{r}, \eta_{-}^{r})(\hat{\xi}_1)
$.
Hence, the value of the pseudo-Mach angle of the flow behind the double-sonic shock ${\it S}_1:\xi\sin\theta-\eta\cos\theta=\hat{\xi}_1$ at the point $\mathrm{B}$ is
$$
A=\hat{A}_b:=\arcsin \left(\frac{c_2^e}{\sqrt{(\xi_{_\mathrm{B}}-\hat{u}_2\sin\theta)^2+(\eta_{_\mathrm{B}}+\hat{u}_2\cos\theta)^2}}\right),
$$
where the constant $\hat{u}_2$ is already defined in Section 1.5.
\vskip 2pt

We describe the $C_{-}$ characteristic curve $\wideparen{\mathrm{BE}}$ by the parametric form
$\xi=\xi_{-}^{l}(\hat{\xi})$, $\eta=\eta_{-}^{l}(\hat{\xi})
$, $\hat{\xi}_2\leq\hat{\xi}\leq\hat{\xi}_1$.
Then
$$
\left\{
  \begin{array}{ll}
    \xi_{-}^{l}(\hat{\xi})=\hat{u}_l(\hat{\xi})\sin\theta-\hat{c}_l(\hat{\xi})
\csc(A_{-}^{l}(\hat{\xi}))\cos\big(\pi+\theta-A_{-}^{l}(\hat{\xi})\big),\\[8pt]
\eta_{-}^{l}(\hat{\xi})=-\hat{u}_r(\hat{\xi})\cos\theta-\hat{c}_r(\hat{\xi})\csc (A_{-}^{l}(\hat{\xi}))\sin\big(\pi+\theta-A_{-}^{l}(\hat{\xi})\big),
    \end{array}
  \right.
$$
where $\hat{c}_l(\hat{\xi})=c(\hat{\tau}_l(\hat{\xi}))$  and $A_{-}^{l}(\hat{\xi})$ is determined by
$$
\left\{
  \begin{array}{ll}
   \displaystyle \frac{{\rm d}A_{-}^{l}(\hat{\xi})}{{\rm d}\hat{\xi}}=\frac{\hat{\tau}_l^2(\hat{\xi})p''(\hat{\tau}_l(\hat{\xi}))}{8c_l^2(\hat{\xi})}
\Big[\varpi(\hat{\tau}_l(\hat{\xi}))-\tan^2(A_{-}^{l}(\hat{\xi}))\Big]\sin (2A_{-}^{l}(\hat{\xi}))\hat{\tau}_l'(\hat{\xi}), & \hbox{$\hat{\xi}_2\leq\hat{\xi}\leq\hat{\xi}_1$;} \\[12pt]
   A_{-}^{l}(\hat{\xi}_1)=\hat{A}_b.
  \end{array}
\right.
$$

In section 4.6 we will show $0<A_{-}^{r}(\hat{\xi})<\frac{\pi}{2}$ for $\hat{\xi}_1\leq\hat{\xi}\leq\hat{\xi}_0$ and
$0<A_{-}^{l}(\hat{\xi})<\frac{\pi}{2}$ for $\hat{\xi}_2\leq\hat{\xi}\leq\hat{\xi}_1$. This implies that the forward $C_{-}$ characteristic curve issued from $\mathrm{P}$ can cross the whole fan-shock-fan composite wave ${\it FSF}_1$ and ends up at the point $\mathrm{E}=(\hat{u}_l(\hat{\xi}_2)\sin\theta, -\hat{u}_l(\hat{\xi}_2)\cos\theta)$.

From (\ref{82301}) and (\ref{62001}) we also have
\begin{equation}\label{62501}
\bar{\partial}_{-}\rho=-\frac{2c\sin 2A}{\tau^4 p''(\tau)}<0\quad \mbox{along}\quad \wideparen{\mbox{PB}}\cup\wideparen{\mbox{BE}}.
\end{equation}

By the symmetry, the forward $C_{+}$ characteristic issued $\mathrm{P}$ can cross the whole fan-shock-fan composite wave ${\it FSF}_2$ and ends up at the point $\mathrm{F}=(\hat{u}_l(\hat{\xi}_2)\sin\theta, \hat{u}_l(\hat{\xi}_2)\cos\theta)$.
Moreover,
\begin{equation}\label{62601}
\beta\equiv\pi-\theta\quad \mbox{on}\quad \wideparen{\mathrm{PD}}\cup\wideparen{\mathrm{DF}}.
\end{equation}
Combining this with (\ref{7}) we have
\begin{equation}\label{62502}
\bar{\partial}_{+}\rho=-\frac{2c\sin 2A}{\tau^4p''(\tau)}<0\quad \mbox{along}\quad \wideparen{\mathrm{PD}}\cup\wideparen{\mathrm{DF}}.
\end{equation}
We need to mention that
the monotonicity conditions (\ref{62501}) and  (\ref{62502}) are crucial in constructing a global piecewise smooth solution to the problem (\ref{42501}, \ref{426011}).

In the following subsections, we are going to construct a solution in the fan-shock-fan composite wave interaction region $\Sigma$ piece-by-piece.

\subsection{Interaction of the fan waves}
We first consider the interaction of the fan waves ${\it F}_{1r}$ and ${\it F}_{2r}$.
So, we
consider (\ref{42501}) with the following boundary conditions:
\begin{equation}\label{6809}
(u, v, \tau)=\left\{
            \begin{array}{ll}
               \big(\sin\theta \hat{u}_r, -\cos\theta u_r, \hat{\tau}_r\big)(\xi\sin\theta-\eta\cos\theta) & \hbox{on $\wideparen{\mathrm{PB}}$;} \\[4pt]
             \big(\sin\theta \hat{u}_r, \cos\theta \hat{u}_r, \hat{\tau}_r\big)(\xi\sin\theta+\eta\cos\theta) & \hbox{on $\wideparen{\mathrm{PD}}$.}  \\[4pt]
            \end{array}
          \right.
\end{equation}

Problem  (\ref{42501}, \ref{6809}) is a standard Goursat problem, since $\wideparen{\mathrm{PB}}$ is a $C_{-}$ characteristic curve and $\wideparen{\mathrm{PD}}$ is a $C_{+}$ characteristic curve and the corresponding characteristic equations are hold along them.
\begin{lem}
When $\tau_1^e-\tau_0$ is sufficiently small, the Goursat problem  (\ref{42501}, \ref{6809}) admits a classical solution
 in a curved quadrilateral domain $\Sigma^1$ closed by characteristic curves $\wideparen{\mathrm{PB}}$, $\wideparen{\mathrm{PD}}$, $\wideparen{\mathrm{BG_1}}$, and $\wideparen{\mathrm{DG_1}}$, where $\wideparen{\mathrm{BG_1}}$ is a forward $C_{+}$ characteristic curve issued from $\mathrm{B}$ and $\wideparen{\mathrm{DG_1}}$ is a forward $C_{-}$ characteristic curve issued from $\mathrm{D}$; see Figure \ref{Fig13}. Moreover, the solution satisfies
\begin{equation}\label{6811}
\Big\|\big(\tau-\tau_1^e, \alpha-\pi-\theta, \beta-\pi+\theta, \bar{\partial}_{\pm}\rho+\mathcal{L}\big)\Big\|_{_{0; \Sigma^1}}\rightarrow 0\quad \mbox{as} \quad \tau_1^e-\tau_0\rightarrow 0,
\end{equation}
where $$\mathcal{L}:=\frac{2c_1^e\sin 2\theta}{(\tau_1^e)^4p''(\tau_1^e)}.$$
\end{lem}

\begin{proof}
When $\tau_1^e-\tau_0$ is sufficiently small, the existence of $C^1$ solution in $\Sigma^1$ follows routinely from the idea of Li and Yu \cite{{Li-Yu}} (Chap. 2). It can be seen as a local solution of a ``bigger" Gorusat problem.
The estimate (\ref{6811}) can be obtained directly by (\ref{62001})--(\ref{62502}) and the characteristic decompositions (\ref{81104}).
\end{proof}

From (\ref{6811}) we immediately have
\begin{equation}\label{72201}
\bar{\partial}_{\pm}\rho<0\quad \mbox{in}\quad \Sigma^1,
\end{equation}
as $\tau_1^e-\tau_0$ is sufficiently small.

We denote the solution of the Goursat problem  (\ref{42501}, \ref{6809}) by $(u,v,\tau)=(u_1, v_1, \tau_1)(\xi, \eta)$.


\subsection{Shock waves in the interaction region}
The double-sonic shocks ${\it S}_1$ and ${\it S}_2$ will go into the interaction region from $\mathrm{B}$ and $\mathrm{D}$, respectively. For convenience, we denote the shocks from $\mathrm{B}$ and $\mathrm{D}$ by ${\it S}_{_\mathrm{B}}$ and ${\it S}_{_\mathrm{D}}$, respectively.

In what follows, we are going to discuss the shock wave issued from $\mathrm{B}$. We assume that the shock curve ${\it S}_{_\mathrm{B}}$ is smooth.
We use subscripts `f' and `b' to denote the forward and the backside states (including $u$, $v$, $\tau$, $c$, $\alpha$, $\beta$, $\sigma$, and $A$) of the shock ${\it S}_{_\mathrm{B}}$, respectively.
We denote by $\chi$ the inclination angle of the shock curve ${\it S}_{_\mathrm{B}}$. 
Let
\begin{equation}\label{62604}
L=U\cos \chi+V\sin \chi, \quad N=U \sin \chi- V\cos \chi.
\end{equation}
Then by the Rankine-Hugoniot relations we have that on ${\it S}_{_\mathrm{B}}$,
\begin{equation}\label{RH1}
\left\{
  \begin{array}{ll}
    \rho_f N_f=\rho_b N_b=m,  \\[4pt]
L_f=L_b,\\[4pt]
   N_f^2+2h(\tau_f)=N_b^2+2h(\tau_b).
  \end{array}
\right.
\end{equation}
By the first and the third relations of (\ref{RH1}) one has
\begin{equation}\label{250101}
m^2=-\frac{2h(\tau_f)-2h(\tau_b)}{\tau_f^2-\tau_b^2}.
\end{equation}

We compute
$$
N=U\sin\chi-V\cos\chi=q\cos\sigma\sin\chi-q\sin\sigma\cos\chi=q\sin(\chi-\sigma).
$$
Thus, we have
\begin{equation}\label{61112}
\chi=\sigma+\arcsin\Big(\frac{N}{q}\Big)\quad \mbox{on}\quad{\it S}_{_\mathrm{B}}.
\end{equation}

Since $\tau_f(\mathrm{B})=\tau_1^e$ and $\tau_b(\mathrm{B})=\tau_2^e$,
we have
\begin{equation}\label{250401}
N_b=c_b=q_b\sin A_b\quad \mbox{and}\quad \alpha_b=\chi=\pi+\theta\quad  \mbox{at}\quad \mathrm{B};
\end{equation}
\begin{equation}\label{102002}
N_f=c_f=q_f\sin A_f\quad \mbox{and}\quad \alpha_f=\chi=\pi+\theta\quad \mbox{at}\quad \mathrm{B}.
\end{equation}

\subsubsection{\bf Derivatives along the shock ${\it S}_{_\mathrm{B}}$}
We define the directional derivative:
$$
\bar{\partial}_s:=\cos\chi\partial_{\xi}+\sin\chi\partial_{\eta}.
$$

From (\ref{250101}) we have
\begin{equation}\label{250104}
\begin{aligned}
2m \bar{\partial}_sm ~=~&\frac{2\tau_f}{\tau_f^2-\tau_b^2}\left(-p'(\tau_f)+\frac{2h(\tau_f)-2h(\tau_b)}{\tau_f^2-\tau_b^2}\right)
\bar{\partial}_s\tau_f\\&
-\frac{2\tau_b}{\tau_f^2-\tau_b^2}\left(-p'(\tau_b)+\frac{2h(\tau_f)-2h(\tau_b)}{\tau_f^2-\tau_b^2}\right)
\bar{\partial}_s\tau_b\quad \mbox{on}\quad {\it S}_{_\mathrm{B}}.
\end{aligned}
\end{equation}
Combining this with $\tau_f(\mathrm{B})=\tau_1^e$ and $\tau_b(\mathrm{B})=\tau_2^e$,  we have
\begin{equation}\label{250402}
\bar{\partial}_sm=0\quad \mbox{at}\quad \mathrm{B}.
\end{equation}

From (\ref{U}) and (\ref{62604}) we have
$$
m=\rho  N =\rho  q  (\cos\sigma \sin\chi-\sin\sigma \cos\chi).
$$
Thus,
\begin{equation}\label{4250103}
\bar{\partial}_s m=\frac{N }{q }\partial_{s}(\rho q )-\rho L \bar{\partial}_s\sigma +\rho L \bar{\partial}_s \chi\quad \mbox{on}\quad {\it S}_{_\mathrm{B}}.
\end{equation}

From the pseudo-Bernoulli law (\ref{5802}) we have
$$
q \bar{\partial}_s q +\tau  p'(\tau )\bar{\partial}_s \tau +L  =0\quad \mbox{on}\quad {\it S}_{_\mathrm{B}}.
$$
This yields
$$
\bar{\partial}_s (\rho q )=\frac{1}{q }(q ^2+\tau ^2 p'(\tau ))\bar{\partial}_s\rho -\frac{\rho L }{q }\quad \mbox{on}\quad {\it S}_{_\mathrm{B}}.
$$
Hence, we get
\begin{equation}\label{250102}
\bar{\partial}_s \chi=\frac{\bar{\partial}_s m}{\rho L }-\frac{N }{\rho L q ^2}(q ^2+\tau ^2 p'(\tau ))\bar{\partial}_s\rho +\frac{N }{q ^2}+\bar{\partial}_s\sigma \quad \mbox{on}\quad {\it S}_{_\mathrm{B}}.
\end{equation}

From the pseudo-Bernoulli law (\ref{5802}), we have
\begin{equation}\label{42803a}
\bar{\partial}_s q =\frac{\tau ^3 p'(\tau )\bar{\partial}_s \rho }{q }-\frac{L }{q }=-\tau q \sin^2 A \bar{\partial}_s \rho -\frac{L }{q }\quad \mbox{on}\quad {\it S}_{_\mathrm{B}}.
\end{equation}
From (\ref{250102}) we have
\begin{equation}\label{42804a}
\bar{\partial}_s \sigma =\bar{\partial}_s \chi-\frac{\bar{\partial}_s m}{\rho L }+\frac{N }{\rho L q ^2}(q ^2+\tau ^2 p'(\tau ))\bar{\partial}_s\rho -\frac{N }{q ^2}\quad \mbox{on}\quad {\it S}_{_\mathrm{B}}.
\end{equation}

We also have
\begin{equation}\label{50701}
\bar{\partial}_su=\bar{\partial}_s(U+\xi)=\bar{\partial}_s(q\cos\sigma)+\cos\chi=\cos\sigma\bar{\partial}_sq-q\sin\sigma\bar{\partial}_s\sigma+\cos\chi,
\end{equation}
\begin{equation}\label{50702}
\bar{\partial}_sv=\bar{\partial}_s(V+\eta)=\bar{\partial}_s(q\sin\sigma)+\sin\chi=\sin\sigma\bar{\partial}_sq
+q\cos\sigma\bar{\partial}_s\sigma+\sin\chi.
\end{equation}

\subsubsection{\bf Existence of post-sonic shocks}

We now assume that shocks issued from $\mathrm{B}$ and $\mathrm{D}$ are post-sonic.
We shall show that the post-sonic shocks issued from $\mathrm{B}$ and $\mathrm{D}$ will stay in $\Sigma^1$ until they intersect with each other at some point $\mathrm{G}$ on the $\xi$-axis; see Figure \ref{Fig13}.
 We denote by $\wideparen{\mathrm{BG}}$ and $\wideparen{\mathrm{DG}}$ the post-sonic shocks issued from $\mathrm{B}$ and $\mathrm{D}$, respectively.


\begin{figure}[htbp]
\begin{center}
\includegraphics[scale=0.46]{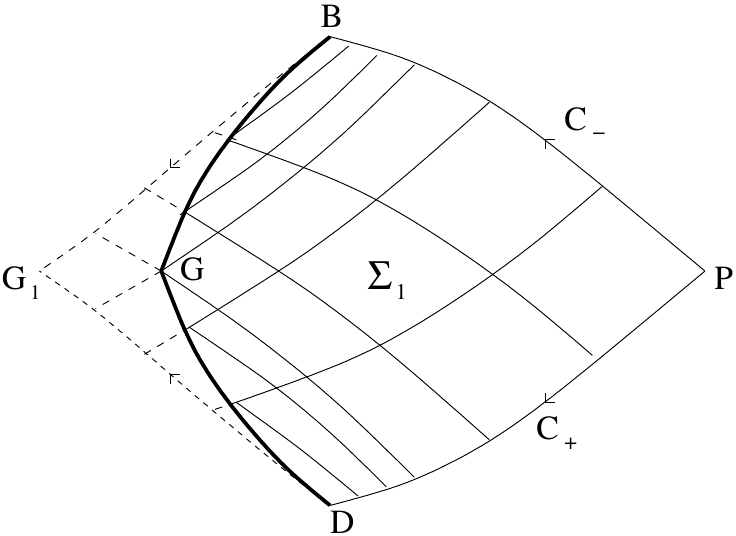}\qquad \qquad \qquad \includegraphics[scale=0.43]{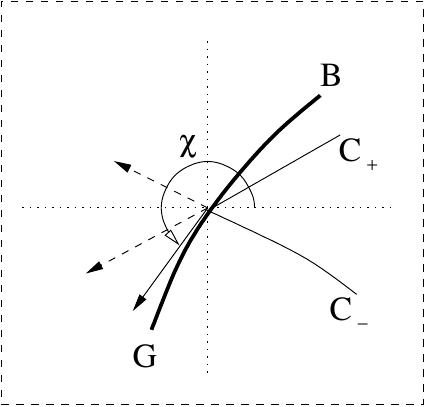}
\caption{\footnotesize Domain $\Sigma_1$ and post-sonic shock curves $\wideparen{BG}$ and  $\wideparen{DG}$.}
\label{Fig13}
\end{center}
\end{figure}

We first assume that $\wideparen{\mathrm{BG}}$ states in the interior of $\Sigma^1$.
Then the front side states $(u_f, v_f, \tau_f)=(u_1, v_1, \tau_1)(\xi, \eta)$ on $\wideparen{\mathrm{BG}}$.
We also assume
\begin{equation}\label{72202}
\tau_1^e<\tau_1<\tau_2^i\quad \mbox{on}\quad \wideparen{\mathrm{BG}}\setminus \mathrm{B}.
\end{equation}
Then by Proposition \ref{51003} we have
\begin{equation}\label{62603}
\tau_b=s_{po}(\tau_1)\quad \mbox{and}\quad
m^2=-p'(s_{po}(\tau_1))>-p'(\tau_1)\quad \mbox{on}\quad \wideparen{\mathrm{BG}}\setminus \mathrm{B}.
\end{equation}
This implies
\begin{equation}\label{6810}
N_b=c_b:=\sqrt{-\tau_b^2p'(\tau_b)}\quad \mbox{and}\quad N_1>c_1:=\sqrt{-\tau_1^2p'(\tau_1)}\quad \mbox{on}\quad \wideparen{\mathrm{BG}}\setminus \mathrm{B}.
\end{equation}

We still denote by $\chi$ the inclination angle of the shock curve $\wideparen{\mathrm{BG}}$.
By (\ref{6810}) and (\ref{61112}) we have
\begin{equation}\label{61113}
\chi=\sigma_b+A_b=\alpha_b\quad\mbox{and}\quad \chi>\sigma_1+A_1=\alpha_1  \quad \mbox{on}\quad \wideparen{\mathrm{BG}}.
\end{equation}


We differentiate $m^2=-p'(s_{po}(\tau_1))$ along $\wideparen{\mathrm{BG}}$ to obtain
\begin{equation}\label{61105}
\bar{\partial}_s m=\frac{-p''(s_{po}(\tau_1))s_{po}'(\tau_1)\bar{\partial}_s \tau_1}{2m }\quad \mbox{on}\quad \wideparen{\mathrm{BG}}.
\end{equation}


As in (\ref{250102}), we have
\begin{equation}\label{61115}
\bar{\partial}_s \chi=\frac{\bar{\partial}_s m}{\rho_1L_1}-\frac{N_1}{\rho_1L_1q_1^2}(q_1^2+\tau_1^2 p'(\tau_1))\bar{\partial}_s\rho_1+\frac{N_1}{q_1^2}+\bar{\partial}_s\sigma_1\quad \mbox{on}\quad \wideparen{\mathrm{BG}}.
\end{equation}

In what follows we are going to discuss the existence of the post-sonic shock curve $\wideparen{\mathrm{BG}}$.
We denote by $\xi=\xi_{+}(\eta)$, $0\leq \eta\leq \eta_{_\mathrm{B}}$ the $C_{+}$ characteristic curve $\wideparen{\mathrm{BG}_{1}}$ and by
$\xi=\xi_{s}(\eta)$, $0\leq \eta\leq \eta_{_\mathrm{B}}$ the post-sonic shock curve $\wideparen{\mathrm{BG}}$.
We have
\begin{equation}\label{61102}
\begin{aligned}
&\xi_{s}'(\eta)=\cot(\chi(\xi_s(\eta), \eta)), \quad \xi_{+}'(\eta)=\cot(\alpha_1(\xi_+(\eta), \eta)), \quad\\&
\xi_{s}(\eta_{_\mathrm{B}})=\xi_{+}(\eta_{_\mathrm{B}}), \quad  \xi_{s}'(\eta_{_\mathrm{B}})=\xi_{+}'(\eta_{_B})=\cot(\pi+\theta).
\end{aligned}
\end{equation}

Actually, in order to obtain $\xi_s(\eta)$, we only need to consider (\ref{61115}) with data
\begin{equation}\label{61114}
\chi(\mathrm{B})=\alpha_1(\mathrm{B})=\pi+\theta.
\end{equation}
Since the the initial value problem (\ref{61115}, \ref{61114}) use the data in $\Sigma^1$, in order to prove the existence of $\wideparen{\mathrm{BG}}$ we need to prove
\begin{equation}\label{61101}
\xi_{s}(\eta)>\xi_{+}(\eta), \quad 0\leq \eta<\eta_{_\mathrm{B}}.
\end{equation}

From $\cot(\alpha_1(\xi_{+}(\eta),\eta))=\xi_{+}'(\eta)$ we have
\begin{equation}\label{61103}
\xi_{+}''(\eta)=-\frac{\bar{\partial}_{+}\alpha_1(\xi_{+}(\eta),\eta)}{\sin^3\alpha_1(\xi_{+}(\eta),\eta)},
\end{equation}
where $\bar{\partial}_{+}\alpha_1(\xi_{+}(\eta),\eta)=\cos(\alpha_1(\xi_{+}(\eta),\eta))\partial_{\xi}\alpha_1(\xi_{+}(\eta),\eta)+
\sin(\alpha_1(\xi_{+}(\eta),\eta))\partial_{\eta}\alpha_1(\xi_{+}(\eta),\eta)$.

From $\cot(\chi(\xi_{s}(\eta),\eta))=\xi_{s}'(\eta)$ we have
\begin{equation}\label{61104}
\xi_{s}''(\eta)=-\frac{\bar{\partial}_{s}\chi(\xi_{s}(\eta),\eta)}{\sin^3\chi(\xi_{s}(\eta),\eta)}.
\end{equation}

From (\ref{61115}), (\ref{7}), (\ref{10}), (\ref{72201}),(\ref{61114}), (\ref{102002}), and (\ref{250402}) we have that at the point $\mathrm{B}$,
\begin{equation}\label{61108}
\begin{aligned}
\bar{\partial}_s \chi&=\frac{\bar{\partial}_s m}{\rho_1L_1}-\frac{N_1}{\rho_1L_1q_1^2}(q_1^2+\tau_1^2 p'(\tau_1))\bar{\partial}_s\rho_1+\frac{N_1}{q_1^2}+\bar{\partial}_s\sigma_1\\&=
-\frac{N_1}{\rho_1L_1q_1^2}(q_1^2+\tau_1^2 p'(\tau_1))\bar{\partial}_{+}\rho_1+\frac{N_1}{q_1^2}+\bar{\partial}_{+}\Big(\frac{\alpha_1+\beta_1}{2}\Big)
\\&=-\frac{\sin 2A_1}{2\rho_1}\bar{\partial}_{+}\rho_1+\frac{N_1}{q_1^2}+\frac{\bar{\partial}_{+}\alpha_1}{2}-\frac{p''(\tau_1)\tan A_1}{4c_1^2\rho_1^4}\bar{\partial}_{+}\rho_1-\frac{\sin^2 A_1}{c_1}
\\&=-\frac{\sin 2A_1}{2\rho_1}\bar{\partial}_{+}\rho_1+\frac{\bar{\partial}_{+}\alpha_1}{2}-\frac{p''(\tau_1)\tan A_1}{4c_1^2\rho_1^4}\bar{\partial}_{+}\rho_1\\&>-\frac{\sin 2A_1}{2\rho_1}\bar{\partial}_{+}\rho_1+\frac{\bar{\partial}_{+}\alpha_1}{2}
\end{aligned}
\end{equation}
and
$$
\begin{aligned}
\bar{\partial}_+ \alpha_1&=-\frac{p''(\tau_1)}{4c_1^2\rho_1^4}\left(-\frac{4p'(\tau_1)+\tau_1 p''(\tau_1)}{\tau_1 p''(\tau_1)}-\tan^2A_1\right)\sin 2A_1\bar{\partial}_{+}\rho_1\\&<
-\frac{p''(\tau_1)}{4c_1^2\rho_1^4}\left(-\frac{4p'(\tau_1)}{\tau_1 p''(\tau_1)}\right)\sin 2A_1\bar{\partial}_{+}\rho_1\\&=-\frac{\sin 2A_1}{\rho_1}\bar{\partial}_{+}\rho_1.
\end{aligned}
$$
This implies
\begin{equation}\label{61106}
\bar{\partial}_s \chi>\bar{\partial}_+ \alpha_1\quad \mbox{at}\quad \mathrm{B}.
\end{equation}

Combining (\ref{61102}) and (\ref{61106}) we know that there exists a sufficiently small $\varepsilon>0$ such that $\xi_{s}(\eta)>\xi_{+}(\eta)$ for $\eta_{_B}-\varepsilon<\eta<\eta_{_B}$.
Furthermore, by the  inequality of (\ref{61113}) we can obtain
$$
\chi(\xi_s(\eta), \eta)>\alpha_1(\xi_s(\eta), \eta)
$$
when $(\xi_s(\eta), \eta)$ lies in the interior region of $\Sigma^1$.
Thus, by an argument of continuity we get (\ref{61101}).

A direct computation yields
$$
(\cos\chi, \sin\chi)=\frac{\sin (\beta_1-\chi)}{\sin (\beta_1-\alpha_1)}(\cos\alpha_1, \sin\alpha_1)-
\frac{\sin (\alpha_1-\chi)}{\sin (\beta_1-\alpha_1)}(\cos\beta_1, \sin\beta_1).
$$
Thus, we have
\begin{equation}\label{82003}
\bar{\partial}_{s}\tau_1=\frac{\sin (\beta_1-\chi)}{\sin (\beta_1-\alpha_1)}\bar{\partial}_{+}\tau_1-\frac{\sin (\alpha_1-\chi)}{\sin (\beta_1-\alpha_1)}\bar{\partial}_{-}\tau_1.
\end{equation}

From (\ref{61114}) and (\ref{61115}) we have
\begin{equation}\label{82004}
\|\chi-(\pi+\theta)\|_{0; \wideparen{\mathrm{BG}}}\rightarrow 0 \quad \mbox{as}\quad \tau_1^e-\tau_0\rightarrow 0.
\end{equation}
Therefore, by (\ref{6811}), (\ref{82003}) and (\ref{82004}) we know that
\begin{equation}\label{82008}
\bar{\partial}_{s}\tau_1>0\quad \mbox{on}\quad \wideparen{\mathrm{BG}},
\end{equation}
as $\tau_1^e-\tau_0$ is sufficiently small.
Combining this with $\tau_1(\mathrm{B})=\tau_1^e$ and (\ref{6811}), we immediately obtain that
(\ref{72202}) holds for $\tau_1^e-\tau_0$ is sufficiently small.
This completes the proof for the existence of the post-sonic shock $\wideparen{\mathrm{BG}}$.
The back side states $(u_b, v_b, \tau_b)(\xi, \eta)$ on $\wideparen{\mathrm{BG}}$ can then be determined by (\ref{RH1}), (\ref{62603}), and (\ref{6810}).

By the symmetry, the post-sonic shock $\wideparen{\mathrm{DG}}$ can be represented by $\xi=\xi_s(-\eta)$, $\eta_{_\mathrm{D}}<\eta<0$.
Moreover, the backside states of $\wideparen{\mathrm{DG}}$ are
$$
 (u, v, \tau)=(u_b, -v_b, \tau_b)(\xi, -\eta)\quad \mbox{and}\quad (\alpha, \beta)=(2\pi-\beta_b, 2\pi-\alpha_b)(\xi, -\eta)  \quad \mbox{for}\quad  (\xi, \eta)\in \wideparen{\mathrm{DG}}.
$$

Using (\ref{7}), (\ref{10}), (\ref{72201}), (\ref{102002}), (\ref{250102}), and (\ref{250402}) we have
\begin{equation}\label{250407}
\begin{aligned}
\bar{\partial}_s \chi&=\frac{\bar{\partial}_s m}{\rho_1L_1}-\frac{N_1}{\rho_1L_1q_1^2}(q_1^2+\tau_1^2 p'(\tau_1))\bar{\partial}_s\rho_1+\frac{N_1}{q_1^2}+\bar{\partial}_s\sigma_1\\&=
-\frac{N_1}{\rho_1L_1q_1^2}(q_1^2-c_1^2)\bar{\partial}_{+}\rho_1+\frac{N_1}{q_1^2}
+\bar{\partial}_{+}\Big(\frac{\alpha_1+\beta_1}{2}\Big)
\\&=-\frac{\sin 2A_1}{2\rho_1}\bar{\partial}_{+}\rho_1-\frac{p''(\tau_1)}{4c_1^2\rho_1^4}
\Big(\underbrace{\big(\varpi(\tau_1)-\tan^2A_1\big)\sin A_1\cos A_1+\tan A_1}_{>0}\Big)\bar{\partial}_{+}\rho_1\\&>-\frac{\sin 2A_1}{2\rho_1}\bar{\partial}_{+}\rho_1>0\quad\mbox{at}\quad \mathrm{B}.
\end{aligned}
\end{equation}
Using (\ref{6811}), (\ref{250407}), and (\ref{250402}) we know that when $\tau_1^e-\tau_0$ is sufficiently small,
\begin{equation}\label{61109}
\bar{\partial}_s \chi>\frac{\tau_1^e\sin2\theta}{4}\mathcal{L}_1\quad \mbox{on}\quad \wideparen{\mathrm{BG}}.
\end{equation}
This immediately implies that when $\tau_1^e-\tau_0$ is sufficiently small $\wideparen{\mathrm{BG}}$ is convex.

Following the idea of Li and Zheng \cite{Li3},
we shall use the method of hodograph transformation to construct a solution near the backside of the post-sonic shock curve $\wideparen{\mathrm{BG}}$. In order to establish the global one-to-one inversion of the hodograph transformation, we need to establish some important monotonicity conditions about the inversion. The following estimates  are crucial to deduce the monotonicity conditions.

\begin{lem}\label{61302}
When $\tau_1^e-\tau_0$ is sufficiently small, there hold the following estimates:
\begin{equation}\label{61111}
\cos\alpha_b \bar{\partial}_s u_b+\sin\alpha_b \bar{\partial}_s v_b>0\quad \mbox{on}\quad \wideparen{\mathrm{BG}};
\end{equation}
\begin{equation}\label{61305}
\cos\beta_b \bar{\partial}_s u_b+\sin\beta_b \bar{\partial}_s v_b<0\quad \mbox{on}\quad \wideparen{\mathrm{BG}}.
\end{equation}
\end{lem}
\begin{proof}
{\bf 1.} Let $U_b=u_b-\xi$ and $V_b=v_b-\eta$. Then by (\ref{61113}) we have
\begin{equation}\label{42801}
\begin{aligned}
&\cos\alpha_b \bar{\partial}_s u_b+\sin\alpha_b \bar{\partial}_s v_b\\~=~&
\cos\alpha_b \bar{\partial}_s (U_b+\xi)+\sin\alpha_b \bar{\partial}_s (V_b+\eta)\\~=~&
\cos\alpha_b \bar{\partial}_s (q_b\cos\sigma_b)+\sin\alpha_b \bar{\partial}_s (q_b\sin\sigma_b)+\cos\alpha_b\cos\chi+\sin\alpha_b\sin\chi\\~=~&
\cos A_b \bar{\partial}_s q_b+q_b\sin A_b \bar{\partial}_s \sigma_b+1\quad \mbox{on}\quad \wideparen{BG}.
\end{aligned}
\end{equation}

From the pseudo-Bernoulli law (\ref{5802}) and recalling the first equality in (\ref{6810}), we have
\begin{equation}\label{42803}
\bar{\partial}_s q_b=\frac{\tau_b^3 p'(\tau_b)\bar{\partial}_s \rho_b}{q_b}-\frac{L_b}{q_b}=-\tau_bq_b\sin^2 A_b\bar{\partial}_s \rho_b-\cos A_b\quad \mbox{on}\quad \wideparen{\mathrm{BG}}.
\end{equation}

From (\ref{42804a}) one has
\begin{equation}\label{42804}
\bar{\partial}_s \sigma_b=\bar{\partial}_s \chi-\frac{\bar{\partial}_s m}{\rho_bL_b}+\frac{N_b}{\rho_bL_bq_b^2}(q_b^2+\tau_b^2 p'(\tau_b))\bar{\partial}_s\rho_b-\frac{N_b}{q_b^2}\quad \mbox{on}\quad \wideparen{\mathrm{BG}}.
\end{equation}

Inserting (\ref{42803}) and  (\ref{42804}) into  (\ref{42801}) and recalling the first equality in (\ref{6810}), we get
$$
\begin{aligned}
&\cos\alpha_b \bar{\partial}_s u_b+\sin\alpha_b \bar{\partial}_s v_b\\=~&
\cos A_b\left(-\tau_bq_b\sin^2 A_b\bar{\partial}_s \rho_b-\cos A_b\right)
+q_b\sin A_b\left(\bar{\partial}_s \chi-\frac{\bar{\partial}_s m}{\rho_bL_b}+\frac{\sin A_b\cos A_b}{\rho_b}\bar{\partial}_s\rho_b-\frac{N_b}{q_b^2}\right)+1\\=~&
q_b\sin A_b \left(\bar{\partial}_s \chi-\frac{\bar{\partial}_s m }{\rho_bL_b}\right) \quad \mbox{on}\quad \wideparen{\mathrm{BG}}.
\end{aligned}
$$
Therefore, by (\ref{250402}) and (\ref{61109}) we know that when $\tau_1^e-\tau_0$ is sufficiently small, the estimate (\ref{61111}) holds.

{\bf 2.} From (\ref{61109}) we have
\begin{equation}\label{82006}
\begin{aligned}
&\cos\beta_b \bar{\partial}_s u_b+\sin\beta_b \bar{\partial}_s v_b\\~=~&
\cos\beta_b \bar{\partial}_s (U_b+\xi)+\sin\beta_b \bar{\partial}_s (V_b+\eta)\\~=~&
\cos\beta_b \bar{\partial}_s (q_b\cos\sigma_b)+\sin\beta_b \bar{\partial}_s (q_b\sin\sigma_b)+\cos\beta_b\cos\chi+\sin\beta_b\sin\chi\\~=~&
\cos A_b \bar{\partial}_s q_b-q_b\sin A_b \bar{\partial}_s \sigma_b+\cos\beta_b\cos\chi+\sin\beta_b\sin\chi\\~=~&
\cos A_b \bar{\partial}_s q_b-q_b\sin A_b \bar{\partial}_s \sigma_b+\cos(2A_b)\quad \mbox{on}\quad \wideparen{\mathrm{BG}}.
\end{aligned}
\end{equation}
Inserting (\ref{42803}) and (\ref{42804}) into (\ref{82006}), we get
$$
\begin{aligned}
&\cos\beta_b \bar{\partial}_s u_b+\sin\beta_b \bar{\partial}_s v_b\\=~&
\cos A_b\left(-\tau_bq_b\sin^2 A_b\bar{\partial}_s \rho_b-\cos A_b\right)
-q_b\sin A_b\left(\bar{\partial}_s \chi-\frac{\bar{\partial}_s m}{\rho_bL_b}+\frac{\sin A_b\cos A_b}{\rho_b}\bar{\partial}_s\rho_b-\frac{N_b}{q_b^2}\right)\\&+\cos(2A_b)\\=~&
-q_b\sin A_b \left(\bar{\partial}_s \chi-\frac{\bar{\partial}_s m }{\rho_bL_b}\right)-
2\tau_bq_b\cos A_b\sin^2 A_b\bar{\partial}_s \rho_b\quad \mbox{on}\quad \wideparen{\mathrm{BG}}.
\end{aligned}
$$

Recalling $\tau_b=s_{po}(\tau_1)$ and (\ref{82008}), we know that
when $\tau_1^e-\tau_0$ is sufficiently small,
\begin{equation}\label{61308}
\bar{\partial}_s \rho_b=-\frac{1}{\tau_b^2}s_{po}'(\tau_1)\bar{\partial}_s \tau_1>0 \quad \mbox{on}\quad \wideparen{\mathrm{BG}}.
\end{equation}
Therefore, when $\tau_1^e-\tau_0$ is sufficiently small, the estimate (\ref{61305}) holds.
This completes the proof.
\end{proof}

 The estimates (\ref{61111}) and (\ref{61305}) can be also geometrically described in Figure \ref{Fig14}.

\begin{figure}[htbp]
\begin{center}
\includegraphics[scale=0.5]{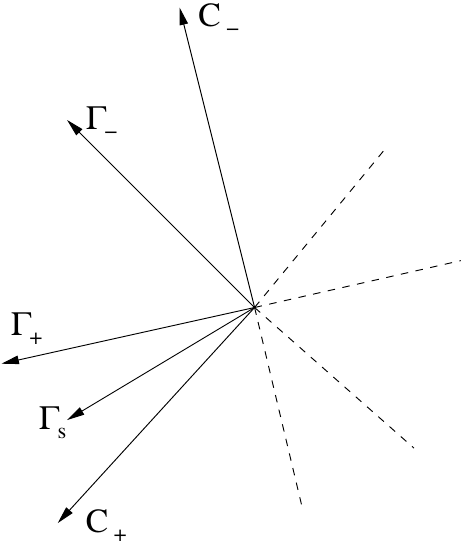}
\caption{\footnotesize Relations between the characteristic directions $(\cos\alpha_b, \sin\alpha_b)$, $(\cos\beta_b, \sin\beta_b)$, and $(\bar{\partial}_s u_b, \bar{\partial}_s v_b)$ on $\wideparen{\mathrm{BG}}$.}
\label{Fig14}
\end{center}
\end{figure}

\begin{lem}\label{lem43}
There holds
\begin{equation}\label{72203}
\lim\limits_{\tau_0\rightarrow \tau_1^e}\Big\|\big(\alpha_b-(\pi+\theta),~ \beta_b-(\pi-\theta-2\sigma_*),~u_b-u_*\sin\theta, ~v_b+u_*\cos\theta, ~\tau_b-\tau_2^e\big)\Big\|_{0; \wideparen{\mathrm{BG}}}~=~ 0,
\end{equation}
where the constants $u_*$ and $\sigma_*$ are defined in (\ref{62903}).
\end{lem}
\begin{proof}
It is easy to see that
$$
\mathrm{B}\rightarrow (c_1^e\csc\theta, 0)~\mbox{and}~
(u_b, v_b)(\mathrm{B})=(\hat{u}_2\sin\theta, -\hat{u}_2\cos\theta)\rightarrow (u_*\sin\theta, -u_*\cos\theta)\quad \mbox{as}\quad \tau_1^e-\tau_0\rightarrow0,
$$
where $\hat{u}_2$ is defined in Section 1.5.
Thus, we have
$$
\sigma_b(\mathrm{B})\rightarrow \pi-\sigma_* \quad \mbox{as}\quad \tau_1^e-\tau_0\rightarrow0.
$$
Combining this with $\alpha_b(\mathrm{B})=\pi+\theta$, we have
$$
\beta_b(\mathrm{B})\rightarrow 2(\pi-\sigma_*)-(\pi+\theta)=\pi-\theta-2\sigma_* \quad \mbox{as}\quad \tau_1^e-\tau_0\rightarrow0.
$$

From (\ref{6811}), (\ref{61113}), (\ref{61105}), (\ref{61115}),  (\ref{42804}), and (\ref{61308}) we know that $(\bar{\partial}_s\tau_b, \bar{\partial}_s\alpha_b, \bar{\partial}_s\beta_b)$ is uniformly bounded on $\wideparen{\mathrm{BG}}$ with respect to $\tau_1^e-\tau_0$.
Therefore, we obtain (\ref{72203}).
\end{proof}

From (\ref{61113}) we know that $\wideparen{\mathrm{BG}}$ and $C_{+}$ characteristic have the same direction at each point on $\wideparen{\mathrm{BG}}$. While, by (\ref{61305}) and the first equation of (\ref{form}) we know that  $\wideparen{\mathrm{BG}}$ is not a characteristic.
In what follows, we are going to use the hodograph transformation method to find a solution near the backside of the post-sonic shock $\wideparen{\mathrm{BG}}$.
We shall show  that $\wideparen{\mathrm{BG}}$ is an envelope of the $C_{+}$ characteristic curves of the flow behind it and the directional derivatives of the unknown functions in the  $C_{-}$ direction are infinity on  $\wideparen{\mathrm{BG}}$. 

By Lemmas \ref{61302} and \ref{lem43} we know that when $\tau_1^e-\tau_0$ is sufficiently small, the map
$(u, v)=(u_b, v_b)(\xi, \eta)$, $(\xi, \eta)\in \wideparen{\mathrm{BG}}$
has an inverse map
$$
(\xi, \eta)=(\xi_b, \eta_b)(u, v), \quad(u, v)\in \wideparen{\mathrm{B_2G_2}},
$$
where
$\wideparen{\mathrm{B_2G_2}}:=\big\{(u, v)|(u, v)=(u_b, v_b)(\xi, \eta), (\xi, \eta)\in \wideparen{\mathrm{BG}}\big\}$, $\mathrm{B}_2=(u_b, v_b)(\mathrm{B})$, and $\mathrm{G}_2=(u_b, v_b)(\mathrm{G})$.
We now consider (\ref{52305}) with data
\begin{equation}\label{6401}
(\alpha, \beta, \tau)=(\alpha_h, \beta_h, \tau_h)(u, v)\quad \mbox{on}\quad \wideparen{\mathrm{B_2 G_2}},
\end{equation}
where $(\alpha_h, \beta_h, \tau_h)(u, v)=(\alpha_b, \beta_b, \tau_b)\big(\xi_b(u, v), \eta_b(u,v)\big)$.

From Lemma \ref{61302} we also know that
the $\Gamma_{\pm}$ characteristic directions $(-\sin\beta, \cos\beta)$ and $(\sin\alpha, -\cos\alpha)$ on
$\wideparen{\mathrm{B_2G_2}}$ point to the same side of $\wideparen{\mathrm{B_2G_2}}$; see Figure \ref{Fig14}.
So, the curve $\wideparen{\mathrm{B_2G_2}}$ is a space-like non-characteristic curve
and
the
problem (\ref{52305}, \ref{6401}) is actually a Cauchy problem.

For $\delta>0$ we let
$$
\Gamma_{\delta}:=\big\{(u, v)\big| (u+\delta, v)\in \wideparen{\mathrm{B_2 G_2}}\big\}.
$$
\begin{lem}
There exists a small $\delta>0$ such that the Cauchy problem (\ref{52305}, \ref{6401}) admits a local solution in a domain $\Sigma_2^h(\delta)$ closed by $\Gamma_{\delta}$, a forward $\Gamma_{+}$ characteristic curve issued from $\mathrm{B}_2$, and a forward $\Gamma_{-}$ characteristic curve issued from $\mathrm{G}_2$; see Figure \ref{Fig15} (left).
\end{lem}
\begin{proof}
The local existence can be obtained by the classical characteristic method; we omit the details.
\end{proof}

\begin{figure}[htbp]
\begin{center}
\includegraphics[scale=0.65]{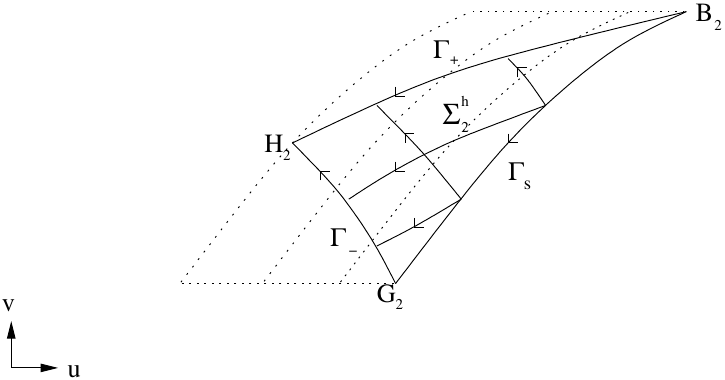}\qquad\quad\includegraphics[scale=0.32]{BD.pdf}
\caption{\footnotesize Solution in the hodograph plane and in the $(\xi, \eta)$-plane.}
\label{Fig15}
\end{center}
\end{figure}
 Next, we will extend the local solution to the whole determinate region.

\begin{lem}
(Uniform $C^1$ norm estimate) Consider the Cauchy problem (\ref{52305}, \ref{6401}). Assume $\tau_1^e-\tau_0$ is sufficiently small. Assume as well that there is a $C^1$ solution in $\Sigma_2^h(\delta)$. Then the $C^1$ norm of $\alpha$ and $\beta$ has a uniform bound, which only depends on the $C^0$ and $C^1$ norms of the Cauchy data (\ref{6401}).  That is, there is a constant $\mathcal{D}>0$, depending only the Cauchy data, but not on $\delta$, such that
$$
\|(\alpha, \beta, c)\|_{1; \Sigma_2^h(\delta)}\leq \mathcal{D}.
$$
\end{lem}
\begin{proof}
System (\ref{52305}) is actually a linearly degenerate hyperbolic system. So, the derivatives of solution will not blow up unless $A=0$ or $A=\frac{\pi}{2}$.  Integrating (\ref{52305}) along $\Gamma_{\pm}$ characteristic curves issuing from $\wideparen{\mathrm{B_2G_2}}$ and  recalling (\ref{72203}) and assumption ($\mathbf{A}2$), we know that
for any small $\varepsilon>0$ there exists an $\varrho>0$ independent of $\delta$ such that when $0<\tau_1^e-\tau_0<\varrho$, the solution satisfies $$
\|(\alpha-(\pi+\theta),~ \beta-(\pi-\theta-2\sigma_*),~\tau-\tau_2^e)\|_{1; \Sigma_2^h(\delta)}<\varepsilon.
$$
Actually, from (\ref{52305}) one can see that  $\|(\alpha-(\pi+\theta),~ \beta-(\pi-\theta-2\sigma_*),~\tau-\tau_2^e)\|_{1; \Sigma_2^h(\delta)}$ is independent of the derivatives of the Cauchy data.
So, one can obtained the existence of $\varrho$.
The derivatives of the solution can be controlled by  (\ref{52305}), (\ref{52502}), and the  $C^1$ norms of the Cauchy data (\ref{6401}).
This completes the proof.
\end{proof}

By the classical extension method, we have the following existence.
\begin{lem}
(Global existence in the hodograph plane) Assume that $\tau_1^e-\tau_0$ is sufficiently small. Then the Cauchy problem (\ref{52305}, \ref{6401}) admits a global $C^1$ solution $(\alpha, \beta, \tau)=(\alpha_2^h, \beta_2^h, \tau_2^h)(u,v)$ on a closed curved triangle domain $\Sigma_2^h$ closed by $\wideparen{\mathrm{B_2 G_2}}$, $\wideparen{\mathrm{B_2 H_2}}$, and $\wideparen{\mathrm{G_2 H_2}}$, where $\wideparen{\mathrm{B_2 H_2}}$ is a forward $\Gamma_{+}$ characteristic curve issued from $\mathrm{B}_2$ and $\wideparen{\mathrm{G_2 H_2}}$ is a forward $\Gamma_{-}$ characteristic curve issued from $\mathrm{G}_2$; see Figure \ref{Fig15}. Moreover, the solution satisfies
\begin{equation}\label{72205}
\big\|(\alpha-(\pi+\theta),~ \beta-(\pi-\theta-2\sigma_*),~\tau-\tau_2^e)\big\|_{0, \Sigma_2^h}\rightarrow0\quad \mbox{as}\quad \tau_1^e-\tau_0\rightarrow 0.
\end{equation}
\end{lem}

We now establish the global one-to-one inversion of the hodograph transformation.

\begin{lem}\label{102502}
For the $C^1$ solution of  the Cauchy problem (\ref{52305}, \ref{6401}),
there hold
\begin{equation}\label{61316}
\mathcal{Z}_{+}<0\quad \mbox{and}\quad \mathcal{Z}_{-}>0\quad \mbox{in}\quad \Sigma_2^h\setminus\wideparen{\mathrm{B_2G_2}},
\end{equation}
where the variables $\mathcal{Z}_{\pm}$ are defined in (\ref{62701}).
\end{lem}
\begin{proof}
We define  the directional derive
$$
\hat{\partial}_{_{\Gamma_s}}:=\Gamma_s^1\partial_u+\Gamma_s^2\partial_v\quad \mbox{along}\quad \wideparen{\mathrm{B_2G_2}},
$$
where $(\Gamma_s^1, \Gamma_s^2)=(\bar{\partial}_s u_b, \bar{\partial}_s v_b)$.
Then by (\ref{317}) we have
\begin{equation}\label{61309}
\hat{\partial}_{_{\Gamma_s}}=\frac{(\Gamma_s^1\cos\alpha+\Gamma_s^2\sin\alpha)\hat{\partial}_{+}}{\sin 2A}+
\frac{(\Gamma_s^1\cos\beta+\Gamma_s^2\sin\beta)\hat{\partial}_{-}}{\sin 2A}.
\end{equation}

Since $\chi=\alpha_b$ on $\wideparen{\mathrm{BG}}$, we have
$$
\hat{\partial}_{_{\Gamma_s}}\eta=(\tan\alpha_{2}^{h})\hat{\partial}_{_{\Gamma_s}}\xi \quad \mbox{on}\quad \wideparen{\mathrm{B_2G_2}}.
$$
Thus, by (\ref{61309}) we obtain
$$
\begin{aligned}
&(\Gamma_s^1\cos\alpha_{2}^{h}+\Gamma_s^2\sin\alpha_{2}^{h})\hat{\partial}_{+}\eta+
(\Gamma_s^1\cos\beta_{2}^{h}+\Gamma_s^2\sin\beta_{2}^{h})\hat{\partial}_{-}\eta\\[2pt]~=~&\Big(
(\Gamma_s^1\cos\alpha_{2}^{h}+\Gamma_s^2\sin\alpha_{2}^{h})\hat{\partial}_{+}\xi+
(\Gamma_s^1\cos\beta_{2}^{h}+\Gamma_s^2\sin\beta_{2}^{h})\hat{\partial}_{-}\xi\Big)\tan\alpha_{2}^{h}\quad \mbox{on}\quad \wideparen{\mathrm{B_2G_2}}.
\end{aligned}
$$
Inserting (\ref{38}) into this we get
$$
(\Gamma_s^1\cos\beta_{2}^{h}+\Gamma_s^2\sin\beta_{2}^{h})(\tan\beta_{2}^{h}-\tan\alpha_{2}^{h})\hat{\partial}_{-}\xi=0  \quad \mbox{on}\quad \wideparen{\mathrm{B_2G_2}}.
$$
Thus, by (\ref{61305}) we have
\begin{equation}\label{61310}
\hat{\partial}_{-}\xi=0 \quad \mbox{on}\quad \wideparen{\mathrm{B_2G_2}}.
\end{equation}
Combining this with (\ref{61312}), (\ref{61319}), and (\ref{62701}), we obtain
\begin{equation}\label{82201}
\mathcal{Z}_{-}=0\quad \mbox{on}\quad \wideparen{\mathrm{B_2G_2}}.
\end{equation}

A direct computation yields
$$
\hat{\partial}_{_{\Gamma_s}}\xi=\xi_u\bar{\partial}_s u+\xi_v\bar{\partial}_s v=\bar{\partial}_s\xi=\cos\alpha_{2}^{h}\quad \mbox{on}\quad \wideparen{\mathrm{B_2G_2}}.
$$
So, by (\ref{61312}), (\ref{61311}), (\ref{61111}), (\ref{61309}), and (\ref{61310}) we have
\begin{equation}\label{61314}
\mathcal{Z}_{+}=-\frac{2\sin^2A_{2}^{h}\sin (2A_{2}^{h})}{c_{2}^{h}(\Gamma_s^1\cos\alpha_{2}^{h}+\Gamma_s^2\sin\alpha_{2}^{h})}<0\quad \mbox{on}\quad \wideparen{\mathrm{B_2G_2}},
\end{equation}
where $A_{2}^{h}=\frac{\alpha_{2}^{h}-\beta_{2}^{h}}{2}$.

From (\ref{72205}) we know that when $\tau_1^e-\tau_0$ is sufficiently small,
$$
\frac{\tau_{2}^{h} p''(\tau_{2}^{h})}{4c(\tau_{2}^{h})p'(\tau_{2}^{h})}<0
\quad \mbox{on}\quad \Sigma_2^h.
$$
Therefore,
by integrating (\ref{322}) along the forward $\Gamma_{\pm}$ characteristic curves issued from $\wideparen{\mathrm{B_2G_2}}$ and recalling (\ref{82201}) and (\ref{61314}), we obtain (\ref{61316}).
This completes the proof.
\end{proof}

Next  we shall use the sign preserving property (\ref{61316}) to prove that the hodograph transformation is one-to-one.

From (\ref{61312})--(\ref{61319}), (\ref{72205}) and (\ref{61316}) we have that when $\tau_1^e-\tau_0$ is small,
\begin{equation}\label{61321}
\hat{\partial}_{+}\xi=-\frac{c\cos\alpha}{2\sin^2A}\mathcal{Z}_{+}<0\quad \mbox{and}\quad
\hat{\partial}_{-}\xi=\frac{c\cos\beta}{2\sin^2A}\mathcal{Z}_{-}<0 \quad \mbox{in}\quad \Sigma_2^h\setminus\wideparen{\mathrm{B_2G_2}}.
\end{equation}
Combining this with (\ref{38}) we have
\begin{equation}\label{61322}
\hat{\partial}_{+}\eta<0\quad \mbox{and}\quad
\hat{\partial}_{-}\eta>0 \quad \mbox{in}\quad \Sigma_2^h\setminus\wideparen{\mathrm{B_2G_2}}.
\end{equation}
For any two points in $\Sigma_2^h\setminus\wideparen{\mathrm{B_2G_2}}$, there exist two characteristic curves with different type connecting the two points. By (\ref{61321}) and (\ref{61322}) we know that either $\xi$ or $\eta$ is monotone along the connecting path. Thus, no two points from $(u, v)$ maps to one point in the $(\xi, \eta)$-plane.
So, the map
$$
(\xi, \eta)=\big(\xi_2^{h}, \eta_2^{h}\big)(u, v),\quad (u, v)\in \Sigma_2^h
$$
has an inverse map
$$
(u, v)=(u_2, v_2)(\xi, \eta), \quad  (\xi, \eta)\in \Sigma_2,
$$
where
$$
\xi_2^{h}(u, v):=u-\frac{c(\tau_2^{h}(u, v))\cos(\sigma_2^{h}(u, v))}{\sin (A_2^{h}(u, v))},\quad
\eta_2^{h}(u, v):=v-\frac{c(\tau_2^{h}(u, v))\sin(\sigma_2^{h}(u, v))}{\sin (A_2^{h}(u, v))},
$$
and
$$
\Sigma_2:=\Big\{(\xi, \eta) ~\big|~ (\xi, \eta)=(\xi_2^{h}(u, v), \eta_2^{h}(u, v)), (u, v)\in \Sigma_2^h \Big\}.
$$

Let
$\tau_2(\xi, \eta)=\tau_2^h(u_2(\xi, \eta), v_2(\xi, \eta)).$
We have $(u_2, v_2, \tau_2)(\xi, \eta)\in C^{1}(\Sigma_2\setminus\wideparen{\mathrm{BG}})\cap C^0(\Sigma_2)$.
Let
\begin{equation}\label{82202}
(u, v, \tau)=(u_2, v_2, \tau_2)(\xi, \eta), \quad (\xi, \eta)\in \Sigma_2.
\end{equation}
Then,
the $(u, v, \tau)$ defined in (\ref{82202}) is a solution of (\ref{42501}) in $\Sigma_2\setminus\wideparen{\mathrm{BG}}$, and
$(u, v, \tau)=(u_b, v_b, c_b)$ on $\wideparen{\mathrm{BG}}$.
The domain $\Sigma_2$ is closed by $\wideparen{\mathrm{BG}}$, $\wideparen{\mathrm{BH}}$, and $\wideparen{\mathrm{GH}}$, where $\wideparen{\mathrm{BH}}$ is a $C_{+}$ characteristic curve issued from $\mathrm{B}$ and $\wideparen{\mathrm{GH}}$ is a $C_{-}$ characteristic curve issued from $G$; see Figure \ref{Fig15}(right).
From (\ref{72205}) we also have that the solution satisfies
\begin{equation}\label{72801}
\big\|\big(\alpha-(\pi+\theta),~ \beta-(\pi-\theta-2\sigma_*),~\tau-\tau_2^e\big)\big\|_{0; \Sigma_2}\rightarrow0
\quad \mbox{as}\quad \tau_1^e-\tau_0\rightarrow 0.
\end{equation}
From (\ref{6807}), (\ref{6808}), and (\ref{61316}) we have
\begin{equation}\label{102509}
\bar{\partial}_{+}\rho\mid_{\wideparen{\mathrm{BH}}}~<~0\quad \mbox{and}\quad
\bar{\partial}_{-}\rho\mid_{\wideparen{\mathrm{GH}}}~<~0.
\end{equation}
From (\ref{6808}) and (\ref{82201}) we can see that the solution also satisfies $\bar{\partial}_{-}\rho=-\infty$ on $\wideparen{\mathrm{BH}}$.

We define
$$
\wideparen{\mathrm{D_3G_3}}:=\big\{(u,v)~\big|~ (u, -v)\in \wideparen{\mathrm{B_2G_2}}\big\},\quad \Sigma_3^h:=\big\{(u, v)~\big|~ (u, -v)\in \Sigma_2^h\big\},
$$
and
$$
(\alpha_3^h, \beta_3^h, \tau_3^h)(u, v):=\big(2\pi-\beta_2^h(u, -v), 2\pi-\alpha_2^h(u, -v), \tau_2^h(u, -v)\big),\quad (u, v)\in \Sigma_3^h.
$$
Then $(\alpha, \beta, \tau)=(\alpha_3^h, \beta_3^h, \tau_3^h)(u, v)$ satisfies (\ref{52305}) in $\Sigma_3^h$ and
$$
(\alpha, \beta, \tau)=\big(2\pi-\beta_h(u, -v), 2\pi-\alpha_h(u, -v), \tau_h(u, -v)\big)\quad \mbox{for}\quad (u, v)\in \wideparen{\mathrm{D_3 G_3}},
$$

Let
$\Sigma_3:=\big\{(\xi, \eta)\mid (\xi, -\eta)\in \Sigma_2\big\}$ and $(u_3, v_3, \tau_3)(\xi, \eta)=(u_2, -v_2, \tau_2)(\xi, -\eta)$, $(\xi, \eta)\in \Sigma_3$.
Let
\begin{equation}\label{82203}
(u, v, \tau)=(u_3, v_3, \tau_3)(\xi, \eta), \quad (\xi, \eta)\in \Sigma_3.
\end{equation}
Then by the symmetry we know that the $(u, v, \tau)$ defined in (\ref{82203}) is a solution of (\ref{42501}) in $\Sigma_3\setminus\wideparen{\mathrm{DG}}$ and
$(u, v, \tau)=(u_b, -v_b, \tau_b)(\xi,-\eta)$ on $\wideparen{\mathrm{DG}}$.
The domain $\Sigma_3$ is closed by $\wideparen{\mathrm{DG}}$, $\wideparen{\mathrm{GI}}$, and $\wideparen{\mathrm{DI}}$, where $\wideparen{\mathrm{GI}}$ is a forward $C_{+}$ characteristic curve issued from $\mathrm{G}$ and $\wideparen{\mathrm{DI}}$ is a forward $C_{-}$ characteristic curve issued from $\mathrm{D}$; see Figure \ref{Fig10}.
By the symmetry, we also have
\begin{equation}\label{102510}
\bar{\partial}_{+}\rho\mid_{\wideparen{\mathrm{GI}}}~<~0\quad \mbox{and}\quad
\bar{\partial}_{-}\rho\mid_{\wideparen{\mathrm{DI}}}~<~0.
\end{equation}
By the symmetry we also have
\begin{equation}\label{72802}
\big\|\big(\alpha-(\pi+\theta+2\sigma_*),~ \beta-(\pi-\theta),~\tau-\tau_2^e\big)\big\|_{0; \Sigma_3}\rightarrow0
\quad \mbox{as}\quad \tau_1^e-\tau_0\rightarrow 0.
\end{equation}


From (\ref{72801}) and (\ref{72802}) we also have that when $\tau_1^e-\tau_0$ is sufficiently small, $\tau>\tau_2^i$ on $\Sigma_2\cup \Sigma_3$.

\subsubsection{\bf Non-existence of a transonic shock from $\mathrm{B}$}
In this part we shall show that
${\it S}_{_\mathrm{B}}$ can not be a transonic shock. We are going to prove this by the method of contradiction.

Suppose ${\it S}_{_\mathrm{B}}$ is a smooth transonic shock.
Then ${\it S}_{_\mathrm{B}}\setminus\{\mathrm{B}\}\in \Sigma^1$ and for any point on ${\it S}_{_\mathrm{B}}\setminus\mathrm{B}$, the backward $C_{+}$ characteristic curve issued from this point stays in the angular domain between  $\wideparen{\mathrm{BE}}$ and ${\it S}_{_\mathrm{B}}$ until it intersects $\wideparen{\mathrm{BE}}$ at some point; see Figure \ref{Fig0}.
Suppose furthermore the flow between ${\it S}_{_\mathrm{B}}$ and $\wideparen{\mathrm{BE}}$ is smooth.

As in (\ref{61109}) we know that when $\tau_1^e-\tau_0$ is sufficiently small,
\begin{equation}\label{250108}
\bar{\partial}_s \chi>\frac{\tau_1^e\sin2\theta}{4}\mathcal{L}_1\quad \mbox{at}\quad \mathrm{B}.
\end{equation}

As in (\ref{82006}), one has
\begin{equation}\label{42801a}
\begin{aligned}
&\cos\beta_b \bar{\partial}_s u_b+\sin\beta_b \bar{\partial}_s v_b\\~=~&
\cos A_b \bar{\partial}_s q_b-q_b\sin A_b \bar{\partial}_s \sigma_b+\cos\beta_b\cos\chi+\sin\beta_b\sin\chi\quad \mbox{on}\quad {\it S}_{_\mathrm{B}}.
\end{aligned}
\end{equation}
Inserting (\ref{42803a}) and  (\ref{42804a}) into  (\ref{42801a}), we get
\begin{equation}\label{250405}
\begin{aligned}
&\cos\beta_b \bar{\partial}_s u_b+\sin\beta_b \bar{\partial}_s v_b\\=~&
\cos A_b\left(-\tau_bq_b\sin^2 A_b\bar{\partial}_s \rho_b-\frac{L_b}{q_b}\right)
-q_b\sin A_b\left(\bar{\partial}_s \chi-\frac{\bar{\partial}_s m}{\rho_bL_b}+\frac{N_b\cos^2 A_b}{\rho_bL_b}\bar{\partial}_s\rho_b-\frac{N_b}{q_b^2}\right)\\&+\cos\beta_b\cos\chi+\sin\beta_b\sin\chi\\=~&
-q_b\sin A_b \left(\bar{\partial}_s \chi-\frac{\bar{\partial}_s m }{\rho_bL_b}\right)-
 \frac{q_b\sin A_b\cos A_b}{\rho_b L_b}\left(N_b\cos A_b+L_b\sin A_b\right)\bar{\partial}_s \rho_b\\&+\cos\beta_b\cos\chi+\sin\beta_b\sin\chi-\frac{L_b\cos A_b-N_b\sin A_b}{q_b} \quad \mbox{on}\quad {\it S}_{_\mathrm{B}}.
\end{aligned}
\end{equation}

\begin{figure}[htbp]
\begin{center}
\includegraphics[scale=0.39]{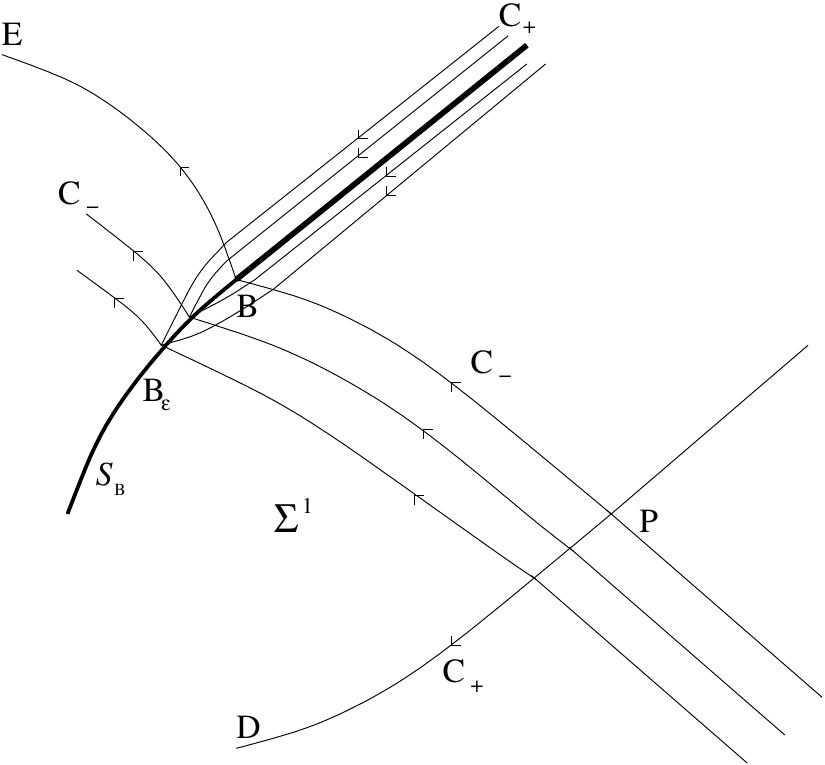}\qquad \qquad \qquad \includegraphics[scale=0.43]{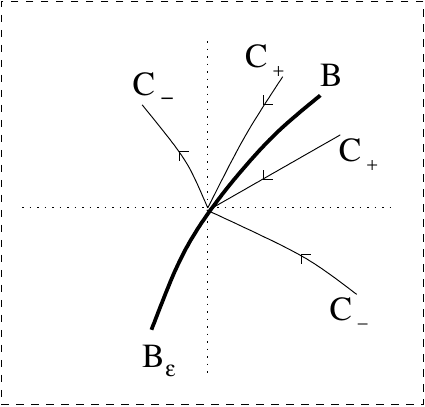}
\caption{\footnotesize An impossible transonic shock from $\mathrm{B}$.}
\label{Fig0}
\end{center}
\end{figure}

For a given small $\varepsilon>0$, we let $\mathrm{B}_{\varepsilon}$ be the point on ${\it S}_{_\mathrm{B}}$ such that $\dist(\mathrm{B}, \mathrm{B}_{\varepsilon})=\varepsilon$.
When $\varepsilon$ is sufficiently small, we have $\tau_1^e\leq \tau_1<\tau_1^i$ on  $\wideparen{\mathrm{BB_\varepsilon }}$. Since the shock curve $\wideparen{\mathrm{BB_\varepsilon }}$ is assumed to be transonic, we use Propositions \ref{pre} and \ref{tran}
to obtain
$\tau_b<\tau_2^e$ on $\wideparen{\mathrm{BB_\varepsilon}}$.
This implies
\begin{equation}\label{250406}
\bar{\partial}_{s}\rho_b\geq 0 \quad \mbox{at}\quad \mathrm{B}.
\end{equation}

Combining (\ref{250401}), (\ref{250108}), (\ref{250405}), and (\ref{250406}) we have
\begin{equation}\label{250105}
\begin{aligned}
&\cos\beta_b \bar{\partial}_s u_b+\sin\beta_b \bar{\partial}_s v_b\\=~&
-q_b\sin A_b \bar{\partial}_s \chi-
 \frac{q_b\sin A_b\cos A_b}{\rho_b L_b}\left(N_b\cos A_b+L_b\sin A_b\right)\bar{\partial}_s \rho_b\\<~&-N_b\bar{\partial}_s \chi<-\frac{c(\tau_2^e)\tau_1^e\sin2\theta}{3}\mathcal{L}_1\quad \mbox{at}\quad \mathrm{B}.
\end{aligned}
\end{equation}

Integrating the first equation of (\ref{81104}) along $C_{+}$ characteristic curves issued from $\wideparen{\mathrm{BE}}$ and recalling (\ref{62501}), (\ref{11}), and (\ref{72804}), we know that when $\varepsilon>0$ is sufficiently small, $\bar{\partial}_- u_b$ and $\bar{\partial}_- v_b$ are bounded on $\wideparen{\mathrm{B_\varepsilon B}}\setminus \mathrm{B}$.
So, by (\ref{250401}) and
(\ref{250105}) we know that when $\varepsilon>0$ is sufficiently small,
\begin{equation}\label{250403}
\begin{aligned}
&\cos\beta_b \bar{\partial}_+ u_b+\sin\beta_b \bar{\partial}_+ v_b
\\~=~&\frac{\sin (2A_b)}{\sin(\chi-\beta_b)}(\cos\beta_b \bar{\partial}_s u_b+\sin\beta_b \bar{\partial}_s v_b)\\&-
\frac{\sin (\alpha_b-\chi)}{\sin(\chi-\beta_b)}(\cos\beta_b \bar{\partial}_- u_b+\sin\beta_b \bar{\partial}_- v_b)\neq 0\quad \mbox{on}\quad \wideparen{\mathrm{B_\varepsilon B}}\setminus \mathrm{B}.
\end{aligned}
\end{equation}

Since $\wideparen{\mathrm{BB_\varepsilon}}$ is a non-characteristic curve of the flow behind it.
By the existence and uniqueness of classical solutions for the Cauchy problems of strictly hyperbolic systems, the flow behind $\wideparen{\mathrm{BB_\varepsilon}}$ is smooth and satisfies the characteristic equations (\ref{form}). While, by (\ref{250403}) we know that the equation $\bar{\partial}_{+}u+\lambda_{-}\bar{\partial}_{+}v
  =0$ does not hold on $\wideparen{\mathrm{B_\varepsilon B}}$. This leads to a contraction. So, the shock wave from $\mathrm{B}$ can not be a transonic shock.

\subsubsection{\bf Non-existence of a pre-sonic shock from $\mathrm{B}$}
In this part we shall show that
${\it S}_{_\mathrm{B}}$ can not be a pre-sonic shock.

Suppose that ${\it S}_{_\mathrm{B}}$ is pre-sonic.
 Then ${\it S}_{_\mathrm{B}}$  lies outside the domain $\Sigma^1$, and for any point on ${\it S}_{_\mathrm{B}}\setminus\mathrm{B}$, the backward $C_{+}$ characteristic curve issued from this point stays in the angular domain between  $\wideparen{\mathrm{BE}}$ and ${\it S}_{_\mathrm{B}}$ until it intersects $\wideparen{\mathrm{BE}}$ at some point; see Figure \ref{Fig00}.
Suppose furthermore that the flow between ${\it S}_{_\mathrm{B}}$ and $\wideparen{\mathrm{BE}}$ is smooth.

\begin{figure}[htbp]
\begin{center}
\includegraphics[scale=0.38]{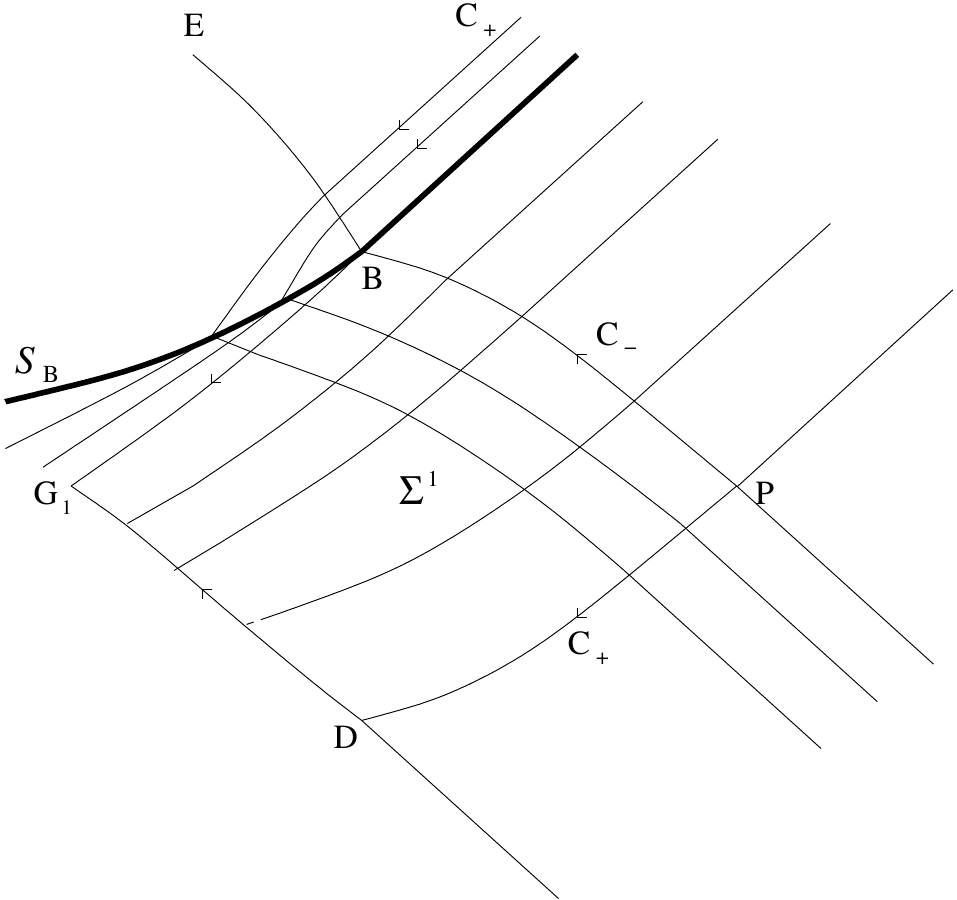}
\caption{\footnotesize An impossible pre-sonic shock from $\mathrm{B}$.}
\label{Fig00}
\end{center}
\end{figure}
We denote by
$\wideparen{\mathrm{P_hD_h}}$ and $\wideparen{\mathrm{P_hB_h}}$ the images of $\wideparen{\mathrm{PD}}$ and $\wideparen{\mathrm{PB}}$ on the $(u, v)$-plane, respectively.
Obviously, we have
$$
\mathcal{Z}_{+}\mid_{\wideparen{\mathrm{P_hD_h}}}~=~\frac{\tau p''(\tau)\tan A}{2cp'(\tau)}<0\quad\mbox{and}\quad \mathcal{Z}_{-}\mid_{\wideparen{\mathrm{P_hB_h}}}~=~-\frac{\tau p''(\tau)\tan A}{2cp'(\tau)}>0.
$$
Integrating $\hat{\partial}_{-}\mathcal{Z}_{+}+\mathcal{W}\mathcal{Z}_{+}=\displaystyle\frac{\tau p''(\tau)}{4cp'(\tau)}(\tan^2 A+1)\mathcal{Z}_{-}$ along $\wideparen{\mathrm{P_hB_h}}$ from $\mathrm{P_h}$ to $\mathrm{B_h}$,  we get $\mathcal{Z}_{+}(\mathrm{B_h})<0$.
Here, $\mathrm{P_h}=(0, 0)$ and $\mathrm{B_h}=(\sin\theta, -\cos\theta)\hat{u}_{r}(\hat{\xi}_1)$, where the function $\hat{u}_{r}(\hat{\xi})$ is defined in Sec. 1.4.
As in (\ref{61314}), we have
$$
\mathcal{Z}_{+}(\mathrm{B_h})=-\frac{2\sin^2 A_f \sin (2A_f)}{c_1^e (\cos\alpha_f\partial_su_f+\sin\alpha_f\partial_sv_f)}\Bigg|_{\mathrm{B}}.
$$
Thus, one has
$$
\cos\alpha_f\partial_su_f+\sin\alpha_f\partial_sv_f>0\quad \mbox{at}\quad \mathrm{B}.
$$

We compute the directional derivatives of the front side states of ${\it S}_{_\mathrm{B}}$ and use (\ref{102002}) and (\ref{250402}) to get
$$
\begin{aligned}
0<&\cos\alpha_f \bar{\partial}_s u_f+\sin\alpha_f \bar{\partial}_s v_f\\=~&
\cos A_f\left(-\tau_fq_f\sin^2 A_f\bar{\partial}_s \rho_f-\cos A_f\right)\\&
+q_f\sin A_f\left(\bar{\partial}_s \chi-\frac{\bar{\partial}_s m}{\rho_fL_f}+\frac{\sin A_f\cos A_f}{\rho_f}\bar{\partial}_s\rho_f-\frac{N_f}{q_f^2}\right)+1\\=~&
q_f\sin A_f \bar{\partial}_s \chi\quad \mbox{at}\quad \mathrm{B}.
\end{aligned}
$$

Since the shock is assumed to be pre-sonic, by the Liu's extended condition we have $\bar{\partial}_{+}\rho_b\geq 0$ at $\mathrm{B}$.
Then, we compute  the directional derivatives of the backside states of ${\it S}_{_\mathrm{B}}$ to get
$$
\begin{aligned}
\bar{\partial}_s \chi&=\frac{\bar{\partial}_s m}{\rho_bL_b}-\frac{N_b}{\rho_bL_bq_b^2}(q_b^2+\tau_b^2 p'(\tau_b))\bar{\partial}_s\rho_b+\frac{N_b}{q_b^2}+\bar{\partial}_s\sigma_b\\&=
-\frac{N_b}{\rho_bL_bq_b^2}(q_b^2+\tau_b^2 p'(\tau_b))\bar{\partial}_{+}\rho_b+\frac{N_b}{q_b^2}+\bar{\partial}_{+}\Big(\frac{\alpha_b+\beta_b}{2}\Big)
\\&=-\left(\frac{\sin 2A_b}{2\rho_b}+ \frac{p''}{4c_b^2\rho_b^4}(1+\tan^2A_b)\sin2A_b\right)\bar{\partial}_{+}\rho_b\leq0\quad \mbox{at}\quad \mathrm{B}.
\end{aligned}
$$
This leads to a contradiction. So, the shock ${\it S}_{_\mathrm{B}}$ can not be pre-sonic.

Similarly, one can prove that  the shock ${\it S}_{_\mathrm{B}}$ can not be double-sonic.

\subsection{Discontinuous Goursat-type problem}
\subsubsection{\bf Discontinuous Goursat-type problem}
In order to obtain the flow between $\wideparen{\mathrm{GH}}$ and $\wideparen{\mathrm{GI}}$,
we consider (\ref{42501}) with data
\begin{equation}\label{61602}
(u, v, \tau)=\left\{
               \begin{array}{ll}
                 (u_2, v_2, \tau_2)(\xi, \eta), & \hbox{$(\xi, \eta)\in\wideparen{\mathrm{GH}}$;} \\[3pt]
                 (u_3, v_3, \tau_3)(\xi, \eta), & \hbox{$(\xi, \eta)\in\wideparen{\mathrm{GI}}$;}
               \end{array}
             \right.
\end{equation}
see Figure \ref{Fig16} (left).

\begin{figure}[htbp]
\begin{center}
\includegraphics[scale=0.46]{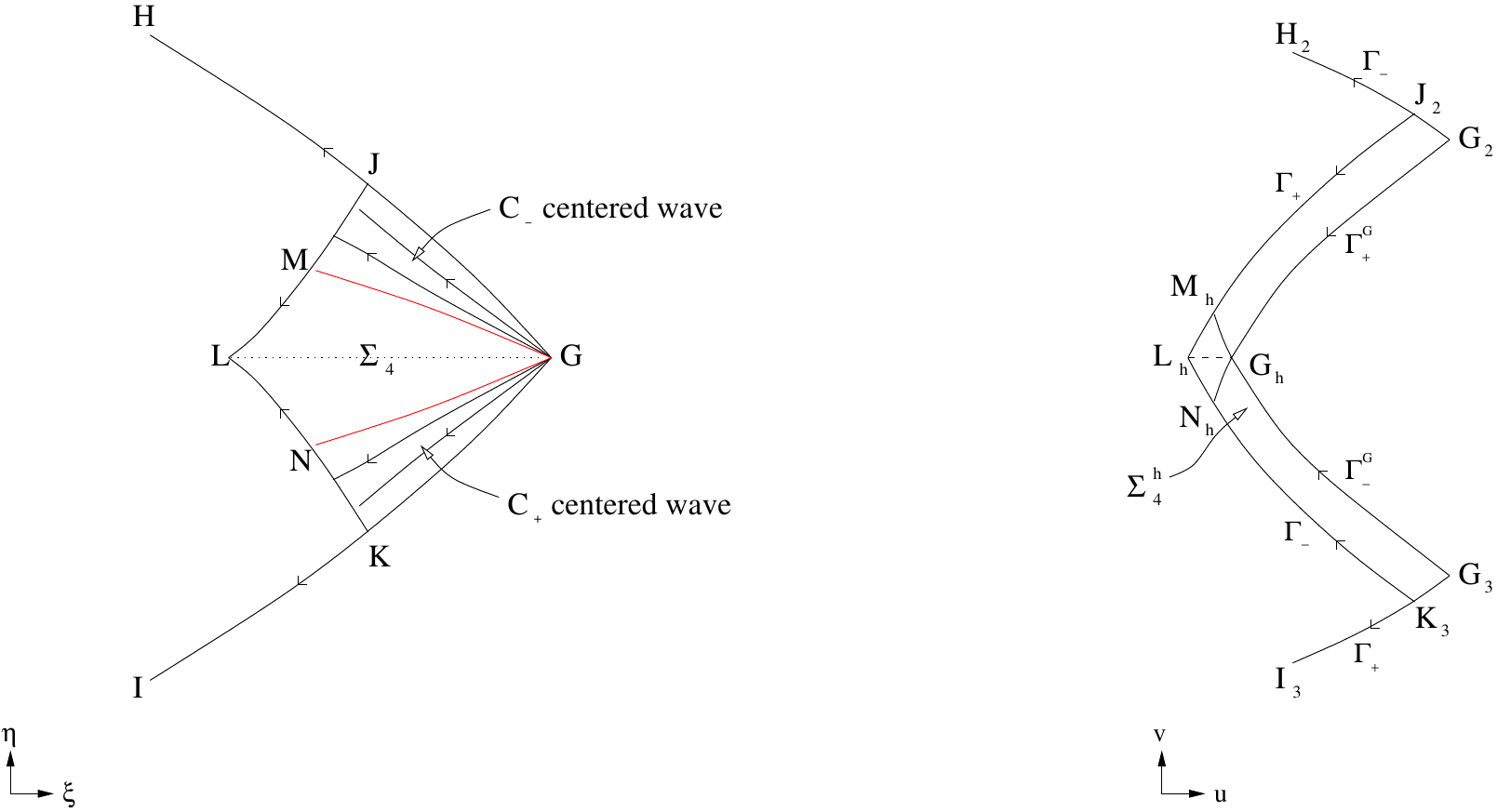}
\caption{\footnotesize The DGP (\ref{42501}, \ref{61602}). Left: $(\xi,\eta)$-plane; right: hodograph plane.}
\label{Fig16}
\end{center}
\end{figure}

Let $\xi_{_\mathrm{G}}=\xi_s(0)$, where the function $\xi_{s}(\eta)$ is defined in (\ref{61102}).
Then by the result of Section 4.3.1 we know that the coordinate of the point $\mathrm{G}$ is $(\xi_{_\mathrm{G}}, 0)$.
For the convenience of the following discussion,
we define the following constants:
\begin{equation}\label{61601}
(u_g, v_g, \tau_g):=(u_2, v_2, \tau_2)(\xi_{_G}, 0), \quad
q_{g}:=\sqrt{\big(u_g-\xi_{_G}\big)^{2}+v_g^{2}},\quad \sigma_{g}:=\arctan\left|\frac{v_g}{u_g-\xi_{_G}}\right|.
\end{equation}
From (\ref{72801}) one sees that $\tau_g>\tau_2^i$ as $\tau_1^e-\tau_0$ is sufficiently small.

By the definition of the function $\xi_s(\eta)$, one knows $\xi_{_\mathrm{G}}\rightarrow c_1^e\csc\theta$ as $\tau_1^e-\tau_0\rightarrow 0$.
Meanwhile, in view of (\ref{72203}) and the definition of $u_*$, we know that
$v_g>0$ and $u_g<\xi_{_\mathrm{G}}$ as $\tau_1^e-\tau_0$ is sufficiently small.
This implies
\begin{equation}\label{82303}
\sigma_2(\mathrm{G})=\pi-\sigma_{g}\quad \mbox{and}\quad
\sigma_3(\mathrm{G})=\pi+\sigma_{g}.
\end{equation}
From (\ref{72203}) we also have
\begin{equation}\label{72803}
(q_g, \tau_g, \sigma_{g})\rightarrow (q_*, \tau_2^e, \sigma_*)\quad \mbox{as}\quad \tau_1^e-\tau_0\rightarrow 0.
\end{equation}

From $v_2(\mathrm{G})=-v_3(\mathrm{G})$ and $v_2(\mathrm{G})>0$, one sees that
the problem (\ref{42501}, \ref{61602}) is a discontinuous Goursat-type boundary value problem (DGP).
This problem can be also seen as a generalized Riemann problem proposed in Li and Yu \cite{Li-Yu}.
Actually, due to $(u_2, v_2, \tau_2)(\mathrm{G})=(u_3, -v_3, \tau_3)(\mathrm{G})$ and $v_2(\mathrm{G})>0$, there will be two centered waves issued from the point $\mathrm{G}$ and the problem (\ref{42501}, \ref{61602}) admits a local shock-free piecewise smooth solution.

\subsubsection{\bf Centered waves}
For the DGP (\ref{42501}, \ref{61602}), there is a multi-valued singularity at the point $\mathrm{G}$. So, we need to use the concept of centered waves for first order hyperbolic systems (see Li and Yu \cite{Li-Yu}, pp. 188-190).
We present the definition of centered waves for the system (\ref{42501}).
\begin{defn}\label{defn2}
  Let $\Delta(t)$ be an angular domain with curved boundaries:
\begin{equation}
\Delta(t):=\big\{(\xi,\eta)\mid \xi_{_\mathrm{G}}-t\leq \xi\leq \xi_{_\mathrm{G}},~ \eta_{2}(\xi)\leq \eta\leq \eta_{1}(\xi)\big\},
\end{equation}
where
$\eta_{1}(\xi_{_G})=\eta_{2}(\xi_{_\mathrm{G}})=0$
and
$\eta_{1}'(\xi_{_G})<\eta_{2}'(\xi_{_\mathrm{G}})$; see Figure \ref{Fig17}.

A function $(u, v, \tau)(\xi,\eta)$ is called a $C_{+}$  centered wave solution of (\ref{42501}) with $\mathrm{G}$ as the center point if the following properties are satisfied:
\begin{enumerate}
  \item $(u, v, \tau)$ can be implicitly determined by the functions
$\eta=j(\xi,\nu)$ and
$(u, v, \tau)=(\hat{u},\hat{v}, \hat{\tau})(\xi,\nu)$
defined on a rectangular domain
$T(t):=\{(\xi,\nu)\mid \xi_{_\mathrm{G}}-t\leq \xi\leq\xi_{_\mathrm{G}}, \nu_1\leq \nu\leq \nu_2\}$ where $\nu_1=\eta_{1}'(\xi_{_\mathrm{G}})$ and $\nu_2=\eta_{2}'(\xi_{_\mathrm{G}})$.
Moreover, $j$ and $(\hat{u}, \hat{v}, \hat{\tau})$ belong to $C^{1}(T(t))$, and for any $(\xi,\nu)\in T(t)\setminus\{\xi=\xi_{_\mathrm{G}}\}$ there holds
$j_{\nu}(\xi,\nu)<0$.
\item The function $(u, v, \tau)(\xi,\eta)$ defined above satisfies (\ref{42501}) on $\Delta(t)\setminus \mathrm{G}$.
\item For any fixed $\nu\in[\nu_1,\nu_2]$, $\eta=j(\xi,\nu)$ gives a $C_{+}$ characteristic curve issued from $\mathrm{G}$ with the slope $\nu$ at $\mathrm{G}$, i.e.,
\begin{equation}\label{pgz1}
j_{\xi}(\xi, \nu)=\hat{\lambda}_{+}(\xi, \nu),\quad
    j(\xi_{_\mathrm{G}},\nu)=0,
\quad
j_{\xi}(\xi_{_\mathrm{G}},\nu)=\nu.
\end{equation}
Here, $\hat{\lambda}_{\pm}(\xi, \nu)=\lambda_{\pm}\big(\hat{u}(\xi, \nu)-\xi, \hat{v}(\xi, \nu)-j(\xi, \nu), \hat{c}(\xi, \nu)\big)$ and $\hat{c}(\xi, \nu)=c(\hat{\tau}(\xi, \nu))$.
\item $\nu=\nu_1$ and $\nu=\nu_2$ correspond to $\eta=\eta_1(\xi)$ and $\eta=\eta_2(\xi)$, respectively.
\end{enumerate}
Let $(\widetilde{u}_{+},\widetilde{v}_{+}, \widetilde{\tau}_{+})(\nu)=(\hat{u},\hat{v}, \hat{\tau})(\xi_{_\mathrm{G}},\nu)$, $\nu_1\leq \nu\leq \nu_2$.
 Then, $(u, v, \tau)=(\widetilde{u}_{+},\widetilde{v}_{+}, \widetilde{\tau}_{+})(\nu)$ is called the principal part of this $C_{+}$ centered wave and $\nu_2-\nu_1$ the amplitude of the centered wave.
\end{defn}

\begin{figure}[htbp]
\begin{center}
\includegraphics[scale=0.55]{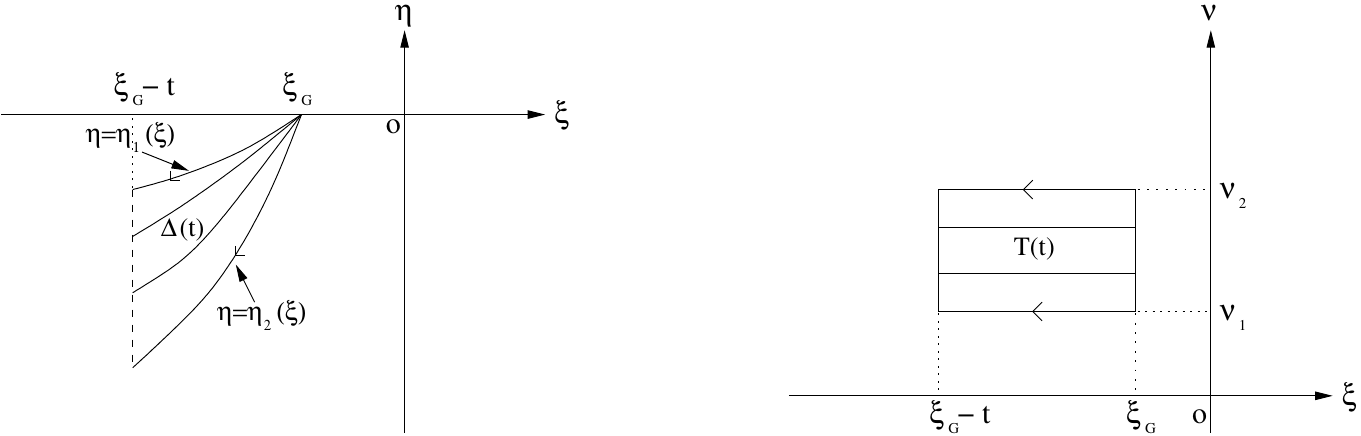}
\caption{\footnotesize A $C_{+}$ type centered wave.}
\label{Fig17}
\end{center}
\end{figure}

By the symmetry, we can define a $C_{-}$ centered wave of (\ref{42501}) with $\mathrm{G}$ as the center point.

\subsubsection{\bf Principal parts of the centered waves}
In order to use the hodograph transformation method to construct a local solution to the DGP (\ref{42501}, \ref{61602}), we need to construct the principal parts of the $C_{\pm}$ centered waves issued from the point $\mathrm{G}$.

In view of the values of $\alpha$ and $\beta$ on $\wideparen{\mathrm{GH}}\cup\wideparen{\mathrm{GI}}$, we
set
\begin{equation}\label{1802}
\hat{\alpha}(\xi, \nu)=\pi+\arctan\big(\hat{\lambda}_{+}(\xi, \nu)\big),\quad
\hat{\beta}(\xi, \nu)=\pi+\arctan\big(\hat{\lambda}_{-}(\xi, \nu)\big).
\end{equation}
Thus,
from (\ref{pgz1}) we  have that for the $C_{+}$ centered wave defined above,
\begin{equation}
\frac{\partial j(\xi,\nu)}{\partial \nu}=\int_{\xi_{_\mathrm{G}}}^{\xi}\sec^{2}(\hat{\alpha}
(\xi,\nu))\frac{\partial \hat{\alpha}(\xi, \nu)}{\partial \nu} d\xi.
\end{equation}

Introducing the transformation $\xi=\xi$, $\eta=j(\xi, \nu)$,  we have
\begin{equation}\label{transd}
\frac{\partial}{\partial \xi}= \frac{\partial}{\partial \xi}-\frac{\partial j}{\partial \xi} \Big(\frac{\partial j}{\partial \nu}\Big)^{-1}\frac{\partial}{\partial \nu}\quad\mbox{and} \quad
\frac{\partial}{\partial \eta}= \Big(\frac{\partial j}{\partial \nu}\Big)^{-1}\frac{\partial}{\partial \nu}.
\end{equation}
Inserting this into (\ref{42501}) and (\ref{5802}) and letting $\xi\rightarrow\xi_{_\mathrm{G}}$, we know that the principal part of the $C_{+}$ centered wave satisfies
\begin{equation}\label{principal}
\left\{
  \begin{array}{ll}
   \displaystyle -\nu\widetilde{\rho}_{+}\frac{{\rm d} \widetilde{U}_{+}}{{\rm d} \nu}-
\nu\widetilde{U}_{+}\frac{{\rm d} \widetilde{\rho}_{+}}{{\rm d} \nu} +\widetilde{\rho}_{+}\frac{{\rm d} \widetilde{V}_{+}}{{\rm d} \nu}+
\widetilde{V}_{+}\frac{{\rm d} \widetilde{\rho}_{+}}{{\rm d} \nu}=0,\\[6pt]
 \displaystyle \frac{{\rm d} \widetilde{U}_{+}}{{\rm d}\nu}+\nu\frac{{\rm d} \widetilde{V}_{+}}{{\rm d}\nu}=0,
  \end{array}
\right.
\end{equation}
and
\begin{equation}\label{principal1}
 \frac{\widetilde{U}_{+}^2+\widetilde{V}_{+}^2}{2}+h(\widetilde{\tau}_{+})+\varphi(\mathrm{G})=0,
\end{equation}
where $\widetilde{U}_{+}=\widetilde{u}_{+}(\nu)-\xi_{_\mathrm{G}}$, $\widetilde{V}_{+}=\widetilde{v}_{+}(\nu)$, and $\widetilde{c}_{+}(\nu)=c(\widetilde{\tau}_{+}(\nu))$.

Inserting (\ref{transd}) into (\ref{82301})--(\ref{8}) and letting $\xi\rightarrow\xi_{_\mathrm{G}}$, we also have
\begin{equation}\label{8304}
\widetilde{c}_{+}\frac{{\rm d}  \widetilde{\alpha}_{+}}{{\rm d} \nu}=\frac{p''(\widetilde{\tau}_{+})}{2\widetilde{c}_{+}\widetilde{\rho}_{+}^4}\tan \widetilde{A}_{+}\frac{{\rm d}  \widetilde{\rho}_{+}}{{\rm d} \nu}
\end{equation}
and
\begin{equation}\label{8303}
\widetilde{c}_{+}\frac{{\rm d}  \widetilde{\beta}_{+}}{{\rm d} \nu}=\frac{p''(\widetilde{\tau}_{+})}{4\widetilde{c}_{+}\widetilde{\rho}_{+}^4}
\big(\varpi(\widetilde{\tau}_{+})-\tan^2\widetilde{A}_{+}\big)\sin2 \widetilde{A}_{+}\frac{{\rm d}  \widetilde{\rho}_{+}}{{\rm d} \nu},
\end{equation}
where $(\widetilde{\alpha}_{+},\widetilde{\beta}_{+})(\nu)=(\hat{\alpha}, \hat{\beta})(\xi_{_\mathrm{G}},\nu)$ and $\widetilde{A}_{+}=\frac{\widetilde{\alpha}_{+}-\widetilde{\beta}_{+}}{2}$.

From (\ref{pgz1}) and (\ref{1802}) we have
\begin{equation}\label{82501}
\widetilde{\alpha}_{+}(\nu)=\pi+\arctan\nu.
\end{equation}
  Thus, by (\ref{8304}) we have
\begin{equation}\label{120601}
 \frac{{\rm d}\widetilde{\rho}_{+}}{{\rm d} \nu}>0\quad \mbox{if}\quad \widetilde{\tau}_{+}>\tau_2^i.
\end{equation}
In addition, we have
\begin{equation}\label{82502}
\widetilde{c}_{+}\frac{{\rm d}  \widetilde{\sigma}_{+}}{{\rm d} \nu}=\frac{p''(\widetilde{\tau}_{+})}{8\widetilde{c}_{+}\widetilde{\rho}_{+}^4}
\Big[\big(\varpi(\widetilde{\tau}_{+})-\tan^2\widetilde{A}_{+}\big)\sin2 \widetilde{A}_{+}+2\tan \widetilde{A}_{+}\Big]\frac{{\rm d}  \widetilde{\rho}_{+}}{{\rm d} \nu}>0 \quad \mbox{if}\quad \widetilde{\tau}_{+}>\tau_2^i,
\end{equation}
where $\widetilde{\sigma}_{+}=\frac{\widetilde{\alpha}_{+}+\widetilde{\beta}_{+}}{2}$.

Let $(\widetilde{u}_{-},\widetilde{v}_{-}, \widetilde{\tau}_{-})(\nu)$ be the principal part of a $C_{-}$ centered wave
with $\mathrm{G}$ as the centered point.
Similarly, we have
\begin{equation}\label{82306}
\left\{
  \begin{array}{ll}
   \displaystyle -\nu\widetilde{\rho}_{-}\frac{{\rm d} \widetilde{U}_{-}}{{\rm d} \nu}-
\nu\widetilde{U}_{-}\frac{{\rm d} \widetilde{\rho}_{-}}{{\rm d} \nu} +\widetilde{\rho}_{-}\frac{{\rm d} \widetilde{V}_{-}}{{\rm d} \nu}+
\widetilde{V}_{-}\frac{{\rm d} \widetilde{\rho}_{-}}{{\rm d} \nu}=0,\\[6pt]
 \displaystyle \frac{{\rm d} \widetilde{U}_{-}}{{\rm d}\nu}+\nu\frac{{\rm d} \widetilde{V}_{-}}{{\rm d}\nu}=0,
  \end{array}
\right.
\end{equation}
and
\begin{equation}\label{82307}
 \frac{\widetilde{U}_{-}^2+\widetilde{V}_{-}^2}{2}+h(\widetilde{\tau}_{-})+\varphi(\mathrm{G})=0,
\end{equation}
where $\widetilde{U}_{-}=\widetilde{u}_{-}(\nu)-\xi_{_G}$, $\widetilde{V}_{-}=\widetilde{v}_{-}(\nu)$, and $\widetilde{c}_{-}(\nu)=c(\widetilde{\tau}_{-}(\nu))$.
Moreover, from (\ref{7}) and (\ref{10}) we have
\begin{equation}\label{8306}
\widetilde{c}_{-}\frac{{\rm d}  \widetilde{\beta}_{-}}{{\rm d} \nu}=-\frac{p''(\widetilde{\tau}_{-})}{2\widetilde{c}_{-}\widetilde{\rho}_{-}^4}\tan \widetilde{A}_{-}\frac{{\rm d}  \widetilde{\rho}_{-}}{{\rm d} \nu}
\end{equation}
and
\begin{equation}\label{8305}
\widetilde{c}_{-}\frac{{\rm d}  \widetilde{\alpha}_{-}}{{\rm d} \nu}=-\frac{p''(\widetilde{\tau}_{-})}{4\widetilde{c}_{-}\widetilde{\rho}_{-}^4}
\big(\varpi(\widetilde{\tau}_{-})-\tan^2\widetilde{A}_{-}\big)\sin2 \widetilde{A}_{-}\frac{{\rm d}  \widetilde{\rho}_{-}}{{\rm d} \nu},
\end{equation}
where  $\widetilde{A}_{-}=\frac{\widetilde{\alpha}_{-}-\widetilde{\beta}_{-}}{2}$.
As in (\ref{82501}) we have $\widetilde{\beta}_{-}(\nu)=\pi+\arctan\nu$.
So, by (\ref{8306}) we have
\begin{equation}\label{010901}
 \frac{{\rm d}\widetilde{\rho}_{-}}{{\rm d} \nu}<0\quad \mbox{if}\quad \widetilde{\tau}_{-}>\tau_2^i.
\end{equation}

From (\ref{823301}), (\ref{52204}), and the second equation of (\ref{principal}) we know that
the images of the principal part of a $C_{+}$  centered wave in the $(u, v)$-plane is a $\Gamma_{-}$ characteristic curve.
Similarly, the images of the principal part of a $C_{-}$  centered wave in the $(u, v)$-plane is a $\Gamma_{+}$ characteristic curve.  The $\Gamma_{\pm}$ curves are actually the epicycloids of the 2D isentropic irrotational steady flow for polytropic gases; see Courant and Friedrichs \cite{CaF}.
In what follows, we are going to construct the principal parts of the $C_{\pm}$ centered waves of the DGP (\ref{42501}, \ref{61602}).

Since the outer border of the $C_{+}$ centered wave to the DGP (\ref{42501}, \ref{61602}) is $\wideparen{\mathrm{GI}}$, we
set
 \begin{equation}\label{82904}
\nu_2=\tan(\alpha_3(\mathrm{G}))
\end{equation}
and
\begin{equation}\label{82304}
(\widetilde{u}_{+},\widetilde{v}_{+}, \widetilde{\rho}_{+})(\nu_2)=(u_g, -v_g, \rho_g),
\end{equation}
where $\rho_g=1/\tau_g$.
In order to obtain the principal part of the $C_{+}$ centered wave to the  problem (\ref{42501}, \ref{61602}), we  consider (\ref{principal})--(\ref{principal1}) with  initial data (\ref{82304}).

\begin{lem}
When $\tau_1^e-\tau_0$ is sufficiently small, there exists a unique $\nu_1\in (0, \nu_2)$ such that the initial value problem  (\ref{principal})--(\ref{principal1}), (\ref{82304}) admits a solution on $[\nu_1, \nu_2]$ and the solution satisfies $\widetilde{v}_{+}(\nu_1)=0$,
$\widetilde{\rho}_{+}(\nu_1)>0$, and $\widetilde{v}_{+}'(\nu)<0$ and $\widetilde{\rho}_{+}'(\nu)>0$ for $\nu\in (\nu_1, \nu_2]$.
\end{lem}
\begin{proof}
From (\ref{8303}) and (\ref{8304}) we can see that the initial value problem  (\ref{principal})--(\ref{principal1}), (\ref{82304}) admits a local solution.
The system (\ref{principal}) has a solution provided that $\widetilde{\rho}_{+}>0$ and $\widetilde{A}_{+}<\frac{\pi}{2}$.
In what follows we only discuss the solution for $\nu<\nu_2$.

We introduce the Riemann invariants:
$$
R_{+}(\widetilde{U}, \widetilde{V}):=\underbrace{\pi+\arctan\Big(\frac{\widetilde{V}}{\widetilde{U}}\Big)}_{\widetilde{\sigma}}+ \int_{q_{g}}^{\widetilde{q}}\frac{\sqrt{q^2-c^2(\check{\tau}(q))}}{qc(\check{\tau}(q))}dq,
$$
where $\widetilde{q}=\sqrt{\widetilde{U}^2+\widetilde{V}^2}$ and $\check{\tau}(q)$ is determined by
\begin{equation}\label{102501}
\frac{1}{2}q^2+\int_{\tau_{g}}^{\check{\tau}(q)}\tau p'(\tau)~{\rm d}\tau=\frac{1}{2}q_{g}^2.
\end{equation}
By the second equation of (\ref{principal}), we have
$$
\begin{array}{rcl}
&&\displaystyle\frac{{\rm d}}{{\rm d} \nu} R_{+}\big(\widetilde{U}_{+}(\nu), \widetilde{V}_{+}(\nu)\big)
\\[8pt]&=&\displaystyle\frac{\widetilde{U}_{+}\frac{{\rm d}\widetilde{V}_{+}}{{\rm d} \nu}-\widetilde{V}_{+}\frac{{\rm d}\widetilde{U}_{+}}{{\rm d} \nu}}{\widetilde{q}_{+}^2}
+\frac{\sqrt{\widetilde{q}_{+}^2-\widetilde{c}_{+}^2}}{\widetilde{q}_{+}\widetilde{c}_{+}}\cdot\frac{\widetilde{U}_{+}\frac{{\rm d}\widetilde{U}_{+}}{{\rm d} \nu}+\widetilde{V}_{+}\frac{{\rm d}\widetilde{V}_{+}}{{\rm d} \nu}}{\widetilde{q}_{+}}
\\[8pt]&=&\displaystyle\frac{1}{\widetilde{q}_{+}}\left(\cos\widetilde{\sigma}_{+}\frac{{\rm d}\widetilde{V}_{+}}{{\rm d} \nu}-\sin\widetilde{\sigma}_{+}\frac{{\rm d}\widetilde{U}_{+}}{{\rm d} \nu}+\cot\widetilde{A}_{+}\Big(
\cos\widetilde{\sigma}_{+}\frac{{\rm d}\widetilde{U}_{+}}{{\rm d} \nu}+\sin\widetilde{\sigma}_{+}\frac{{\rm d}\widetilde{V}_{+}}{{\rm d} \nu}\Big)\right)
\\[8pt]&=&\displaystyle\frac{1}{\widetilde{q}_{+}\sin\widetilde{A}_{+}}\bigg(\cos\widetilde{\alpha}_{+}\frac{{\rm d}\widetilde{U}_{+}}{{\rm d} \nu}+\sin\widetilde{\alpha}_{+}\frac{{\rm d}\widetilde{V}_{+}}{{\rm d} \nu}\bigg)~=~0
\end{array}
$$
where $\widetilde{q}_{+}=\sqrt{\widetilde{U}_{+}^2+\widetilde{V}_{+}^2}$.

In addition, by (\ref{82303}) we have
$$
R_{+}(u_g-\xi_{_\mathrm{G}}, -v_g)=\sigma_3(\mathrm{G})=\widetilde{\sigma}_{+}(\nu_2)=\pi+\sigma_g.
$$
Thus, we have
\begin{equation}\label{82901}
R_{+}\big(\widetilde{U}_{+}(\nu), \widetilde{V}_{+}(\nu)\big)\equiv\pi+\sigma_g.
\end{equation}
Combining this with (\ref{principal1}), (\ref{120601}), and (\ref{82502}), we have
\begin{equation}\label{82903}
\widetilde{q}_{+}'(\nu)<0,\quad \widetilde{\sigma}_{+}'(\nu)>0,\quad  \mbox{and}\quad \widetilde{\rho}_{+}'(\nu)>0\quad  \mbox{for}\quad \nu<\nu_2.
\end{equation}

From (\ref{8304})--(\ref{82501}) we have
\begin{equation}\label{102502a}
\frac{{\rm d}  \widetilde{A}_{+}}{{\rm d} \nu}=
\big(1-\varpi(\widetilde{\tau}_{+})\cos^2 \widetilde{A}_{+}+\sin^2\widetilde{A}_{+}\big)\frac{1}{2(1+\nu^2)}.
\end{equation}
Combining this with (\ref{82903}) and $\tau_g>\tau_2^i$, we know that the solution of the initial value problem  (\ref{principal})--(\ref{principal1}), (\ref{82304}) satisfies
\begin{equation}\label{102410}
\widetilde{A}_{+}(\nu)\leq\max\left\{\widetilde{A}_{+}(\nu_2), ~\arccos \sqrt{\sup\limits_{\tau\in (\tau_g, +\infty)}\varpi(\tau)+1}\right\}<\frac{\pi}{2}\quad \mbox{for}\quad \nu<\nu_2.
\end{equation}

From (\ref{82502}) and (\ref{82903}) we have that if $\pi<\widetilde{\sigma}_{+}(\nu)<\frac{3\pi}{2}$ for $\nu_1<\nu<\nu_2$ where $\nu_1>0$ is a constant yet to be determined, then
\begin{equation}\label{82902}
\left\{
  \begin{array}{ll}
 \widetilde{u}_{+}'(\nu)~=~\underbrace{\widetilde{q}_{+}'(\nu)\cos\tilde{\sigma}_{+}(\nu)}_{>0}-
\underbrace{\widetilde{q}_{+}(\nu)\sin\tilde{\sigma}_{+}(\nu)\tilde{\sigma}_{+}'(\nu)}_{<0}>0 & \hbox{for~ $\nu_1<\nu<\nu_2$,} \\[12pt]
    \widetilde{v}_{+}'(\nu)=\displaystyle-\frac{\widetilde{u}_{+}'(\nu)}{\nu}<0 & \hbox{for~ $\nu_1<\nu<\nu_2$.}
  \end{array}
\right.
\end{equation}

Let
\begin{equation}\label{theta0}
q_{\infty}=\bigg(q_{g}^{2}-2\int_{\tau_{g}}^{+\infty}\tau p_{\tau}(\tau)~{\rm d}\tau\bigg)^{1/2}\quad \mbox{and}\quad \sigma_0=\int_{q_g}^{q_{\infty}}\frac{\sqrt{q^{2}-c^{2}(\check{\tau}(q))}}{qc(\check{\tau}(q))}~ {\rm d}q.
\end{equation}
By (\ref{72803}) we know that
$$
\sigma_0\rightarrow \sigma_{\infty}\quad \mbox{as}\quad \tau_1^e-\tau_0\rightarrow 0.
$$
So, by $\sigma_*<\sigma_{\infty}$ we know that when $\tau_1^e-\tau_0$ is sufficiently small,
there exists a $q_m>q_g$ such that
$$
\pi+\int_{q_{g}}^{q_m}\frac{\sqrt{q^{2}-c^2(\check{\tau}(q) )}}{qc(\check{\tau}(q))} ~{\rm d}q=\pi+\sigma_g.
$$

Let
$$
\tau_m:=\check{\tau}(q_{m})\quad\mbox{and}\quad
 A_{m}:=\arcsin\bigg(\frac{c(\tau_{m})}{q_{m}}\bigg).
$$
Then we have
$\lambda_{\pm}(-q_m, 0, c(\tau_{m}))=\pm\tan A_{m}$.
Now we let
\begin{equation}\label{010903}
\nu_1:=\tan A_{m}.
\end{equation}

Then by (\ref{82903})--(\ref{82902}) we know that the initial value problem  (\ref{principal})--(\ref{principal1}), (\ref{82304}) admits a solution on $[\nu_1, \nu_2]$. Moreover, the solution satisfies $\widetilde{v}_{+}(\nu_1)=0$,
$\widetilde{\rho}_{+}(\nu_1)=\rho_m:=1/\tau_m$, and $\widetilde{v}_{+}(\nu)<0$ for $\nu\in (\nu_1, \nu_2]$.
This completes the proof.
\end{proof}

In order to obtain the principal part of the $C_{-}$ centered wave to the DGP (\ref{42501}, \ref{61602}), we consider (\ref{82306})-(\ref{82307}) with data
\begin{equation}\label{82305}
(\widetilde{u}_{-},\widetilde{v}_{-}, \widetilde{\rho}_{-})(-\nu_2)=(u_g, v_g, \rho_g).
\end{equation}
By the symmetry, the initial value problem  (\ref{82306}, \ref{82307}), (\ref{82305}) admits a solution on $[-\nu_2, -\nu_1]$ and the solution satisfies $$(\widetilde{u}_{-}, \widetilde{v}_{-}, \widetilde{\rho}_{-})(-\nu)=(\widetilde{u}_{+}, -\widetilde{v}_{+}, \widetilde{\rho}_{+})(\nu)\quad \mbox{for}\quad \nu\in [\nu_1, \nu_2].$$

Let $\Gamma_{-}^{\mathrm{G}}$ be the curve $(u, v)=(\widetilde{u}_{+},\widetilde{v}_{+})(\nu)$, $\nu_1\leq \nu\leq\nu_2$ and $\Gamma_{+}^{\mathrm{G}}$  the curve $(u, v)=(\widetilde{u}_{-},\widetilde{v}_{-})(\nu)$, $-\nu_2\leq\nu\leq-\nu_1$ in the hodograph $(u, v)$-plane.



From (\ref{82903}) we have that the function $\rho=\widetilde{\rho}_{+}(\nu)$, $\nu\in[\nu_1,\nu_2]$ has an inverse function $\nu=\widetilde{\rho}_{+}^{-1}(\rho)$, $\rho\in[\rho_m, \rho_g]$.
Let $(\check{\alpha}_{+}, \check{\beta}_{+}, \check{\sigma}_{+}, \check{q}_{+})(\tau)=(\widetilde{\alpha}_{+}, \widetilde{\beta}_{+}, \widetilde{\sigma}_{+}, \widetilde{q}_{+})(\widetilde{\rho}_{+}^{-1}(\frac{1}{\tau}))$. Then by (\ref{8304}) and (\ref{8303}) we have
\begin{equation}\label{61902}
\left\{
  \begin{array}{ll}
\displaystyle \frac{{\rm d}  \check{\alpha}_{+}}{{\rm d}  \tau}=-\frac{\tau^2p''(\tau)}{2c^2(\tau)}\tan \check{A}_{+},  \\[10pt]
  \displaystyle  \frac{{\rm d}  \check{\beta}_{+}}{{\rm d}  \tau }=-\frac{\tau^2p''(\tau)}{4 c^2(\tau)
  }\big(\varpi(\tau)-\tan^2\check{A}_{+}\big)\sin2 \check{A}_{+}
  \end{array}
\right.
\end{equation}
where $\check{A}_{+}=\frac{\check{\alpha}_{+}-\check{\beta}_{+}}{2}$.
From (\ref{82904}) and (\ref{010903}) we know
\begin{equation}\label{61803}
(\check{\alpha}_{+},\check{\beta}_{+})(\tau_{g})=(\alpha_3(\mathrm{G}), \beta_3(\mathrm{G}))
\end{equation}
and
\begin{equation}\label{102501a}
   (\check{\alpha}_{+},\check{\beta}_{+})(\tau_m)=(\pi+A_m, \pi-A_m).
\end{equation}

From (\ref{8306}) we  have that the function $\rho=\widetilde{\rho}_{-}(\nu)$ has an inverse function $\nu=\widetilde{\rho}_{-}^{-1}(\rho)$.
Let  $(\check{\alpha}_{-}(\tau), \check{\beta}_{-}(\tau))=\big(\widetilde{\alpha}_{-}(\widetilde{\rho}_{-}^{-1}(\frac{1}{\tau})), \widetilde{\beta}_{-}(\widetilde{\rho}_{-}^{-1}(\frac{1}{\tau}))\big)$. Then by (\ref{8306})--(\ref{8305}) we have
\begin{equation}\label{120602}
\left\{
  \begin{array}{ll}
\displaystyle \frac{{\rm d}  \check{\beta}_{-}}{{\rm d} \tau}=\frac{\tau^2p''(\tau)}{2c^2(\tau)}\tan \check{A}_{-},  \\[8pt]
  \displaystyle \frac{{\rm d}  \check{\alpha}_{-}}{{\rm d}  \tau}=\frac{\tau^2p''(\tau)}{4c^2(\tau)
}\big(\varpi(\tau)-\tan^2\check{A}_{-}\big)\sin2 \check{A}_{-}
  \end{array}
\right.
\end{equation}
where $\check{A}_{-}=\frac{\check{\alpha}_{-}-\check{\beta}_{-}}{2}$.
Similarly, we have
\begin{equation}\label{120603}
(\check{\alpha}_{-},\check{\beta}_{-})(\tau_{g})=(\alpha_2(\mathrm{G}), \beta_2(\mathrm{G}))
\end{equation}
and
\begin{equation}\label{61401}
   (\check{\alpha}_{-},\check{\beta}_{-})(\tau_m)=(\pi+A_m, \pi-A_m).
\end{equation}

From (\ref{72803}) one can see
\begin{equation}\label{82401}
A_m\rightarrow A_d\quad \mbox{as}\quad \tau_1^e-\tau_0\rightarrow 0
\end{equation}
where the constant $A_d$ is defined in (\ref{101801}).

\subsubsection{\bf Some important estimates for the principal parts}
For the needs of the following discussion about the hyperbolicity of the system (\ref{42501}), we shall derive some important estimates for the principal parts of the $C_{\pm}$ centered waves of the DGP (\ref{42501}, \ref{61602}).

Under assumptions ($\mathbf{A}2$) and ($\mathbf{A}3$), there exists a $\psi\in \big(\max\{\theta+2\sigma_*, \bar{A}(\tau_2^e)\}, \frac{\pi}{2}\big)$ such that
\begin{equation}\label{618033}
\psi+A_d>2\bar{A}(\tau_2^e)\quad \mbox{and} \quad \psi+\theta>2\bar{A}(\tau_2^e).
\end{equation}

For $s\geq 0$, we define
\begin{equation}
\Upsilon(s):=\Big\{(\alpha, \beta)\mid \alpha^{l}(s)<\alpha<\alpha^{r},~\beta^{l}<\beta<\beta^{r}(s),~\alpha-\beta>A_{min}\Big\},
\end{equation}
where
\begin{equation}\label{102503}
A_{min}:=\min\limits_{\tau\in [\tau_g, \tau_m]}\arcsin\bigg(\frac{c(\tau)}{q_m}\bigg)>0,
\end{equation}
\begin{equation}\label{102504}
\alpha^{l}(s):=\pi-\psi+2\bar{A}(\tau_g+s), \quad \beta^{l}:=\pi-\psi,\quad \alpha^{r}:=\pi+\psi, \quad \beta^{r}(s):=\pi+\psi-2\bar{A}(\tau_g+s).
\end{equation}
By assumption ($\mathbf{A}1$) and (\ref{72803}) we know that if $\tau_1^e-\tau_0$ is sufficiently small, then $\psi>\bar{A}(\tau_g)$ and
$\varpi'(\tau)<0$ for $\tau\in (\tau_g, +\infty)$.
Consequently,
\begin{equation}\label{32402aaa}
\bar{A}'(\tau_g+s)<0\quad \mbox{for}\quad s>0; \quad
\Upsilon(s_1)\subset\Upsilon(s_2)\quad \mbox{for~any}\quad 0\leq s_1<s_2.
\end{equation}

\begin{lem}\label{62901}
Suppose that assumptions ($\mathbf{A}1$)--($\mathbf{A}3$) hold.
Then, when $\tau_1^e-\tau_0$ is sufficiently small, there hold
$$
(\check{\alpha}_{-},\check{\beta}_{-})(\tau)\in \Upsilon(\tau-\tau_g)\quad \mbox{and}\quad (\check{\alpha}_{+},\check{\beta}_{+})(\tau)\in \Upsilon(\tau-\tau_g)\quad \mbox{for}\quad \tau\in [\tau_g, \tau_m].
$$
\end{lem}
\begin{proof}
From (\ref{72803}) and (\ref{82401}) we know that when $\tau_1^e-\tau_0$ is sufficiently small,
\begin{equation}\label{61901}
\psi>\theta+2\sigma_g, \quad \psi+A_m>2\bar{A}(\tau_g), \quad \mbox{and}\quad  \psi+\theta>2\bar{A}(\tau_g).
\end{equation}
Then by (\ref{72802}) and (\ref{61803}) we have that when $\tau_1^e-\tau_0$ is sufficiently small,
\begin{equation}\label{61903}
(\check{\alpha}_{+},\check{\beta}_{+})(\tau_g)\in \Upsilon(0).
\end{equation}
So,  there exists a small $\varepsilon>0$ such that
$(\check{\alpha}_{+},\check{\beta}_{+})(\tau)\in \Upsilon(\tau-\tau_g)$ for $\tau\in [\tau_g, \tau_g+\varepsilon)$.
In order to use the argument of continuity to prove the desired estimate,
we prove the following assertion:
 \begin{itemize}
   \item  For any fixed $s\in(0, \tau_m-\tau_g)$, if $(\check{\alpha}_{+},\check{\beta}_{+})(\tau)\in \Upsilon(\tau-\tau_g)$ for $\tau\in [\tau_g, \tau_g+s)$ then
$(\check{\alpha}_{+},\check{\beta}_{+})(\tau_g+s)\in \Upsilon(s)$.
 \end{itemize}
The proof of the assertion proceeds in three steps.

\noindent
{\it Step 1.}
It is obvious that
 $$(\check{\alpha}_{+}-\check{\beta}_{+})(\tau_g+s)=2\arcsin\left( \frac{c(\tau_g+s)}{\check{q}_{+}(\tau_g+s)}\right)\geq 2A_{min}>A_{min}.$$

\noindent
{\it Step 2.}
By the first equation of (\ref{61902}), (\ref{61803}), and (\ref{62501}) we have
$$
\pi+A_m\leq \check{\alpha}_{+}(\tau_g+s)\leq \alpha_3(\mathrm{G}).
$$
From (\ref{72802}), (\ref{72803}), and (\ref{61901}) we know that when $\tau_1^e-\tau_0$ is small,
there holds $\alpha_3(G)<\pi+\psi$.
By (\ref{32402aaa}) and (\ref{61901}), we have $\pi+A_m>\pi-\psi+2\bar{A}(\tau_g+s)$.
Thus, we have
\begin{equation}\label{82905}
\pi-\psi+2\bar{A}(\tau_g+s)< \check{\alpha}_{+}(\tau_g+s)<\pi+\psi.
\end{equation}

\noindent
{\it Step 3.} We now prove $\pi-\psi<\check{\beta}_{+}(\tau_g+s)<\pi+\psi-2\bar{A}(\tau_g+s)$.

Suppose $\check{\beta}_{+}(\tau_g+s)=\pi+\psi-2\bar{A}(\tau_g+s)$. Then by the assumption of the assertion and (\ref{32402aaa}) we have
$
\check{\beta}_{+}'(\tau_g+s)\geq 0
$. While, by the second equation of (\ref{61902}) and the right inequality of (\ref{82905}) we have
 $$
\frac{{\rm d}  \check{\beta}_{+}}{{\rm d}  \tau }=-\frac{\tau^2p''(\tau)}{4 c^2(\tau)
  }\big(\tan^2 \bar{A}-\tan^2\check{A}_{+}\big)\sin2 \check{A}_{+}<0\quad \mbox{at}\quad \tau=\tau_g+s,
$$
where we have used $$\check{A}_{+}(\tau_g+s)=\frac{\check{\alpha}_{+}(\tau_g+s)-\check{\beta}_{+}(\tau_g+s)}{2}
=\frac{\check{\alpha}_{+}(\tau_g+s)-(\pi+\psi-2\bar{A}(\tau_g+s))}{2}<\bar{A}(\tau_g+s).$$
This leads to a contradiction. So, we have
$\check{\beta}_{+}(\tau_g+s)<\pi+\psi-2\bar{A}(\tau_g+s)$.

 Suppose $\check{\beta}_{+}(\tau_g+s)=\pi-\psi$. Then by the assumption of the assertion and (\ref{32402aaa})  we have
$
\check{\beta}_{+}'(\tau_g+s)\leq 0
$. While, by the second equation of (\ref{61902}) and the left inequality of (\ref{82905}) we have
 $$
\frac{{\rm d}  \check{\beta}_{+}}{{\rm d}  \tau }=-\frac{\tau^2p''(\tau)}{4 c^2(\tau)
  }\big(\tan^2 \bar{A}-\tan^2\check{A}_{+}\big)\sin2 \check{A}_{+}>0\quad \mbox{at}\quad \tau=\tau_g+s,
$$
where we have used $$\check{A}_{+}(\tau_g+s)=\frac{\check{\alpha}_{+}(\tau_g+s)-\check{\beta}_{+}(\tau_g+s)}{2}
=\frac{\check{\alpha}_{+}(\tau_g+s)-(\pi-\psi)}{2}>\bar{A}(\tau_g+s).$$  This leads to a contradiction. So, we have $\check{\beta}_{+}(\tau_g+s)>\pi-\psi$.
This completes the prove of the assertion.

Therefore, by (\ref{32402aaa}), (\ref{61903}), and the argument of continuity we obtain the estimate for $(\check{\alpha}_{+},\check{\beta}_{+})(\tau)$. The prove for the other is similar; we omit the details.

This completes the proof.
\end{proof}

\subsubsection{\bf Local solution to the DGP}
We shall use the hodograph transformation method to find a local classical solution to the problem (\ref{42501}, \ref{61602}).

We define
$$
\bar{\tau}(u, v)=\check{\tau}\Big(\sqrt{(u-\xi_{_G})^2+v^2}\Big),\quad \bar{c}(u, v)=c(\bar{\tau}(u, v)).
$$
$$
\bar{\alpha}(u, v)=\pi+\arctan(\lambda_{+}(u-\xi_{_G}, v, \bar{c}(u, v))),\quad \mbox{and}
\quad
\bar{\beta}(u, v)=\pi+\arctan(\lambda_{-}(u-\xi_{_G}, v, \bar{c}(u, v))).
$$
Here, the function $\check{\tau}(q)$ is defined in (\ref{102501}).
Then we have
\begin{equation}\label{83001}
(\widetilde{\alpha}_{\pm},\widetilde{\beta}_{\pm}, \widetilde{\tau}_{\pm})(\nu)=(\bar{\alpha}, \bar{\beta}, \bar{\tau})(\widetilde{u}_{\pm}(\nu),\widetilde{v}_{\pm}(\nu)), \quad \nu_1\leq \pm \nu\leq \nu_2.
\end{equation}

We consider (\ref{52305}) with data
\begin{equation}\label{61701}
 (\alpha, \beta, \tau)=\left\{
  \begin{array}{ll}
(\bar{\alpha}, \bar{\beta}, \bar{\tau})(u, v), & \hbox{$(u, v)\in\Gamma_{-}^{G}\cup\Gamma_{+}^{G}$;} \\[3pt]
   ({\alpha}_{2}^{h}, {\beta}_{2}^{h}, {\tau}_{2}^{h})(u, v), & \hbox{$(u, v)\in\wideparen{\mathrm{G_2H_2}}$;} \\[3pt]
   ({\alpha}_{3}^{h}, {\beta}_{3}^{h}, {\tau}_{3}^{h})(u, v), & \hbox{$(u, v)\in\wideparen{\mathrm{G_3I_3}}$,}
  \end{array}
\right.
\end{equation}
where $\wideparen{\mathrm{G_3I_3}}=\big\{(u, v)\mid (u, -v)\in \wideparen{\mathrm{G_2H_2}}\big\}$.

Problem (\ref{52305}, \ref{61701}) is a characteristic boundary value problem and the corresponding characteristic equations hold on these characteristic boundaries.

\begin{lem}\label{72907}
For small $\epsilon>0$, we let $\mathrm{J}_2=(u_\epsilon, v_\epsilon)$ be a point on $\wideparen{\mathrm{G_2H_2}}$ such that $| u_\epsilon-u_g|+|v_\epsilon-v_g|=\epsilon$ and let $K_3=(u_\epsilon, -v_\epsilon)$. Then when $\epsilon$ is sufficiently small the characteristic boundary value problem  (\ref{52305}, \ref{61701}) admits a classical solution $(\alpha, \beta, \tau)=(\alpha_4^h, \beta_4^h, \tau_4^h)(u, v)$ on a region $\Sigma_4^h$ closed by $\Gamma_{+}^{G}$, $\Gamma_{-}^{\mathrm{G}}$, $\wideparen{\mathrm{G_2J_2}}$, $\wideparen{\mathrm{G_3K_3}}$, $\wideparen{\mathrm{J_2L_h}}$, and $\wideparen{\mathrm{K_3L_h}}$,
where  $\wideparen{\mathrm{J_2L_h}}$ is a forward $\Gamma_{+}$ characteristic curve issued from $\mathrm{J}_2$ and $\wideparen{\mathrm{K_3L_h}}$  is a forward $\Gamma_{-}$ characteristic curve issued from $\mathrm{K}_3$; see Figure \ref{Fig16} (right). Moreover, when $\epsilon$ is sufficiently small  the solution satisfies
\begin{equation}\label{61904}
(\alpha, \beta)\in \Upsilon(\tau-\tau_g).
\end{equation}
\end{lem}
\begin{proof}
The existence of a continuous and piecewise $C^1$ solution in $\Sigma_4^h$ follows routinely from the local existence of classical solutions for Goursat problems (see, e.g., Li and Yu \cite{Li-Yu}). We omit the details.
The $\Gamma_{+}$ characteristic curve $\wideparen{\mathrm{J_2L_h}}$ and the forward $\Gamma_{-}$ characteristic curve issued from $\mathrm{G}_h=(\tilde{u}_{+}(\nu_1), 0)$ insect with each other at a point $\mathrm{M_{h}}$.
The $\Gamma_{-}$ characteristic curve $\wideparen{\mathrm{K_3L_h}}$ and the forward $\Gamma_{+}$ characteristic curve issued from $\mathrm{G}_h$ insect with each other at a point $\mathrm{N_{h}}$. The solution may be weakly discontinuous on the $\Gamma_{\pm}$ characteristic curve $\wideparen{\mathrm{G_hM_h}}$ and $\wideparen{\mathrm{G_hN_h}}$.
The estimate (\ref{61904}) can be obtained directly by Lemma \ref{62901}.
\end{proof}

By the result of Section 4.3 (Lemma \ref{102502}), we know
\begin{equation}\label{61702}
\mathcal{Z}_{-}\big|_{\wideparen{\mathrm{G_2J_2}}}>0\quad\mbox{and}\quad \mathcal{Z}_{+}\big|_{\wideparen{\mathrm{G_3K_3}}}<0.
\end{equation}
Since $\xi\equiv\xi_{_\mathrm{G}}$ on $\Gamma_{+}^{\mathrm{G}}$, we have $\hat{\partial}_{+}\xi=0$ on $\Gamma_{+}^{\mathrm{G}}$. This immediately implies
\begin{equation}\label{61703}
\mathcal{Z}_{+}=0\quad \mbox{on}\quad \Gamma_{+}^{\mathrm{G}}.
\end{equation}
Similarly, one has
\begin{equation}\label{61704}
\mathcal{Z}_{-}=0\quad \mbox{on}\quad \Gamma_{-}^{\mathrm{G}}.
\end{equation}
Integrating (\ref{322}) along $\Gamma_{\pm}$ characteristic curves issued from points on the characteristic boundaries and noticing (\ref{61702})--(\ref{61704}), we know that the solution satisfies
\begin{equation}\label{102506}
\mathcal{Z}_{-}>0\quad \mbox{and}\quad \mathcal{Z}_{+}<0\quad \mbox{on}\quad \Sigma_4^h\setminus (\Gamma_{+}^{\mathrm{G}}\cup \Gamma_{-}^{\mathrm{G}}).
\end{equation}
Therefore, as in the previous discussion we can establish the global one-to-one inversion of the hodograph transformation.

Let
$$\xi_4(u, v):=u-\frac{c_4^h\cos\sigma_4^h}{\sin A_4^h}, \quad \eta_4(u, v):=v-\frac{c_4^h\sin\sigma_4^h}{\sin A_4^h}, $$
and
$$
\Sigma_4:=\big\{(\xi,\eta)~\big|~ (\xi, \eta)=(\xi_4(u, v), \eta_4(u,v)), (u, v)\in \Sigma_4^h\big\},
$$
where $A_4^h=\frac{\alpha_4^h-\beta_4^h}{2}$, $\sigma_4^h=\frac{\alpha_4^h+\beta_4^h}{2}$, and $c_4^h=c(\tau_4^h)$.
Let $\mathrm{J}=(\xi_4(\mathrm{J}_2), \eta_4(\mathrm{J}_2))$,  $\mathrm{K}=(\xi_4(\mathrm{K}_3), \eta_4(\mathrm{K}_3))$, $\mathrm{L}=(\xi_4(\mathrm{L}_h), \eta_4(\mathrm{L}_h))$,
$\mathrm{M}=(\xi_4(\mathrm{M_h}), \eta_4(\mathrm{M_h}))$, $\mathrm{N}=(\xi_4(\mathrm{N_h}), \eta_4(\mathrm{N_h}))$,
 $$
\wideparen{\mathrm{GM}}=\big\{(\xi,\eta)~\big|~ (\xi, \eta)=(\xi_4(u, v), \eta_4(u,v)), (u, v)\in \wideparen{\mathrm{G_hM_h}}\big\},
$$
 $$
\wideparen{\mathrm{GN}}=\big\{(\xi,\eta)~\big|~ (\xi, \eta)=(\xi_4(u, v), \eta_4(u,v)), (u, v)\in \wideparen{\mathrm{G_hN_h}}\big\},
$$
 $$
\wideparen{\mathrm{JL}}=\big\{(\xi,\eta)~\big|~ (\xi, \eta)=(\xi_4(u, v), \eta_4(u,v)), (u, v)\in \wideparen{\mathrm{J_2L_h}}\big\},
$$
and
 $$
\wideparen{\mathrm{KL}}=\big\{(\xi,\eta)~\big|~ (\xi, \eta)=(\xi_4(u, v), \eta_4(u,v)), (u, v)\in \wideparen{\mathrm{K_3L_h}}\big\}.
$$
Then we have the following local existence.
\begin{lem}
 The DGP (\ref{42501}, \ref{61602}) has a continuous and piecewise $C^1$ solution in  $\Sigma_4$ which is a curved quadrilateral domain bounded by characteristic curves $\wideparen{\mathrm{GJ}}$, $\wideparen{\mathrm{GK}}$, $\wideparen{\mathrm{JL}}$, and $\wideparen{\mathrm{KL}}$; see Figure \ref{Fig16} (left). The solution may be weakly discontinuous on $\wideparen{\mathrm{GM}}$ and $\wideparen{\mathrm{GN}}$.
\end{lem}

We denote the local solution of the DGP by $(u, v, \tau)=(u_{4}, v_{4}, \tau_{4})(\xi, \eta)$.
From (\ref{6807}), (\ref{6808}), and  (\ref{102506})
we have that the local solution of the DGP  (\ref{42501}, \ref{61602}) satisfies
\begin{equation}\label{102508}
\bar{\partial}_{+}\rho\mid_{\wideparen{\mathrm{JL}}}<0 \quad \mbox{and}\quad \bar{\partial}_{-}\rho\mid_{\wideparen{\mathrm{KL}}}<0.
\end{equation}

\subsection{Characteristic boundary value problem}
We now consider (\ref{42501}) with the boundary conditions (see Figure \ref{Fig10})
\begin{equation}\label{61905}
(u, v, \tau)=\left\{
               \begin{array}{ll}
               \big(\sin\theta \hat{u}_l, -\cos\theta \hat{u}_l, \hat{\tau}_l\big)(\xi\sin\theta-\eta\cos\theta) & \hbox{on $\wideparen{\mathrm{BE}}$;} \\[2pt]
             \big(\sin\theta \hat{u}_l, \cos\theta \hat{u}_l, \hat{\tau}_l\big)(\xi\sin\theta+\eta\cos\theta) & \hbox{on $\wideparen{\mathrm{DF}}$;}\\[2pt]
                 (u_2, v_2, \tau_2)(\xi, \eta) & \hbox{on $\wideparen{\mathrm{BH}}\cup \wideparen{\mathrm{JH}}$;} \\[2pt]
              (u_3, v_3, \tau_3)(\xi, \eta) & \hbox{on $\wideparen{\mathrm{DI}}\cup \wideparen{\mathrm{KI}}$;} \\[2pt]
                  (u_{4}, v_{4}, \tau_{4})(\xi, \eta) & \hbox{on $\wideparen{\mathrm{JL}}\cup \wideparen{\mathrm{KL}}$.}
               \end{array}
             \right.
\end{equation}

\subsubsection{\bf Local solution}
For convenience we let $\Pi_{-}=\wideparen{\mathrm{BE}}\cup\wideparen{\mathrm{JH}}\cup\wideparen{\mathrm{KL}}
\cup\wideparen{\mathrm{DI}}$,  $\Pi_{+}=\wideparen{\mathrm{BH}}\cup\wideparen{\mathrm{JL}}\cup\wideparen{\mathrm{KI}}
\cup\wideparen{\mathrm{DF}}$, and $\Pi=\Pi_{+}\cup\Pi_{-}$.
From (\ref{62501}), (\ref{62502}), (\ref{102509}), (\ref{102510}), and (\ref{102508}) we have
\begin{equation}\label{83101}
\bar{\partial}_{\pm}\rho<0\quad  \mbox{on}\quad \Pi_{\pm}.
\end{equation}
 So, we have
$$
\min\limits_{(\xi, \eta)\in\Pi}\tau(\xi, \eta)=\min\big\{\tau_2^e, \tau(\mathrm{J})\big\}>\tau_g.
$$
Let $s_0:=\min\{\tau_2^e-\tau_g, \tau(\mathrm{J})-\tau_g\}$. For $s>s_0$, we define $\Pi(s):=\{(\xi, \eta)\mid (\xi, \eta)\in\Pi,~ \tau(\xi, \eta)\leq \tau_g+s\}$.
We have the following local existence.
\begin{lem}
Assume $s-s_0>0$ is sufficiently small. Then
there exists a level curve of $\tau=\tau_g+s$ denoted by ${\it l}(s)$ such that
the boundary value problem (\ref{42501}, \ref{61905}) admits a classical solution in a domain $\Sigma_5(s)$  bounded by $\Pi(s)$ and ${\it l}(s)$, and the solution satisfies
$$
\bar{\partial}_{\pm}\tau>0\quad \mbox{and}\quad \tau\leq \tau_g+s\quad \mbox{in}\quad \Sigma_5(s).
$$
\end{lem}

\begin{proof}
The local existence of a classical solution follows routinely from Li and Yu \cite{Li-Yu}. Using (\ref{83101}) and the characteristic decompositions (\ref{81104}) we know that the solution satisfies
$\bar{\partial}_{\pm}\tau>0$. So, when $s-s_0$ is sufficiently small, the level curve ${\it l}(s)$ exists in the domain of the local solution.
This completes the proof.
\end{proof}

The classical solution of the characteristic boundary value problem (\ref{42501}, \ref{61905}) may be weakly discontinuous on the forward $C_{\pm}$ characteristic curves issued from $\mathrm{H}$ and $\mathrm{I}$, the forward $C_{-}$ characteristic curve issued from $\mathrm{M}$, and the forward $C_{+}$ characteristic curve issued from $\mathrm{N}$; see the red lines in Figure \ref{Fig10}.

\subsubsection{\bf Boundary data estimates}
\begin{prop}\label{61501}
There exists a sufficiently large positive integer $n$ independent of $\tau_1^e-\tau_0$ such that $\mathcal{H}>0$ for $\tau>\tau_g$, where $\mathcal{H}$ is defined in (\ref{82803}).
\end{prop}
\begin{proof}
By (\ref{72803}) we have that when $\tau_1^e-\tau_0$ is small,
$\tau_g>\frac{\tau_2^e+\tau_2^i}{2}$.
This proposition can be proof by a direct computation, we omit the details.
\end{proof}

Let
$$
A_{\infty}:=\lim\limits_{\tau\rightarrow +\infty}\arctan\sqrt{\varpi(\tau)}=\arctan\sqrt{\frac{3-\gamma}{\gamma+1}}.
$$
By (\ref{72801})  we know that when $\tau_1^e-\tau_0$ is sufficiently small
\begin{equation}\label{82504}
\frac{(\pi+\theta)-\beta_2(\mathrm{B})}{2}>\frac{\theta+\sigma_*}{2}.
\end{equation}

From (\ref{8}), (\ref{62001}), and (\ref{62501}) we have
that on $\wideparen{\mathrm{BE}}$,
$$
\bar{\partial}_{-}A=-\frac{p''(\tau)}{4c^2\rho^4}\Big(\tan^2\bar{A}(\tau)-\tan^2A\Big)\sin2 A\bar{\partial}_{-}\rho>0
\quad \mbox{if}\quad A\leq A_\infty.
$$
Combining this with (\ref{82504}) we obtain
\begin{equation}\label{82701}
A>\min\left\{\frac{A_\infty}{2}, \frac{\theta+\sigma_*}{2}\right\}\quad  \mbox{on} \quad\wideparen{\mathrm{BE}}.
\end{equation}

From (\ref{62501}) and (\ref{82701}) we have that  $\frac{\rho^n\bar{\partial}_{-}\rho}{\sin^2 A}$ is bounded on $\wideparen{\mathrm{BE}}$. By the symmetry, $\frac{\rho^n\bar{\partial}_{+}\rho}{\sin^2 A}$ is bounded on $\wideparen{\mathrm{DF}}$.  From  (\ref{6808}), (\ref{61316}), (\ref{61314}) and (\ref{102506}) we know that  $\frac{\rho^n\bar{\partial}_{-}\rho}{\sin^2 A}$ is bounded on $\wideparen{\mathrm{JH}}\cup\wideparen{\mathrm{KL}}
\cup\wideparen{\mathrm{DI}}$. From  (\ref{6807}) we know that  $\frac{\rho^n\bar{\partial}_{+}\rho}{\sin^2 A}$ is bounded on $\wideparen{\mathrm{BH}}\cup\wideparen{\mathrm{JL}}\cup\wideparen{\mathrm{KI}}$.
Therefore, we
can define the constant
\begin{equation}\label{51301}
\mathcal{M}:=\min \left\{\inf \limits_{\Pi_{-}}\frac{\rho^n\bar{\partial}_{-}\rho}{\sin^2 A}, \quad
\inf \limits_{\Pi_{+}}\frac{\rho^n\bar{\partial}_{+}\rho}{\sin^2 A}\right\}.
\end{equation}

For any $s\geq 0$, we define the ``invariant" region
$$
\Theta(s):=\Big\{(\alpha, \beta)~\big|~ \alpha^{l}(s)<\alpha<\alpha^r,~\beta^{l}<\beta<\beta^r(s),~\alpha-\beta>\delta(s)\Big\},
$$
where $\alpha^{l}(s)$, $\alpha^r$, $\beta^{l}$, and $\beta^r(s)$ are defined in (\ref{102504}),
$$
\delta(s)=\min\bigg\{\arctan\Big(\frac{-1}{2\mathcal{M}\mathcal{N}(s)}\Big),~ \frac{A_{min}}{2}, ~\frac{A_\infty}{2}, ~\frac{\theta+\sigma_*}{2} \bigg\},\quad\mbox{and} \quad  \mathcal{N}(s)=\max\limits_{\tau_g\leq\tau\leq\tau_g+s}\frac{\tau^{4+n}p''(\tau)}{c(\tau)}.
$$
As in (\ref{32402aaa}) we have
\begin{equation}\label{32402a}
\Theta(s_1)\subset\Theta(s_2)\subset\Theta(+\infty)\quad \mbox{for~any}\quad 0\leq s_1<s_2,
\end{equation}
where
$$
\Theta(+\infty):=\Big\{(\alpha, \beta)\mid \pi-\psi+2A_{\infty}<\alpha<\pi+\psi,~\pi-\psi<\beta<\pi+\psi-2A_{\infty},~\alpha-\beta>0\Big\}.
$$

\begin{lem}\label{82601}
Suppose that assumptions ($\mathbf{A}1$)--($\mathbf{A}3$) hold.
 Then when $\tau_1^e-\tau_0$ is sufficiently small, there holds
\begin{equation}\label{61906}
(\alpha, \beta)(\xi,\eta)\in \Theta\big(\tau(\xi,\eta)-\tau_g\big)\quad \mbox{for}\quad (\xi,\eta)\in \Pi.
\end{equation}
\end{lem}
\begin{proof}
From Lemma \ref{72907} we have $(\alpha, \beta)\in \Theta(\tau-\tau_g)$ on $\wideparen{\mathrm{JL}}\cup\wideparen{\mathrm{KL}}$.
By (\ref{72801}) and (\ref{72802}), we know that when $\tau_1^e-\tau_0$ is sufficiently small,  $(\alpha, \beta)\in \Theta(\tau-\tau_g)$ on $\wideparen{\mathrm{BH}}\cup\wideparen{\mathrm{JH}}\cup\wideparen{\mathrm{KI}}
\cup\wideparen{\mathrm{DI}}$.
So, it only remains to prove $(\alpha, \beta)\in \Theta(\tau-\tau_g)$ on $\wideparen{\mathrm{BE}}$ and  $\wideparen{\mathrm{DF}}$.

We now prove $(\alpha, \beta)\in \Theta(\tau-\tau_g)$ on $\wideparen{\mathrm{BE}}$. The proof proceeds in three steps.

\noindent
{\it Step 1.} From (\ref{82701}) and the definition of $\delta(s)$ we obtain $A>\delta(\tau-\tau_g)$ on  $\wideparen{\mathrm{BE}}$.

\vskip 2pt
\noindent
{\it Step 2.}  By (\ref{62001}), (\ref{62501}), (\ref{618033}), and the assumption ($\mathbf{A}1$)  we have
\begin{equation}\label{82503}
\pi-\psi+2\bar{A}(\tau)<\alpha<\pi+\psi\quad \mbox{on}\quad \wideparen{\mathrm{BE}}.
\end{equation}

\vskip 2pt
\noindent
{\it Step 3.} By (\ref{72801}) and (\ref{618033}) we know that when $\tau_1^e-\tau_0$ is sufficiently small,
$$
\pi-\psi<\beta_2(\mathrm{B})<\pi+\psi-2\bar{A}(\tau_2^e).
$$
For any fixed $s>0$,
we let $\mathrm{E}_s$ be the point on $\wideparen{\mathrm{BE}}$ such that $\tau(\mathrm{E}_s)=\tau_2^e+s$.
In order to apply the continuity argument to prove $\pi-\psi<\beta<\beta^{r}(\tau-\tau_g)$ on $\wideparen{\mathrm{BE}}$,
we shall prove the following assertion.
\begin{itemize}
  \item Assume $\pi-\psi<\beta<\beta^{r}(\tau-\tau_g)$ for every point on $\wideparen{\mathrm{BE}_s}\setminus{\mathrm{E}_s}$ then $\pi-\psi<\beta(\mathrm{E}_s)<\beta^{r}(\tau_2^e+s-\tau_g)$.
\end{itemize}

Suppose $\beta(\mathrm{E}_s)=\pi-\psi$. Then by the assumption of the assertion we have $\bar{\partial}_{-}\beta\leq 0$ at $\mathrm{E}_s$. While, by (\ref{8}), (\ref{62501}) and (\ref{82503}) we have
$$
\begin{aligned}
c\bar{\partial}_{-}\beta&=\frac{p''(\tau)}{4c\rho^4}\Big[\tan^2\bar{A}(\tau_2^e+s)-\tan^2A(\mathrm{E}_s)\Big]\sin2 A\bar{\partial}_{-}\rho\\&>
\frac{p''(\tau)}{4c\rho^4}\Bigg[\tan^2\bar{A}(\tau_2^e+s)-
\tan^2\bigg(\frac{(\pi-\psi+2\bar{A}(\tau_2^e+s))-(\pi-\psi)}{2}\bigg)\Bigg]\sin2 A\bar{\partial}_{-}\rho=0
\end{aligned}
$$
at $\mathrm{E}_s$.
This leads to a contradiction.

Suppose $\beta(E_s)=\beta^r(\tau_2^e+s-\tau_g)$. Then by the assumption of the assertion we have $\bar{\partial}_{-}\beta\geq 0$ at $\mathrm{E}_s$. While, by (\ref{8}), (\ref{62501}) and (\ref{82503}) we have
$$
\begin{aligned}
c\bar{\partial}_{-}\beta&=\frac{p''(\tau)}{4c\rho^4}\Big[\tan^2\bar{A}(\tau_2^e+s)-\tan^2A(\mathrm{E}_s)\Big]\sin2 A\bar{\partial}_{-}\rho\\&<
\frac{p''(\tau)}{4c\rho^4}\Bigg[\tan^2\bar{A}(\tau_2^e+s)-\tan^2\bigg(
\frac{(\pi+\psi)-(\pi+\psi-2\bar{A}(\tau_2^e+s))}{2}\bigg)\Bigg]\sin2 A\bar{\partial}_{-}\rho=0
\end{aligned}
$$
at $\mathrm{E}_s$.
This leads to a contradiction.
We then complete the proof of the assertion.
Therefore, by an argument of continuity we have $\pi-\psi<\beta<\beta^{r}(\tau-\tau_g)$ on $\wideparen{\mathrm{BE}}$.

We then complete the proof for that $(\alpha, \beta)\in \Theta(\tau-\tau_g)$ holds on $\wideparen{\mathrm{BE}}$.

The proof for the estimate on $\wideparen{\mathrm{DF}}$ is similar; we omit the details.
This completes the proof of the lemma.
\end{proof}

\subsubsection{\bf Uniform a priori $C^1$ norm estimate}
In order to extend the local solution to a whole determinate region, one needs to establish a uniform a priori estimate for the
$C^1$ norm of the solution; see Li \cite{LiT}.

\begin{lem}\label{51708}
Suppose that assumptions ($\mathbf{A}1$)--($\mathbf{A}3$) hold.
Suppose furthermore that the boundary value problem (\ref{42501}, \ref{61905}) admits a classical solution  in $\Sigma_5(s)$ for some $s>s_0$. Then the solution satisfies
\begin{equation}\label{51602}
(\alpha, \beta)\in \Theta(s)\quad\mbox{and}\quad 2\mathcal{M}<\frac{\rho^n\bar{\partial}_{\pm}\rho}{\sin^2 A}<0 \quad \mbox{in}\quad \Sigma_5(s).
\end{equation}
\end{lem}
\begin{proof}
Firstly, by Lemma \ref{82601} and the definition of $\mathcal{M}$ we know that $(\alpha, \beta)\in\Theta(s_0)$
and $2\mathcal{M}<\frac{\rho^n\bar{\partial}_{\pm}\rho}{\sin^2 A}<0$ in $\Sigma_5(s_0)$.
So, there exists a sufficiently small $\varepsilon>0$ such that
$(\alpha, \beta)\in\Theta(s_0+\varepsilon)$
and $2\mathcal{M}<\frac{\rho^n\bar{\partial}_{\pm}\rho}{\sin^2 A}<0$ in $\Sigma_5(s_0+\varepsilon)$.
In view of (\ref{32402a}), in order to use the argument of continuity to prove this lemma
it suffices to prove the following assertion.
\begin{itemize}
  \item For any fixed $z\in(z_0, s]$, if the estimates
$(\alpha, \beta)\in \Theta(z)$ and $2\mathcal{M}<\frac{\rho^n\bar{\partial}_{+}\rho}{\sin^2 A}<0$ hold for every point in $\Sigma_5(z)\setminus{\it l}(z)$ then
they also hold on ${\it l}(z)$.
\end{itemize}

For any point $\mathrm{S}\in {\it l}(z)$ the backward $C_{+}$ and $C_{-}$ characteristic curves issued from $\mathrm{S}$ remain in $\Sigma_5(z)$ until they meet $\Pi_{-}$ and $\Pi_{+}$ at some points $\mathrm{S}_{+}$ and $\mathrm{S}_{-}$, respectively; see Figure \ref{Fig18}. (Remark: if $\mathrm{S}\in\Pi_{-}$ then $\mathrm{S}_{+}=\mathrm{S}$; if $\mathrm{S}\in\Pi_{+}$ then $\mathrm{S}_{-}=\mathrm{S}$.)

\begin{figure}[htbp]
\begin{center}
\includegraphics[scale=0.66]{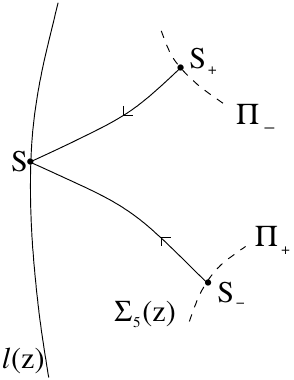}
\caption{\footnotesize Backward $C_{\pm}$ characteristic curves issued from $S$.}
\label{Fig18}
\end{center}
\end{figure}

 Integrating (\ref{new}) along the $C_{\pm}$ characteristic curves from $\mathrm{S}_{\pm}$ to $\mathrm{S}$ and recalling Proposition \ref{61501} and the definition of $\mathcal{M}$, we have
\begin{equation}\label{82602}
2\mathcal{M}<\left(\frac{\rho^n\bar{\partial}_{\pm}\rho}{\sin^2 A}\right)(\mathrm{S})<0.
\end{equation}

 If $\mathrm{S}\in\Pi$, then by Lemma \ref{82601} we have $(\alpha, \beta)(\mathrm{S})\in\Theta(z)$. It only remains to consider the case of $\mathrm{S}\notin\Pi$. When $\mathrm{S}\notin\Pi$, one has $\mathrm{S}_{+}\neq \mathrm{S}$ and $\mathrm{S}_{-}\neq \mathrm{S}$. This implies that the characteristic curves $\wideparen{\mathrm{S}_{+}\mathrm{S}}$ and $\wideparen{\mathrm{S_{-}S}}$ exist. We next prove  $(\alpha, \beta)(\mathrm{S})\notin \partial\Theta(z)$, where $\partial\Theta(z)$ denotes the boundary of $\Theta(z)$.
For convenience, we divide $\partial\Theta(z)$ into the following seven parts (see Figure \ref{Fig19}):
\begin{itemize}
  \item $\gamma_1=\big\{(\alpha, \beta)|\alpha=\alpha^r, \beta^l\leq\beta<\beta^r(z)\big\}$;
  \item $\gamma_2=\big\{(\alpha, \beta)|\beta=\beta^l, \alpha^l(z)<\alpha<\alpha^r\big\}$;
  \item $\gamma_3=\big\{(\alpha, \beta)|\alpha=\alpha^l(z), \beta^l<\beta<\beta^r(z), \alpha-\beta>\delta(z)\big\}$;
  \item $\gamma_4=\big\{(\alpha, \beta)|\beta=\beta^r(z), \alpha^l(z)<\alpha<\alpha^r, \alpha-\beta>\delta(z)\big\}$;
  \item $\gamma_5=\big\{(\alpha, \beta)|\alpha=\alpha^r, \beta=\beta^r(z)\big\}$;
  \item $\gamma_6=\big\{(\alpha, \beta)|\alpha=\alpha^l(z), \beta=\beta^l\big\}$;
  \item $\gamma_7=\big\{(\alpha, \beta)|\alpha^l(z)\leq\alpha\leq \alpha^r, \beta^l\leq\beta\leq\beta^r(z), \alpha-\beta=\delta(z)\big\}$.
\end{itemize}
We shall use the method of contradiction to prove $(\alpha, \beta)(S)\notin \partial\Theta(z)$.

\begin{figure}[htbp]
\begin{center}
\includegraphics[scale=0.52]{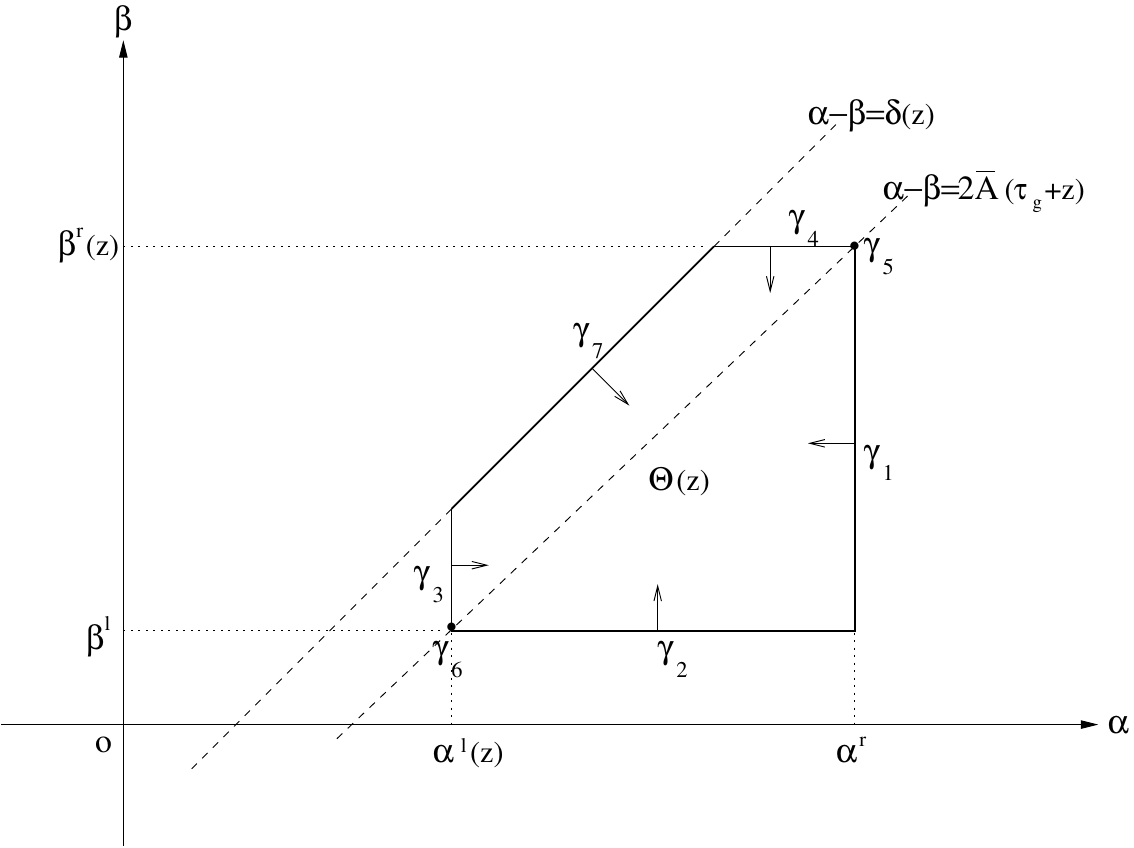}
\caption{\footnotesize Invariant region of $(\alpha, \beta)$.}
\label{Fig19}
\end{center}
\end{figure}

Suppose $(\alpha, \beta)(\mathrm{S})\in\gamma_{1}$. Then
by the assumption that $(\alpha, \beta)\in \Theta(z)$ in $\Sigma_5(z)\setminus{\it l}(z)$ one has
$\bar{\partial}_{+}\alpha(\mathrm{S})\geq0$.
 While, by (\ref{10}) and (\ref{82602}) we have
$$
c\bar{\partial}_{+}\alpha=-\frac{p''(\tau)}{4c\rho^4}\big(\underbrace{\tan^2\bar{A}(\tau_g+z)-\tan^2A}_{<0}\big)\sin2 A\bar{\partial}_{+}\rho<0\quad\mbox{at}\quad \mathrm{S}.
$$
This leads to a contradiction. So, we have $(\alpha, \beta)(S)\notin\gamma_{1}$. Similarly, using (\ref{8}) one can prove $(\alpha, \beta)(S)\notin\gamma_{2}$.

Suppose $(\alpha, \beta)(S)\in\gamma_{3}$.  Then by the assumption that $(\alpha, \beta)\in \Theta(z)$ in $\Sigma_5(z)\setminus{\it l}(z)$ one has
$\bar{\partial}_{+}\alpha(S)\leq 0$.
While, by (\ref{10}) and (\ref{82602}) we have
$$
c\bar{\partial}_{+}\alpha=-\frac{p''(\tau)}{4c\rho^4}\big(\underbrace{\tan^2\bar{A}(\tau_g+z)-\tan^2A}_{>0}\big)\sin2 A\bar{\partial}_{+}\rho>0\quad\mbox{at}\quad\mathrm{ S}.
$$
This leads to a contradiction. So, we have $(\alpha, \beta)(\mathrm{S})\notin\gamma_{3}$. Similarly, one has $(\alpha, \beta)(\mathrm{S})\notin\gamma_{4}$.

Suppose $(\alpha, \beta)(\mathrm{S})\in \gamma_{6}$.
We let $\hat{\alpha}$ be a function defined on $\wideparen{\mathrm{S_{+}S}}$ such that
\begin{equation}\label{112905}
\left\{
  \begin{array}{ll}
   c\bar{\partial}_{+}\hat{\alpha}=\displaystyle-\frac{p''(\tau)}{4c\rho^4}\left[\tan^2 \bar{A}(\tau_g+z)-\tan^2\Big(\frac{\hat{\alpha}
   -\beta^l}{2}\Big)\right]\sin2 A\bar{\partial}_{+}\rho\quad\mbox{along}\quad \wideparen{\mathrm{S_{+}S}}, \\[8pt]
    \hat{\alpha}(S_{+})=\alpha(S_{+}).
  \end{array}
\right.
\end{equation}
Then, by $\alpha(\mathrm{S}_{+})>\alpha^l(z)$ we have
\begin{equation}\label{82604}
\hat{\alpha}(S)>\alpha^l(z).
\end{equation}

From (\ref{10}) and (\ref{112905}) we have
\begin{equation}\label{82603}
\begin{array}{rcl}
  && c\bar{\partial}_{+}(\alpha-\hat{\alpha})=\underbrace{-\frac{p''(\tau)}{4c\rho^4}[\tan^2\bar{A}-\tan^2 \bar{A}(\tau_g+z)]\sin2 A\bar{\partial}_{+}\rho}_{>0, ~\mbox{since}~ \bar{A}>\bar{A}(\tau_g+z)~\mbox{along}~\wideparen{S_{+}S}. }\\[32pt]&&\qquad\qquad\qquad\displaystyle+\frac{p''(\tau)}
{4c\rho^4}\Big[\tan^2\Big(\frac{\alpha-\beta}{2}\Big)-\tan^2\Big(\frac{\hat{\alpha}
-\beta^l}{2}\Big)\Big]\sin2 A\bar{\partial}_{+}\rho
\quad\mbox{along}\quad \wideparen{\mathrm{S_{+}S}}.
  \end{array}
\end{equation}
By the assumption of the assertion we have $(\alpha,\beta)(\xi,\eta)\in\Theta(z)$ for $(\xi,\eta)\in \wideparen{\mathrm{S_{+}S}}\setminus \mathrm{S}$. Thus,
 $$\tan^2\Big(\frac{\alpha-\beta}{2}\Big)-\tan^2\Big(\frac{\hat{\alpha}
-\beta^l}{2}\Big)<0\quad \mbox{if}\quad \alpha=\hat{\alpha}.$$
Therefore, by integrating (\ref{82603}) along $\wideparen{\mathrm{S_{+}S}}$ from $\mathrm{S}_{+}$ to $\mathrm{S}$ and noticing (\ref{82604}) we have $$\alpha(\mathrm{S})>\hat{\alpha}(\mathrm{S})>\alpha^l(z).$$ This leads to a contradiction. So, we have $(\alpha, \beta)(\mathrm{S})\notin \gamma_{6}$.
 Similarly, we can prove that $(\alpha, \beta)(S)\in \gamma_{5}$ is also impossible.

Suppose $(\alpha, \beta)(\mathrm{S})\in\gamma_{7}$.  Then
by the assumption that $(\alpha, \beta)(\xi, \eta)\in \Theta(z)$ for all $(\xi, \eta)\in \Sigma_5(z)\setminus{\it l}(z)$ we have $\bar{\partial}_{-}A\leq 0$ at $\mathrm{S}$.
While, by (\ref{82301}) and (\ref{8}) we have
$$
\begin{aligned}
c\bar{\partial}_{-} A=c\bar{\partial}_{-}\Big(\frac{\alpha-\beta}{2}\Big)&=\sin^2  A \left(1+\frac{p''(\tau)\tan A}{4c\rho^{4+n}}\frac{\rho^n\bar{\partial}_{-}\rho}{\sin^2 A}-\frac{p''(\tau)\Omega\sin^2 A}{8c\rho^{4+n}}\frac{\rho^n\bar{\partial}_{-}\rho}{\sin^2 A}\right)\\&>
\sin^2  A \left(1+\frac{p''(\tau)\tan A}{4c\rho^{4+n}}\frac{\rho^n\bar{\partial}_{-}\rho}{\sin^2 A}+\frac{p''(\tau)\tan A\sin^2 A}{4c\rho^{4+n}}\frac{\rho^n\bar{\partial}_{-}\rho}{\sin^2 A}\right)\\&>
\sin^2  A \Big(1+2\mathcal{M}\mathcal{N}(z)\tan\delta(z)\Big)> 0\quad\mbox{at}\quad \mathrm{S}.
\end{aligned}
$$
This leads to a contradiction.
Then, one gets $(\alpha, \beta)(\mathrm{S})\notin\gamma_{7}$.

This completes the proof of the assertion.

Therefore, by the argument of continuity we know that the solution satisfies (\ref{51602}). This completes the proof of the lemma.
\end{proof}

The estimates in (\ref{51602}) are not enough to ensure a uniform $C^1$ norm estimate of the solution.
In order to establish a uniform $C^1$ norm estimate of the solution, we want to find a uniform bound for  $\frac{\bar{\partial}_{\pm}c}{\sin^2 A}$ and
 a non-zero lower bound for $A$, which are also crucial for the regularity of the vacuum boundary of $\Sigma$.

A direct computation yields that for the polytropic van der Waals gas (\ref{van}),
$$\kappa\rightarrow \frac{2}{\gamma-1}, \quad \mu\rightarrow \frac{\gamma-1}{\gamma+1}, \quad \mbox{and}\quad \frac{c}{\kappa}\cdot\frac{{\rm d}\kappa}{{\rm d}c}\rightarrow 0\quad \mbox{as}\quad  \tau\rightarrow +\infty.$$
So,
there exists a sufficiently large $\tau_*>0$ and a sufficiently small $\varepsilon_*>0$ such that
\begin{equation}\label{51703}
\kappa>\frac{1}{\gamma-1},\quad
0<\mu<\frac{2(\gamma-1)}{\gamma+1}, \quad\mbox{and}\quad\frac{1-\frac{\kappa\sin^2 2A}{\kappa+1}}{2\mu\cos^2 A}-\frac{c}{\kappa}\cdot\frac{{\rm d}\kappa}{{\rm d}c}>0\quad \mbox{for}\quad \tau\geq \tau_*;
\end{equation}
\begin{equation}\label{51704}
\frac{1}{\mu\cos^{2} A}-4\kappa\sin^{2} A-2-\frac{c}{\kappa}\cdot\frac{{\rm d}\kappa}{{\rm d}c}>0\quad \mbox{for}\quad \tau\geq \tau_*\quad \mbox{and}\quad 0<A\leq \varepsilon_*.
\end{equation}
We then define the following constants:
$$
\begin{aligned}
&s_*:=\tau_*-\tau_g,\quad \mathcal{K}:=\frac{2\tau_{*}^{2+n}p'(\tau_*)+\tau_*^{3+n}p''(\tau_*)}{2\sqrt{-p'(\tau_*)}},
\\& \mathcal{C}_1:=\min\left\{\frac{-8(\gamma-1)}{(\gamma+1)\tan A_*}, ~2\mathcal{K}\mathcal{M}\right\}, \quad\mbox{and}\quad  \mathcal{C}_2=\frac{2\mathcal{C}_1}{\sin^2 \varepsilon_*},
\end{aligned}
$$
where $A_*=\min\big\{\frac{A_\infty}{2}, \frac{\theta+\sigma_*}{2}\big\}$.
\vskip 1pt

Let $\delta_*<\min\{\delta(s_*), \varepsilon_*\}$ be a sufficiently small positive constant such that
\begin{equation}\label{51601}
2+\frac{\sin A}{\mu}\Big(\frac{\sin^2 A}{\cos A}+\frac{1}{\cos A}-\varpi(\tau)\cos A\Big)\mathcal{C}_2>0\quad \mbox{for}\quad 0<A\leq \frac{\delta_*}{2}\quad\mbox{and} \quad\tau\geq\tau_*.
\end{equation}


\begin{lem}\label{61510}
Assume that  the boundary value problem (\ref{42501}, \ref{61905})  admits a classical solution in $\Sigma_5(s)$ for some $s>s_*$. Then the solution satisfies
\begin{equation}\label{61509}
\alpha-\beta>\delta_*, \quad \mathcal{C}_1<\bar{\partial}_{\pm}c<0\quad\mbox{and}\quad \mathcal{C}_2<\frac{\bar{\partial}_{\pm}c}{\sin^2 A}<0 \quad \mbox{in}\quad \Sigma_5(s)\setminus\Sigma_5(s_*).
\end{equation}
\end{lem}
\begin{proof}
Firstly, from Lemma \ref{51708} we have
\begin{equation}\label{72906}
\alpha-\beta>\delta_*\quad \mbox{on}\quad {\it l}(s_*)
\end{equation}
By a direct computation we have
$$
\frac{\bar{\partial}_{\pm}c}{\sin ^2A}=\frac{\rho^n\bar{\partial}_{\pm}\rho}{\sin ^2A}\cdot\frac{2\tau^3p'(\tau)+\tau^4p''(\tau)}{2c\rho^n}.
$$
So, by (\ref{51602}) we have
\begin{equation}\label{72902}
2\mathcal{K}\mathcal{M}<\frac{\bar{\partial}_{\pm}c}{\sin ^2A}<0\quad \mbox{on}\quad {\it l}(s_*).
\end{equation}

Let $\mathrm{E}_{*}$ and $\mathrm{F}_{*}$ be the point on $\wideparen{\mathrm{BE}}$ and $\wideparen{\mathrm{DF}}$ such that $\tau(\mathrm{E}_{*})=\tau(\mathrm{F}_{*})=\tau_*$.
Let $\mathrm{E}_{s}$ and $\mathrm{F}_{s}$ be the point on $\wideparen{\mathrm{BE}}$ and $\wideparen{\mathrm{DF}}$ such that $\tau(\mathrm{E}_{s})=\tau(\mathrm{F}_{s})=\tau_g+s$.

From (\ref{82301}) and (\ref{62001}) we have
$$
\frac{\bar{\partial}_{-}c}{2\sin^2 A}=\frac{-2\mu}{\tan A}\quad \mbox{on}\quad \wideparen{\mathrm{E_{*}E}}.
$$
Thus, by (\ref{51703}) and (\ref{82701}) one has
\begin{equation}\label{72903}
-\frac{8(\gamma-1)}{(\gamma+1)\tan A_*}<\frac{\bar{\partial}_{-}c}{\sin^2 A}<0\quad \mbox{on}\quad \wideparen{\mathrm{E_{*}E}},
\end{equation}
By the symmetry we have
\begin{equation}\label{72904}
-\frac{8(\gamma-1)}{(\gamma+1)\tan A_*}<\frac{\bar{\partial}_{+}c}{\sin^2 A}<0\quad \mbox{on}\quad \wideparen{\mathrm{F_{*}F}}.
\end{equation}

Let $\mathrm{Y}$ be an arbitrary point in $\Sigma_5(s)\setminus\Sigma_5(s_*)$. The backward $C_{+}$ and $C_{-}$ characteristic curves intersect $\wideparen{\mathrm{E_*E_s}}\cup{\it l}(s_*)$ and $\wideparen{\mathrm{F}_*\mathrm{F}_s}\cup{\it l}(s_*)$ at some points $\mathrm{Y}_{+}$ and $\mathrm{Y}_{-}$, respectively. We denote by $\Sigma_{_\mathrm{Y}}$ a closed domain bounded by $\wideparen{\mathrm{Y_{-}Y}}$, $\wideparen{\mathrm{Y_{+}Y}}$, and a portion of $\wideparen{\mathrm{E_*E}_s}\cup\wideparen{\mathrm{F_*F}_s}\cup{\it l}(s_*)$.
In order to use the continuity argument to prove (\ref{61509}), we shall prove the following assertion:
assume that the inequalities in (\ref{61509}) hold for every point in $\Sigma_{_\mathrm{Y}}\backslash\{\mathrm{Y}\}$, then they also hold at $\mathrm{Y}$.

\begin{figure}[htbp]
\begin{center}
\includegraphics[scale=0.5]{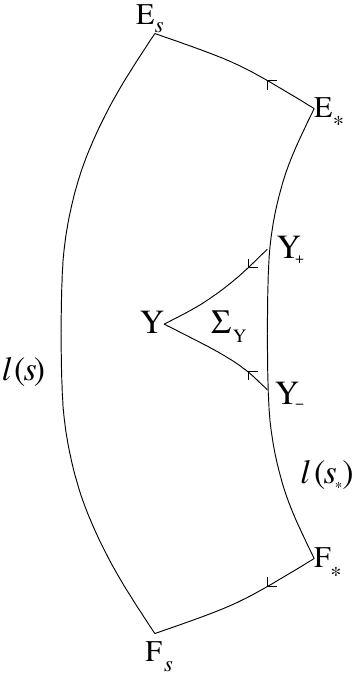}
\caption{\footnotesize Backward $C_{\pm}$ characteristic curves through $Y$ and domain $\Sigma_{_Y}$.}
\label{Fig20}
\end{center}
\end{figure}

Suppose $(\bar{\partial}_{-}c)(\mathrm{Y})=\mathcal{C}_1$. Then by (\ref{72903}) we know $\mathrm{Y}\notin\wideparen{\mathrm{E_{*}E}}$. This implies that the characteristic curve $\wideparen{\mathrm{Y_{+}Y}}$ exists. Hence, by the assumption of the assertion  one has $\bar{\partial}_{+}\bar{\partial}_{-}c\leq 0$ at $\mathrm{Y}$.
While, by the first equation of (\ref{cd}) and (\ref{51703}) one has
$$
\begin{aligned}
c\bar{\partial}_{+}\bar{\partial}_{-}c~=~&\bar{\partial}_{-}c\left(\sin2 A+\frac{ \bar{\partial}_{-}c}{2\mu\cos^2 A}
+
\Big(\frac{1-\frac{\kappa\sin^2 2A}{\kappa+1}}{2\mu\cos^2 A}-\frac{c}{\kappa}\cdot\frac{{\rm d}\kappa}{{\rm d}c}\Big)\bar{\partial}_{+}c\right)\\~>~&\bar{\partial}_{-}c\left(\sin2 A+\frac{ \bar{\partial}_{-}c}{2\mu\cos^2 A}\right)>0\quad \mbox{at}\quad \mathrm{Y}.
\end{aligned}
$$
This leads to a contradiction. So, we get $(\bar{\partial}_{-}c)(\mathrm{Y})>\mathcal{C}_1$. Similarly, one has
$(\bar{\partial}_{+}c)(\mathrm{Y})>\mathcal{C}_1$

Suppose $\big(\frac{\bar{\partial}_{-}c}{\sin^{2} A}\big)(\mathrm{Y})=\mathcal{C}_2$.
Then by $(\bar{\partial}_{-}c)(\mathrm{Y})>\mathcal{C}_1$ one has $A(\mathrm{Y})<\varepsilon_*$.
Moreover, by the assumption of the assertion one has $\bar{\partial}_{+}\left(\frac{\bar{\partial}_{-}c}{\sin^{2} A}\right)\leq 0$ at $\mathrm{Y}$.
While, by the first equation of (\ref{cd1}) and (\ref{51704}) we have
\begin{equation}\label{7205}
\begin{aligned}
 &c\bar{\partial}_{+}\left(\frac{\bar{\partial}_{-}c}{\sin^{2} A}\right)\\=&~\bar{\partial}_{-}c\Bigg(\!\frac{1}{2\mu\cos^{2} A}\bigg(\underbrace{\frac{\bar{\partial}_{-}c}{\sin^{2} A}-\frac{\bar{\partial}_{+}c}{\sin^{2} A}}_{\leq 0}\bigg)
+\bigg(\underbrace{\frac{1}{\mu\cos^{2} A}-4\kappa\sin^{2} A-2-\frac{c}{\kappa}\cdot\frac{{\rm d}\kappa}{{\rm d}c}}_{>0}\bigg)
\frac{\bar{\partial}_{+}c}{\sin^{2} A}\!\Bigg)\\
>&~0\quad \mbox{at}\quad \mathrm{Y}.
\end{aligned}
\end{equation}
This leads to a contradiction. So we get $\big(\frac{\bar{\partial}_{-}c}{\sin^{2} A}\big)(\mathrm{Y})>\mathcal{C}_2$. Similarly, we have $\big(\frac{\bar{\partial}_{+}c}{\sin^{2} A}\big)(\mathrm{Y})>\mathcal{C}_2$.

Suppose $(\alpha-\beta)(\mathrm{Y})=\delta_*$.
Then by (\ref{82701}) we know $\mathrm{Y}\notin \wideparen{\mathrm{E_{*}E}_s}\cup \wideparen{\mathrm{F_{*}F}_s}$.
Thus, the characteristic curves $\wideparen{\mathrm{Y_{+}Y}}$ and $\wideparen{\mathrm{Y_{-}Y}}$ exist.
Hence, by the assumption of the assertion and $(\alpha-\beta)(\mathrm{Y})=\delta_*$ we have $\partial_{-}A\leq 0$ at $\mathrm{Y}$. While,
 by (\ref{82301}), (\ref{8}), and (\ref{51601}) we have
$$
\begin{aligned}
c\bar{\partial}_{-}A&=\frac{\sin^2 A}{2}\left(2+\frac{\sin A}{\mu}\left(\frac{\sin^2 A}{\cos A}+\frac{1}{\cos A}-\varpi(\tau)\cos A\right)\frac{\bar{\partial}_{-}c}{\sin^2 A}\right)\\[2pt]&\geq
\frac{\sin^2 A}{2}\left(2+\frac{\sin A}{\mu}\left(\frac{\sin^2 A}{\cos A}+\frac{1}{\cos A}-\varpi(\tau)\cos A\right)\mathcal{C}_2\right)~>~0\quad\mbox{at}\quad \mathrm{Y}.
\end{aligned}
$$
This leads to a contradiction. So, we get $(\alpha-\beta)(\mathrm{Y})>\delta_*$.



We then complete the proof of the assertion.
Therefore, by an argument of continuity and recalling (\ref{82701}) and (\ref{72906})--(\ref{72902}) we obtain (\ref{61509}). This completes the proof.
\end{proof}

\begin{lem}\label{lem3}
(Uniform a priori $C^1$ estimate)
Suppose that assumptions ($\mathbf{A}1$)--($\mathbf{A}3$) hold.
Suppose furthermore that the boundary value problem (\ref{42501}, \ref{61905})  admits a classical solution on $\Sigma_{5}(s)$ for some $s>s_0$. Then there exist positive constants $\mathcal{T}$ and $\mathcal{R}$  independent of $s$ such that the solution satisfies
$$
|(\nabla u, \nabla v, \nabla \rho)|<\mathcal{T}\quad \mbox{and}\quad |(u,v)|<\mathcal{R}\quad \mbox{in}\quad \Sigma_{5}(s).
$$
\end{lem}
\begin{proof}
By Lemma \ref{51708} we know that there is a constant $\mathcal{T}_1$ such that
$$
| \nabla \rho|<\mathcal{T}_1\quad \mbox{for}\quad \tau\leq\tau_*\quad \mbox{in}\quad \Sigma_{5}(s).
$$

By (\ref{51703}) we know that $\max\limits_{\tau\geq\tau_*}\left|\frac{2c}{2\tau^3p'(\tau)+\tau^4p''(\tau)}\right|$ is finite.
So,
by Lemma \ref{51708}  we have
$$
| \nabla \rho|<\left|\frac{\bar{\partial}_{\pm}\rho}{\sin ^2A}\right|=\left|\frac{\bar{\partial}_{\pm}c}{\sin ^2A}\cdot\frac{2c}{2\tau^3p'(\tau)+\tau^4p''(\tau)}\right|\leq \mathcal{C}_2\cdot\max\limits_{\tau\geq\tau_*}\left|\frac{2c}{2\tau^3p'(\tau)+\tau^4p''(\tau)}\right|\quad \mbox{for}\quad \tau>\tau_*.
$$
Therefore, by (\ref{11}), (\ref{72804}), and (\ref{61513}) we know that there is a constant $\mathcal{T}$ independent of $s$ such that the solution satisfies
$$
\big|(\nabla u, \nabla v, \nabla \rho)\big|<\mathcal{T} \quad\mbox{in}\quad \Sigma_5(s).
$$

Using  (\ref{11}) and (\ref{72804}) and recalling $\bar{\partial}_{\pm}\rho<0$ we know that there exists a $\mathcal{R}>0$ such that
\begin{equation}\label{82702}
|(u,v)|<\mathcal{R} \quad\mbox{in}\quad \Sigma_5(s).
\end{equation}
This completes the proof.
\end{proof}


\subsubsection{\bf Global existence of a classical solution}

\begin{lem}
Under assumptions ($\mathbf{A}1$)--($\mathbf{A}3$),
the boundary value problem (\ref{42501}, \ref{61905}) admits a global classical solution in a bounded region closed by $\Pi$ and a vacuum boundary connecting $\mathrm{D}$ and $\mathrm{E}$.
Moreover, the vacuum boundary can be represented by a Lipschitz continuous function $\xi=\mathcal{V}(\eta)$ defined on $[\hat{u}_l(\hat{\xi}_2)\cos\theta, -\hat{u}_l(\hat{\xi}_2)\cos\theta]$, which satisfies
\begin{equation}\label{198903}
\sup \limits _{\eta_1\neq\eta_2}\frac{|\mathcal{V}(\eta_1)-\mathcal{V}(\eta_2)|}{|\eta_1-\eta_1|}~\leq~\tan(\psi-A_{\infty}).
\end{equation}
\end{lem}
\begin{proof}
Since $\bar{\partial}_{\pm}\tau>0$, each level curve ${\it l}(s)$ has nowhere a characteristic direction.
We take level curves of $\tau$ as the ``Cauchy support" and solved a
mixed initial and characteristic boundary value problem in each extension step; see Lai \cite{Lai1}.
So, by Lemma \ref{lem3} we know that for any large $s>s_0$, the local classical solution of  (\ref{42501}, \ref{61905}) can be extended to $\Sigma_5(s)$. Let $$\Sigma_5:=\bigcup\limits_{s>s_0}\Sigma_5(s).$$ Then the problem (\ref{42501}, \ref{61905}) admits a solution on $\Sigma_5$.

By Lemmas \ref{51708} and \ref{61510} we know that the solution satisfies
\begin{equation}\label{61512}
(\alpha, \beta)\in\Big\{(\alpha, \beta)\mid \pi-\psi+2A_{\infty}<\alpha<\pi+\psi,~\pi-\psi<\beta<\pi+\psi-2A_{\infty},~\alpha-\beta>\delta_*\Big\}.
\end{equation}
This implies that the solution satisfies
 $\frac{\delta_*}{2}<A<\psi$.
Combining this with  $q=\frac{c}{\sin A}$ and (\ref{82702}), we know that the domains $\Sigma_5(s)$ have a uniform bound with respect to $s>s_0$.

Using the pseudo-Bernoulli law (\ref{5802}), we have
\begin{equation}
h_{\xi}=-(u-\xi)u_{\xi}-(v-\eta)v_{\xi}\quad \mbox{and}\quad   h_{\eta}=-(u-\xi)u_{\eta}-(v-\eta)v_{\eta}.
\end{equation}
Thus, along each level curve ${\it l}(s)$,
\begin{equation}\label{62409}
-h_{\eta}u_{\xi}+h_{\xi}u_{\eta}=-(v-\eta)(v_{\xi}u_{\eta}-u_{\xi}v_{\eta})
\end{equation}
and
\begin{equation}\label{62410}
-h_{\eta}v_{\xi}+h_{\xi}v_{\eta}=(u-\xi)(v_{\xi}u_{\eta}-u_{\xi}v_{\eta}).
\end{equation}

Using (\ref{11}), (\ref{72804}), and (\ref{61513}), we have
$$
u_{\xi}=
-\frac{\sin^2\beta\bar{\partial}_{+}\rho+\sin^2\alpha\bar{\partial}_{-}\rho}{\sin2A}\cdot c\tau,\quad
v_{\xi}
=\frac{\sin\beta\cos\beta\bar{\partial}_{+}\rho+\sin\alpha\cos\alpha\bar{\partial}_{-}\rho}{\sin2A}\cdot c\tau,
$$
$$
v_{\eta}=-
\frac{\cos^2\beta\bar{\partial}_{+}\rho+\cos^2\alpha\bar{\partial}_{-}\rho}{\kappa\sin2A}\cdot c\tau,
\quad
u_{\eta}=
\frac{\sin\beta\cos\beta\bar{\partial}_{+}\rho+\sin\alpha\cos\alpha\bar{\partial}_{-}\rho}{\kappa\sin2A}\cdot c\tau.
$$
Thus by Lemma \ref{51708} we have
$$
v_{\xi}u_{\eta}-u_{\xi}v_{\eta}=-c\tau \sin 2A \bar{\partial}_{+}\rho\bar{\partial}_{-}\rho\neq0.
$$
This implies that the images of each level curve ${\it l}(s)$ in the $(u, v)$-plane is a smooth curve.
Moreover, by (\ref{62409}) and (\ref{62410}) we have that along each level curve ${\it l}(s)$,
\begin{equation}\label{2401}
\frac{{\rm d}u}{{\rm d}v}=-\frac{v-\eta}{u-\xi}=-\tan\Big(\frac{\alpha+\beta}{2}\Big).
\end{equation}
We denote by $u=\mathcal{V}_{s}(v)$ the images of  ${\it l}(s)$ in the $(u, v)$-plane. Thus, by (\ref{61512}) and (\ref{2401}) we have
$$
|\mathcal{V}_{s}'(v)|\leq \tan(\psi-A_{\infty}).
$$
Therefore, by  the Arzela-Ascoli theorem  we know that there exists a Lipschite continuous function $u=\mathcal{V}(v)$, $v\in[\hat{u}_l(\hat{\xi}_2)\cos\theta, -\hat{u}_l(\hat{\xi}_2)\cos\theta]$ such that the image of the vacuum boundary in the $(u, v)$-plane can be represented by this function.

From $q=\frac{c}{\sin A}$ and $\frac{\delta_*}{2}<A<\psi$ we know that on the vacuum boundary $q=0$, i.e., $u=\xi$ and $v=\eta$.
So, the vacuum boundary in the physical plane can be represented by $\xi=\mathcal{V}(\eta)$, $[\hat{u}_l(\hat{\xi}_2)\cos\theta, -\hat{u}_l(\hat{\xi}_2)\cos\theta]$.
This completes the proof.
\end{proof}

We denote the solution of the problem (\ref{42501}, \ref{61905}) by $(u, v, \tau)=(u_5, v_5, \tau_5)(\xi, \eta)$.
\subsection{Global solution to the interaction of the fan-shock-fan composite waves}
Let
\begin{equation}\label{82703}
(u, v, \tau)=\left\{
               \begin{array}{ll}
                 (u_1, v_1, \tau_1)(\xi, \eta), & \hbox{$(\xi, \eta)\in \Sigma_1$;} \\[2pt]
                    (u_2, v_2, \tau_2)(\xi, \eta), & \hbox{$(\xi, \eta)\in \Sigma_2$;} \\[2pt]
                   (u_3, v_3, \tau_3)(\xi, \eta), & \hbox{$(\xi, \eta)\in \Sigma_3$;} \\[2pt]
                (u_4, v_4, \tau_4)(\xi, \eta), & \hbox{$(\xi, \eta)\in \Sigma_4$;} \\[2pt]
             (u_5, v_5, \tau_5)(\xi, \eta), & \hbox{$(\xi, \eta)\in \Sigma_5$.}
               \end{array}
             \right.
\end{equation}
Then the $(u, v, \tau)(\xi, \eta)$ defined in (\ref{82703}) is a piecewise smooth solution to the boundary value problem (\ref{42501}, \ref{426011}).
The discontinuity curves include post-sonic shock curves and weak discontinuity curves.
 This completes the proof of Theorem \ref{thm2}.

\vskip 6pt

\section{Conclusion and further remarks}
2D Riemann problems refer to Cauchy problems with special initial data that are
constant along each ray from the origin.
It is well known that the 2D Riemann problem is an important problem in nonlinear hyperbolic conservation laws. Solving  2D Riemann problems would illuminate the physical structure and interaction of waves, and would serve to improve numerical algorithms; see, e.g., Menikoff and Plohr
 \cite{MP}.
In order to solve 2D Riemann problems, one needs to solve various types of  wave interactions (see \cite{ZhangZheng, Zheng1}).
In this paper, we construct a global solution to the interaction of fan-shock-fan composite waves.
 The techniques and ideas developed here may be applied for solving some other 2D Riemann problems of the compressible Euler equations with non-convex equation of state.

The 1D self-similar double-sonic shock for the Euler equations for van der Waals gases was constructed in Cramer and Sen \cite{Cra}. 
The post-sonic shocks, pre-sonic shocks, and double-sonic shocks also frequently appear in 2D steady supersonic flows of Bethe-Zel'dovich-Thompson fluids; see \cite{VKG1,VKG2}.
In the combustion theory, the detonation wave propagation in a static unburnt fluid
after a sudden explosion must be a Chapman–Jouguet detonation which is actually post-sonic; see Landau and Lifschitz \cite{Lan} (Section 130).  The methods developed in this paper may also be applied for studying the stabilities of these sonic-type shocks.

In this paper,
in order to overcome the difficulty cased by the singularity characteristic boundary, we apply the hodograph transformation method to transfer
the 2D self-similar potential flow equations to a linearly degenerate hyperbolic system.
By solving a Cauchy problem for the linearly degenerate hyperbolic system, we obtain the flow near the backside of the post-sonic shocks $\wideparen{\mbox{BG}}$ and $\wideparen{\mbox{DG}}$.
In the real case, the flow behind the post-sonic shocks $\wideparen{\mbox{BG}}$ and $\wideparen{\mbox{DG}}$ is non-isentropic and rotational and is governed by the 2D self-similar full Euler system.
The hodograph transformation method does not work for the full Euler system.
Moreover, for the 2D self-similar potential flow equations, we can use (\ref{10}), (\ref{8}), and (\ref{81104}) to establish
some ``maximum principles" for $\alpha$, $\beta$, and $\bar{\partial}_{\pm}\rho$. These ``maximum principles"
are crucial for the hyperbolicity principle and the a priori $C^1$ norm estimates of the solution.
However, for the 2D self-similar full Euler system, one does not have similar ``maximum principles" since the corresponding characteristic equations would contain nonhomogeneous parts.
So, in order to solve the interaction of the 2D fan-shock-fan composite waves for the full Euler system, new techniques are needed to overcome the difficulties caused by the singular characteristic boundaries, the hyperbolicity principle of the system, and the a priori $C^1$ norm estimate for the solution.

\vskip 4pt
\section{Appendix: van der Waals gas equation of state}
In this appendix we shall show that the assumption ($\mathbf{A}$1) can be satisfied in some cases.
\begin{prop}\label{92001}
Consider the polytropic van der Waals gas (\ref{van}).
When $\mathcal{S}\in (\frac{8}{27}, \frac{81}{256})$ and $\gamma$ is close to $1$, the following conclusions hold:
\begin{itemize}
  \item the isentrope $p=p(\tau)$ has only two inflection points $\tau_1^{i}$ and $\tau_2^i$ and $\tau_1^i<\frac{4}{2-\gamma}<\tau_2^i$;
  \item $p'(\tau)<0$ for $1<\tau<+\infty$;
  \item   $\varpi(\tau)>0$ and $\varpi'(\tau)<0$ for $\tau>\tau_2^i$.
\end{itemize}
\end{prop}
\begin{proof}
By computation, we have
$$
p'(\tau)=\frac{-\gamma \mathcal{S}\tau^3+2(\tau-1)^{\gamma+1}}{\tau^{3}(\tau-1)^{\gamma+1}}\quad
\mbox{and}\quad
p''(\tau)=\frac{\gamma (\gamma+1)\mathcal{S}\tau^4-6(\tau-1)^{\gamma+2}}{\tau^4(\tau-1)^{\gamma+2}}.
$$

We define the following variables:
$$
s_0(\gamma):=\frac{\gamma^{\gamma}(2-\gamma)^{2-\gamma}}{4},\quad
s_1(\gamma):=\frac{2(1+\gamma)^{1+\gamma}(2-\gamma)^{2-\gamma}}{27\gamma},\quad
s_2(\gamma):=\frac{6(2+\gamma)^{2+\gamma}(2-\gamma)^{2-\gamma}}{256\gamma(\gamma+1)}.
$$
From $s_0(1)=\frac{1}{4}$, $s_1(1)=\frac{8}{27}$, and $s_2(1)=\frac{81}{256}$ we know that when $\gamma$ is close to $1$ there holds
$s_0(\gamma)<s_1(\gamma)<s_2(\gamma)$.

By a direct computation we have the following conclusions:
 when $\mathcal{S}>s_0(\gamma)$ one has
$p(\tau)>0$ for $\tau>1$;
when $\mathcal{S}>s_1(\gamma)$ one has
$p'(\tau)<0$ for $\tau>1$;
when $\mathcal{S}>s_2(\gamma)$ one has $p''(\tau)>0$ for $\tau>1$;
 when $s_1(\gamma)<\mathcal{S}<s_2(\gamma)$ the isentrope $p=p(\tau)$ has  only two inflection points $\tau_1^i$ and $\tau_2^i$, where $\tau_1^i<\frac{4}{2-\gamma}<\tau_2^i$.

A direct computation yields
$$
\varpi(\tau)=\frac{\vartheta_1(\gamma, \mathcal{S}, \tau)}
{\gamma (\gamma+1)\mathcal{S}\tau^4-6(\tau-1)^{\gamma+2}}
\quad
\mbox{and}
\quad
\varpi'(\tau)=\frac{\vartheta_2(\gamma, \mathcal{S}, \tau)}
{\big(\gamma (\gamma+1)\mathcal{S}\tau^4-6(\tau-1)^{\gamma+2}\big)^2},
$$
where
$$
\vartheta_1(\gamma, \mathcal{S}, \tau)=\gamma \mathcal{S}\tau^3\big((3-\gamma)\tau-4\big)-2(\tau-1)^{\gamma+2}
$$
and
$$
\begin{array}{rcl}
&\vartheta_2(\gamma, \mathcal{S}, \tau)=
-2\gamma (\gamma+1)(\gamma+2)\mathcal{S}\tau^4(\tau-1)^{\gamma+1}+4\gamma ^2(\gamma+1)\mathcal{S}^2\tau^6
+\gamma (14\gamma-10)\mathcal{S}\tau^3(\tau-1)^{\gamma+2}\\[8pt]&\qquad\qquad+6\gamma (\gamma+2)\mathcal{S}\tau^3(\tau-1)^{\gamma+1}
\big((3-\gamma)\tau-4\big)-18\gamma \mathcal{S}\tau^2(\tau-1)^{\gamma+2}
\big((3-\gamma)\tau-4\big).
\end{array}
$$

The signs of $\vartheta_1(\gamma, \mathcal{S}, \tau)$ and $\vartheta_2(\gamma, \mathcal{S}, \tau)$ for $\tau>\tau_2^i$ are difficult to determine. In what follows, we are going to estimate the signs of $\vartheta_1(\gamma, \mathcal{S}, \tau)$ and $\vartheta_2(\gamma, \mathcal{S}, \tau)$ for some special cases.

When $0\leq \gamma-1<\frac{10}{37}$ and $\mathcal{S}>\frac{1}{4^{2-\gamma}}$, we have
$$
\begin{array}{rcl}
\vartheta_1(\gamma, \mathcal{S}, \tau)&=&
\Big(\frac{\gamma (3-\gamma)\mathcal{S}\tau^3}{(\tau-1)^{\gamma-1}}\big(\tau-\frac{4}{3-\gamma}\big)
-2(\tau-1)^{3}\Big)(\tau-1)^{\gamma-1}\\[8pt]
&>&\Big(\gamma (3-\gamma)s\tau^{4-\gamma}\big(\tau-\frac{4}{3-\gamma}\big)
-2(\tau-1)^{3}\Big)(\tau-1)^{\gamma-1}
\\[8pt]&>&2(\tau-1)^{\gamma-1}\Big((3-\frac{4}{3-\gamma})\tau^2-3\tau+1\Big)>0\quad \mbox{for}\quad \tau>4.
\end{array}
$$
Combining this with  $\tau_2^i>4$ we have  $\varpi(\tau)>0$ for $\tau>\tau_2^i$.

Let
$$
\begin{array}{rcl}
&\hat{\vartheta}_2(\gamma, \mathcal{S}, \tau)=
-2\gamma (\gamma+1)(\gamma+2)\mathcal{S}\tau^4(\tau-1)^{2}+4\gamma ^2(\gamma+1)\mathcal{S}^2\tau^6
+\gamma (14\gamma-10)\mathcal{S}\tau^3(\tau-1)^{3}\\[8pt]&\qquad\qquad+6\gamma (\gamma+2)\mathcal{S}\tau^3(\tau-1)^{2}
\big((3-\gamma)\tau-4\big)-18\gamma \mathcal{S}\tau^2(\tau-1)^{3}
\big((3-\gamma)\tau-4\big).
\end{array}
$$
Then we have
$$
\hat{\vartheta}_2(1, \mathcal{S}, \tau)=-8\mathcal{S}\tau^2\big((1-\mathcal{S})\tau^4-6\tau^3+18\tau^2-22\tau+9\big).
$$
It is easy to check that $\hat{\vartheta}(1, \frac{81}{256}, 4)=0$ and
$$
\hat{\vartheta}_2\Big(1, \frac{81}{256}, \tau\Big)<0\quad \mbox{for}\quad \tau>4.
$$
So, when $\mathcal{S}<\frac{81}{256}$ there holds that
$$
\hat{\vartheta}_2(1, \mathcal{S}, \tau)<0\quad \mbox{for}\quad \tau\geq4.
$$

By continuity we also have that for $\mathcal{S}<\frac{81}{256}$ and $\gamma$ is sufficiently close to $1$,
$$
\hat{\vartheta}_2(\gamma, \mathcal{S}, \tau)<0\quad \mbox{for}\quad \tau\geq4.
$$
Consequently, we have
$$
\vartheta_2(\gamma, \mathcal{S}, \tau)<(\tau-1)^{\gamma-1}\hat{\vartheta}_2(\gamma, \mathcal{S}, \tau)<0\quad \mbox{for}\quad \tau>\tau_2^i.
$$

This completes the proof of the proposition.
\end{proof}

\vskip 16pt
\noindent
{\bf Acknowledgements.} The author was partially supported by NSF of China (No. 12071278) and NSF of Shanghai (No. 23ZR1422100).

\vskip 16pt
\noindent
{\bf  Declarations}
\vskip 4pt
\noindent
{\bf Conflicts of interests} ~This work does not have any conflicts of interest.
\vskip 4pt
\noindent
{\bf  Data availability}
~No data was used for the research described in the article.

\vskip 32pt


\begin{thebibliography}{aa}

\bibitem{BCF1}{\sc M. Bae, G. Q. Chen, and M. Feldman}, {\em Regularity of solutions to regular shock reflection for potential flow}, Invent. Math., 175 (2009) 505--543.

\bibitem{BCF2}{\sc M. Bae, G. Q. Chen, and M. Feldman}, {\em Prandtl-Meyer reflection configurations, transonic shocks, and free boundary problems}, Mem. Amer. Math. Soc., 301 (2024), no. 1507.



\bibitem{Bethe} {\sc H. A. Bethe},
{\em The theory of shock waves for an arbitatry equation of state}, Tech. Paper 45, Office of Scientific Research and Development, 1942.

 \bibitem{BBKN} {\sc A. A. Borisov, Al. A. Borisov, S. S. Kutateladze and V. E. Nakoryakov}, {\em Rarefaction shock wave near the critical liquid vapor point},
J. Fluid Mech., 126 (1983) 59--73.

\bibitem{Cal} {\sc H. B. Callen}, {\em Thermodynamics and an introduction to thermostatistics}, second ed., John Wiley \& Sons, 1988.



\bibitem{Canic} {\sc S. Canic, B. L. Keyfitz and E. H. Kim}, {\em A free boundary
  problem for a quasilinear degenerate elliptic equation: regular reflection
  of weak shock},
  Comm. Pure Appl. Math., 55 (2002) 71--92.

    \bibitem{CCHLQ} {\sc G. Q. Chen, A. Cliffe, F. M. Huang, S. Liu and Q. Wang}, {\em Global Solutions of the two-dimensional Riemann problem with four-shock interactions for the Euler equations for potential flow}, arXiv: 2305.15224v1.

    \bibitem{CX} {\sc G. Q. Chen, X. M. Deng and W. Xiang}, {\em Shock diffraction by convex cornered wedges for the nonlinear wave system}, Arch. Ration. Mech. Anal., 211 (2014) 61--112.

\bibitem{CF1} {\sc G. Q. Chen and M. Feldman}, {\em Global solutions of
shock reflection by large-angle wedges for potential flow}, Ann.
Math., 171 (2010) 1067--1182.

\bibitem{CF2} {\sc G. Q. Chen and M. Feldman}, {\em The Mathematics of Shock Reflection-Diffraction and von Neumann’s Conjectures}, Princeton Math Series in Annals of Mathematical Studies, 197, Princeton University Press, 2017.

    \bibitem{CFX} {\sc G. Q. Chen,  M. Feldman and W. Xiang}, {\em Uniqueness and stability for the shock reflection-diffraction problem for potential flow}, Arch. Ration. Mech. Anal.,
238 (2020) 47--124.

\bibitem{CSX} {\sc S. X. Chen},
{\em Mach configuration in pseudo-stationary compressible flow}, J. Amer. Math. Soc. 21 (2008) 63--100.


\bibitem{Chen1} {\sc X. Chen and Y. X. Zheng},
{\em The interaction of rarefaction waves of the two-dimensional Euler equations}, Indiana Univ. Math. J., 59 (2010) 231--256.

\bibitem{CaF} {\sc R. Courant and K. O. Friedrichs}, {\em Supersonic flow and shock
waves}, Interscience, New York, 1948.


 \bibitem{Cra} {\sc M. S. Cramer and R. Sen}, {\em Exact solutions for sonic shocks in van der Waals gases},
 Phys. Fluids, 30 (1987) 377--385.


\bibitem{Dai} {\sc Z. H. Dai ang T. Zhang}, {\em
Existence of a global smooth solution for a degenerate Goursat
problem of gas dynamics}, Arch. Ration. Mech. Anal., 155 (2000)
277--298.



\bibitem{Elling1}{\sc V. Elling}, {\em Regular reflection in self-similar potential flow and the sonic criterion}, Commun. Math. Anal., 8 (2010) 22--69.


\bibitem{ELiu1} {\sc V. Elling and T. P. Liu}, {\em
The ellipticity principle for self-similar potential flows}, J. Hyper. Diff. Equ., 2 (2005)
  909--917.

 \bibitem{ELiu2} {\sc V. Elling and T.P. Liu}, {\em
  Supersonic flow onto a solid wedge}, Comm. Pure Appl. Math., 61(2008)
  1347--1448.




\bibitem{Hu}{\sc Y. B. Hu, J. Q. Li and W. C. Sheng}, {\em Degenerate Goursat-type boundary value problems arising from the study of two-dimensional isothermal Euler equations}, Z. Angew. Math. Phys., 63 (2012) 1021--1046.


\bibitem{Lai1} {\sc G. Lai}, {\em On the expansion of a wedge of van der Waals gas into a vacuum}, J. Diff. Equ., 259 (2015) 1181--1202.

\bibitem{Lai2} {\sc G. Lai}, {\em On the expansion of a wedge of van der Waals gas into a vacuum II}, J. Diff. Equ., 260 (2016) 3538--3575.


\bibitem{Lai3} {\sc G. Lai}, {\em Global solutions to a class of two-dimensional Riemann problems for the isentropic Euler equations with general equations of state}, Indian. Univ. Math. J. 68 (2019) 1409--1464.


    \bibitem{Lai6} {\sc G. Lai}, {\em Interactions of rarefaction waves and rarefaction shocks of the two-dimensional compressible Euler Equations with general equation of state}, J. Dyn. Diff. Equat., 35 (2023) 381--419.



           \bibitem{Lan}  {\sc L. D. Landau and E. M. Lifschitz},   {\em  Fluid mechanics},  Pergamon Press, Oxford, 1987.

   \bibitem{Lax}  {\sc P. Lax},
{\em Hyperbolic systems of conservation laws}, II. Comm. Pure Appl.
Math., 10 (1957) 537--566.


 \bibitem{LeFloch} {\sc P. G. LeFloch and M. D. Thanh}, {\em The Riemann problem in continuum physics}, Applied Mathematical Sciences 219, Springer, 2023.

\bibitem{Levine}{\sc L. E. Levine}, {\em The expansion of a wedge of gas into a vacuum}, Proc. Camb. Phil. Soc.
64 (1968) 1151--1163.

\bibitem{Li1} {\sc J. Q. Li}, {\em On the two-dimensional gas expansion for the compressible Euler equations},
 SIAM J. Appl. Math., 62 (2002) 831--852.




\bibitem{Li-Yang-Zheng} {\sc J. Q. Li, Z. C. Yang and Y. X. Zheng}, {\em Characteristic decompositions and interaction of rarefaction waves of 2-D Euler equations},
J. Diff. Eqs., 250 (2011) 782--798.


 \bibitem{Li-Zhang-Zheng} {\sc J. Q. Li, T. Zhang and Y. X. Zheng}, {\em Simple waves and a characteristic decomposition of the
 two dimensional compressible Euler equations},
Comm. Math. Phys., 267 (2006) 1--12.

 \bibitem{Li3} {\sc J. Q. Li and Y. X. Zheng},
 {\em Interaction of rarefaction waves of the two-dimensional self-similar Euler equations},
 Arch. Ration. Mech. Anal., 193 (2009) 623--657.

 \bibitem{Li6} {\sc J. Q. Li and Y. X. Zheng}, {\em Interaction of four rarefaction waves in the bi-symmetric class of the two-dimensional Euler equations}, Comm. Math. Phys., 296 (2010) 303--321.



\bibitem{LiT} {\sc T. T. Li},
{\em Global classical solutions for quasilinear hyperbolic system}, John Wiley and Sons, 1994.

\bibitem{Li-Yu} {\sc T. T. Li and W. C. Yu},
{\em Boundary value problem for quasilinear hyperbolic systems},
Duke University, 1985.

\bibitem{Liu1} {\sc T. P. Liu},
{\em  The Riemann problem for $2\times 2$ conservation laws},
Trans. Amer. Math. Soc.,  199 (1974) 89--112.

\bibitem{Liu2} {\sc T. P. Liu},
{\em  The Riemann problem for general systems of conservation laws},
J Diff. Eqs., 18 (1975) 218--234.


\bibitem {MP} {\sc R. Menikoff and B. J. Plohr},  {\em The Riemann problem for fluid flow of real materials}, Rew. Mod. Physics, 61 (1989) 75--130.


\bibitem{NSMGC} {\sc  N. R. Nannan, C.  Sirianni, T. Mathijssen, A. Guardone and P. Colonna}, {\em The admissibility domain of rarefaction shock waves in the near-cirtical vapour-liquid equilibrium region of pure typical fluids}, J. Fluid Mech., 795 (2016) 241--261.


\bibitem{Serre1} {\sc D. Serre}, {\it Shock relfection in gas dynamics.}
Handbook of Mathermathematical Fluid Dynamics, vol. IV. Eds: S. Friedlander, D. Serre., Elsevier, North-Holland, (2007) 39--122.



\bibitem{Suc} {\sc V. A. Suchkow}, {\em Flow into a vacuum along an oblique wall}, J. App. Math. Mech., 27 (1963) 1132--1134.



\bibitem{Thompson2} {\sc P. A. Thompson and K. C. Lambrakis}, {\em Negatie shock waves}, J. Fluid. Mech., 60 (1973) 187--208.


\bibitem{VKG1} {\sc  D. Vimercati, A. Kluwick and A. Guardone}, {\em Oblique waves in steady supersonic flows of Bethe-Zel'dovich-Thompson fluids}, J. Fluid Mech., 885 (2018) 445--468.

    \bibitem{VKG2} {\sc  D. Vimercati, A. Kluwick and A. Guardone}, {\em Shock interactions in two-dimensional steady flows of Bethe-Zel'dovich-Thompson fluids}, J. Fluid Mech., 887 (2020) A12.

\bibitem{Wen1} {\sc  B. Wendroff}, {\em The Riemann problem for materials with nonconvex equation of state, I: Isentropic flow}, J. Math. Anal. Appl., 38 (1972) 454--466.

\bibitem{Wen2} {\sc  B. Wendroff}, {\em The Riemann problem for materials with nonconvex equation of state, II: General flow}, J. Math. Anal. Appl., 38 (1972) 640--658.


\bibitem{ZGC} {\sc  C. Zamfirescu, A. Guardone and P. Colonna}, {\em Admissibility region for rarefaction shock waves in dense gases}, J. Fluid Mech., 599 (2008) 363--381.


    \bibitem{ZE} {\sc Y. B. Zel'dovich},
{\em On the possibility of rarefaction shock waves}, Zh. Eksp. Teor. Fiz., 4 (1946) 363--364.



\bibitem{ZhangZheng} {\sc  T. Zhang and Y. X. Zheng}, {\em Conjecture on the structure of
solution of the Riemann problem for 2D gas dynamics system}, SIAM J. Math. Anal., 21(1990) 593--630.

\bibitem{Zhao} {\sc  W. X. Zhao}, {\em The expansion of gas from a wedge with small
angle into a vacuum}, Comm. Pure Appl. Anal., 12 (2013) 2319--1330.



\bibitem{Zheng1} {\sc Y. X. Zheng}, {\em Systems of Conservation Laws: 2-D Riemann
Problems}, 38 PNLDE, $Bikh\ddot{a}user$, Boston, 2001.

\bibitem{Zheng2}  {\sc Y. X. Zheng}, {\em Two-dimensional regular shock
reflection for the pressure gradient system of conservation Laws},
Acta Math. Appl. Sinica (English Ser.), 22 (2006) 177--210.


\bibitem{Zheng3} {\sc Y. X. Zheng}, {\em Absorption of characteristics by sonic curve of the two-dimensional Euler equations}, Discrete Contin. Dyn. Syst. 23 (2009) 605--616.




\end{thebibliography}
\end{document}